\documentclass[10pt]{article}
\usepackage{amsmath, amssymb, amsthm, graphicx, subfigure, multicol}
\usepackage[dvipsnames]{xcolor}
\usepackage{fullpage}
\usepackage{placeins}
\usepackage[numbers,sort&compress]{natbib}

\usepackage{bbold, amsfonts, bbm}

\title{A unified framework to generate optimized compact finite difference schemes}
\author{
\begin{tabular}{ccc}
Vedang M. Deshpande\footnote{vedang.deshpande@tamu.edu}, & Raktim Bhattacharya\footnote{raktim@tamu.edu},  & Diego A. Donzis\footnote{donzis@tamu.edu}
\end{tabular}\\[4mm]
Aerospace Engineering, Texas A\&M University\\
College Station, TX 77843-3141, USA.}

\date{}
\newcommand{\indi}{i}
\newcommand{\indj}{j}

\newcommand{\npt}{M} % no of grid pts
\newcommand{\nptAug}{\hat{\npt}}
\newcommand{\NAug}{\hat{N}}
\newcommand{\NAc}{\bar{N}}
\newcommand{\kdx}{\eta}
\newcommand{\mkdx}{\tilde{\kdx}}  % modified \kdx
\newcommand{\wtfun}{\gamma(\kdx)}
\newcommand{\err}{e(\kdx)}
\newcommand{\conjErr}{\bar{e}(\kdx)}
\DeclareMathOperator{\sech}{sech}

\newcommand{\dx}{\Delta x}
\newcommand{\dt}{\Delta t}
\newcommand{\vo}[1]{\mathbf{#1}}

\newcommand{\shift}[1]{\vo{\Phi}_{#1}}
\newcommand{\Ltwo}[1]{\left\|#1\right\|_{\mathcal{L}_2}}

\newcommand{\semiDiscSysMat}{\vo{\Lambda}}

\newcommand{\rkA}{\boldsymbol{\mathit{A}}}
\newcommand{\rkB}{\boldsymbol{\mathit{b}}} % \mathbbmss{B}
\newcommand{\rkC}{\boldsymbol{\mathit{c}}}
\newcommand{\rkE}{\boldsymbol{\mathit{e}}}
\newcommand{\rkZ}{\boldsymbol{\mathit{Z}}}
\newcommand{\rka}{\mathit{a}}
\newcommand{\rkb}{\mathit{b}}
\newcommand{\rkc}{\mathit{c}}
\newcommand{\rks}{\mathit{s}}

\newcommand{\real}{\mathbb{R}}
\newcommand{\complex}{\mathbb{C}}
\newcommand{\ith}{i^{\text{th}}}

\newcommand{\ofd}[5]{$\boldsymbol{\mathcal{O}}^{#1,#2}_{#3,#4}(#5)$} % optimized finite difference: Most general case
\newcommand{\ofdC}[5]{$\boldsymbol{\mathcal{O}}^{#1}_{#3}(#5)$} % optimized finite difference: Central schemes
\newcommand{\sfdC}[5]{$\boldsymbol{\mathcal{S}}^{#1}_{#3}(#5)$} % std finite difference: Central schemes

\newcommand{\alert}[1]{\textcolor{red}{#1}}
\newcommand{\blah}[1]{\textcolor{lightgray}{#1}}

\newcommand{\eqnlabel}[1]{\label{eqn:#1}}
\newcommand{\figlabel}[1]{\label{fig:#1}}

\newcommand{\bdel}{\boldsymbol{\delta}}

\newcommand{\eqn}[1]{Eq.(\ref{eqn:#1})}
\newcommand{\fig}[1]{Fig.\ref{fig:#1}}
\newcommand{\Fig}[1]{Fig.\ref{fig:#1}}
\newcommand{\sect}[1]{Section \ref{#1}}

\newtheorem{lem}{Lemma}
\newtheorem{remark}{Remark}

\newtheorem{cor}{Corollary}
\newtheorem{defn}{Definition}

\newcommand{\keywords}[1]{\textit{\textbf{Keywords:}} #1}

\begin{document}
\maketitle

\begin{abstract}
A unified framework to derive optimized compact schemes for a uniform grid is presented. The optimal scheme coefficients are determined analytically by solving an optimization problem to minimize the spectral error subject to equality constraints that ensure specified order of accuracy. A rigorous stability analysis for the optimized schemes is also presented.
We analytically prove the relation between order of a derivative and symmetry or skew-symmetry of the optimal coefficients approximating it. We also show that other types of schemes e.g., spatially explicit, and biased finite differences, can be generated as special cases of the framework.

\end{abstract}

\section{Introduction} \label{intro}
Finite differences have been extensively studied and used to solve ordinary differential equations (ODEs) and partial differential equations (PDEs) numerically. Finite differences, in a broad sense, are classified into two categories, namely, explicit and implicit (compact) finite differences.
In the explicit formulation, approximation of a function derivative at a grid point depends only on the function values at that grid point and grid points adjacent to it. On the other hand, the implicit formulation approximates a function derivative at a grid point by using not only the function values but also the derivatives of the function at adjacent grid points.
Each one of these two formulations has benefits associated with it. For example, explicit schemes are computationally more efficient than compact schemes. However, for a given stencil size, compact schemes can achieve better accuracy and stability than explicit schemes. Thus, depending upon the problem at hand, one can choose to use explicit or compact formulation.

Kumari, et al. (2019) proposed a unified approach to derive optimized explicit schemes in \cite{kumari2018unified}. In this paper, we present a unified framework to derive optimized compact schemes as a natural extension and generalization of the approach presented in \cite{kumari2018unified}.
We also show that optimized schemes derived using the explicit formulation can be recovered as special cases of the compact formulation presented in this paper.

One of the widely used formulations for deriving compact schemes is the 7 grid points stencil which leads to pentadiagonal finite differences \cite{lele1992compact, kim1996opt, kim2007bcopt, zang2003prefactor, zang2011prefactor, hamr2008opt}. Such formulations approximate the first derivative at $i^\text{th}$ grid point as
\begin{align}
\beta f_{i-2} ^\prime + \alpha f_{i-1} ^\prime + f_{i} ^\prime + \alpha f_{i+1} ^\prime + \beta f_{i+2} ^\prime
= c\frac{f_{i+3}-f_{i-3}}{6\dx} + b\frac{f_{i+2}-f_{i-2}}{4\dx} + a\frac{f_{i+1}-f_{i-1}}{2\dx},
\eqnlabel{penta7stencil}
\end{align}
where  $f_i := f(x_i)$ and $f^\prime_i:=\frac{\partial f}{\partial x}\Big|_{x=x_i}$. Note that left hand side (LHS) and right hand side (RHS) stencil sizes for \eqn{penta7stencil} are 5 and 7 respectively. After substituting Taylor series expansion of the terms in the above equation, relations between the unknown parameters $a,b,c,\alpha,$ and $\beta$ are obtained by matching the coefficients.
If the first unmatched coefficient is associated with $\dx^{p+1}$, then the formal order of accuracy for the discretization is said to be $p+1$. When $f$ is periodic over the $x$-domain, derivative at each grid point can be calculated by solving a cyclic linear system of equations with formal truncation error as a global measure of error \cite{lele1992compact}.

In real physical problems which involve a spectrum of scales, one is also keen to quantify and minimize errors incurred at different wavenumbers in the frequency domain. The method of matching coefficients does not guarantee the desired spectral behavior of the compact scheme derived from \eqn{penta7stencil}.
To ensure that all the relevant scales in a problem are properly resolved, it is crucial to have desired spectral accuracy at the wavenumbers of interest.
Over the years, various frameworks were proposed to construct optimized compact schemes with desired spectral accuracy \cite{kim1996opt, kim2007bcopt, zang2003prefactor, zang2011prefactor, hamr2008opt}.
They all formulate compact schemes as generalizations of Pad\'{e} approximation similar to \eqn{penta7stencil}, and calculate optimal coefficients by minimizing a cost function,
most commonly, weighted $\mathcal{L}_2$ norm of suitably defined spectral error over the wavenumber space. Different techniques have been used to solve the minimization problem.
For example, \cite{kim1996opt, kim2007bcopt, zang2003prefactor, zang2011prefactor} used a weighting function constructed from unknown variables involved in the optimization problem to make the cost function integrable.
Then optimal solution is calculated analytically by finding the local minimum of the cost function subject to order of accuracy constraints.
On the other hand, \cite{hamr2008opt} minimizes the cost function weighted equally across all wavenumbers and optimal solution is found using sequential quadratic programming (SQP). An extended optimized compact schemes framework for the implementation of boundary conditions was presented in \cite{kim2007bcopt}.
To circumvent the computationally expensive inversion of a band matrix in case of compact schemes, \cite{zang2003prefactor, zang2011prefactor, rona2017prefactor} proposed prefactored optimized compact finite differences which use forward and backward biased schemes that can be solved explicitly.
Frameworks devised specifically to derive optimized explicit schemes were presented in \cite{webb1993drp,zhuang2002upwind,bogey2004bailly,zhang2013maxnorm}. Unlike others, \cite{zhang2013maxnorm} defined the cost function as maximum (infinity) norm of the spectral error
and used simulated annealing technique to find the optimal solution.
Each afore discussed framework is restrictive in at least one of the following ways:
\begin{enumerate}
  \item Frameworks for optimal compact schemes are limited to derive pentadiagonal (or tridiagonal) schemes. They can not be easily scaled up for larger stencil sizes.
	\item Symmetry (skew-symmetry) of coefficients for periodic domains is imposed a priori for even (odd) derivatives. Whereas, in this paper we show that symmetry or skew-symmetry of optimal coefficients is a consequence of optimization based on an integrated formulation.
  \item Most of the frameworks rely on a carefully constructed weighting function to make the cost function integrable, which facilitates the analytical computation of optimal coefficients. Therefore, the weighting function has to have a special form, which restricts the choice of weighting function.	Moreover, the weighting function is dependent upon the variables of the optimization problem itself. Therefore, nature of the weighting function is not completely known.
	\item Frameworks are requirement specific and there is no unifying generalized framework. For example,  \cite{kim1996opt, kim2007bcopt, hamr2008opt, zang2003prefactor, zang2011prefactor} focus on pentadiagonal compact schemes, \cite{webb1993drp,zhuang2002upwind,bogey2004bailly,zhang2013maxnorm,kumari2018unified}
	are exclusively for explicit finite difference, and \cite{kim1996opt,webb1993drp,zhuang2002upwind,bogey2004bailly} approximate only the first derivative.
\end{enumerate}

In this paper we present a unifying framework to derive optimized compact schemes, which overcomes the limitations of existing frameworks. The key features of the framework are as follows:
\begin{enumerate}
\item Optimal coefficients of compact schemes to approximate a derivative of any order are determined analytically by solving a minimization problem with equality constraints.
\item First, we develop a framework to derive optimal compact schemes of equal LHS and RHS stencil sizes. Numerical simulations show that schemes with equal LHS and RHS stencil sizes perform better than those with unequal stencil sizes.
\item This generalized framework also allows us to derive compact schemes with unequal LHS and RHS stencil sizes, explicit finite differences, and biased schemes for non-periodic domains by imposing additional equality constraints in the optimization problem, which are discussed as special cases of the framework.
\item The weighting function can be chosen independent of the cost function.
\item A rigorous stability analysis of the schemes is performed to maximize the time step $\dt$ while guaranteeing the stability of discretization.
\end{enumerate}

The organization of the paper is as follows. In \sect{implicitFramework} we present a generalized framework to derive central optimal compact schemes and prove few important theorems. We perform stability analysis of the schemes for semi-discrete and fully discrete case in \sect{stability}.
In \sect{specialCases} we discuss a few special cases of the framework. Numerical results obtained using optimal schemes derived in this paper are presented in \sect{numericalResults}. Conclusions of this study are discussed in \sect{conclusions}.
Proof of a theorem is presented in Appendix \ref{sec:appInvarianceGrid}, and values of the coefficients used for numerical simulations are tabulated in Appendix \ref{sec:appOptimCoeff} and \ref{sec:appButcherTab}.

\section{Formulation of the unified framework} \label{implicitFramework}

A generalization of \eqn{penta7stencil} which is a Pad\'{e} approximation of $d^\text{th}$ spatial derivative, $\frac{\partial^d f(x)}{\partial x^d}$, with equal LHS and RHS stencil sizes is given as:
\begin{align}
\sum_{m = -\npt} ^{\npt} \mathrm{b}_m f^{(d)}_{i+m}	= \frac{1}{(\dx)^d}\sum_{m = -\npt} ^{\npt} \mathrm{a}_m f_{i+m},
\eqnlabel{pade1}
\end{align}
where  $f_i := f(x_i)$, $f^{(d)}_i:=\frac{\partial^d f}{\partial x^d}\Big|_{x=x_i}$, $\npt>0$ and $N:=2\npt+1$ is a positive integer which denotes the stencil size.
Note that index $m=0$ corresponds to the $\ith$ grid point at which derivative is to be approximated.

For convenience and clarity of discussion, we denote $(p+1)^{\text{th}}$ order accurate optimized compact finite difference schemes obtained using \eqn{pade1} as
\ofdC{\npt}{\npt}{\npt}{\npt}{p+1}. The superscript $\npt$ denotes the range of $m$ ($-\npt$ to $\npt$) on the RHS of \eqn{pade1}, and the subscript $\npt$ is for the range of $m$ on the LHS of \eqn{pade1}. For example, we denote $4^\text{th}$ order accurate optimized scheme obtained using $\npt=3$ by \ofdC{3}{3}{3}{3}{4}.
Schemes derived by standard method of matching coefficients without any optimization, such as presented in \cite{lele1992compact}, are denoted by \sfdC{\npt}{\npt}{\npt}{\npt}{\cdot}.

Using Taylor series expansion, we get
\begin{multline*}
	f_{i+m} = f_i + f_i'\, m\Delta x +\cdots +  f_i^{(d-1)}\frac{(m\Delta x)^{(d-1)}}{(d-1)!} + f_i^{(d)}\frac{(m\Delta x)^d}{d!} + f_i^{(d+1)}\frac{(m\Delta x)^{(d+1)}}{(d+1)!}  + \cdots \\ + f_i^{(d+r)}\frac{(m\Delta x)^{(d+r)}}{(d+r)!} + \cdots ,
\end{multline*}
and
\begin{align*}
	f^{(d)}_{i+m} = f^{(d)}_i + f_i^{(d+1)}(m\Delta x) + \cdots +  f_i^{(d+r)}\frac{(m\Delta x)^{r}}{r!} + \cdots .
\end{align*}
Substituting them in \eqn{pade1}, and matching coefficients we get the following linear constraints in $\mathrm{a}_m$ and $\mathrm{b}_m$,
\begin{align*}
\sum_m \mathrm{a}_m m^j  & = 0, \text{ for } j = 0, \cdots, d-1;\\
\sum_m \left(\frac{m^{d+r}}{(d+r)!}\mathrm{a}_m - \frac{m^r}{r!}\mathrm{b}_m \right)& = 0, \text{ for } r = 0,\cdots,p.
\end{align*}

We use the definition $0^0=1$ for the equations above. For a more compact representation of the formulation, let $m\in\{-\npt,\cdots,\npt\}$ for periodic domains. Let us define vectors $\vo{a}_d\in\real^N$ and $\vo{b}_d\in\real^N$ as \footnote{$\real$ denotes the set of all real numbers and $\complex$ denotes the set of all complex numbers.}
\begin{align*}
\vo{a}_d &:= \begin{bmatrix}\mathrm{a}_{-\npt} & \mathrm{a}_{-\npt+1} & \cdots & \mathrm{a}_{\npt-1} & \mathrm{a}_{\npt}\end{bmatrix}^T, \\
\vo{b}_d &:= \begin{bmatrix}\mathrm{b}_{-\npt} & \mathrm{b}_{-\npt+1} & \cdots & \mathrm{b}_{\npt-1} & \mathrm{b}_{\npt}\end{bmatrix}^T.
\end{align*}
Let us also define vector
\begin{align*}
\vo{m} := \begin{bmatrix} -\npt & -\npt+1 & \cdots & \npt-1 & \npt \end{bmatrix}^T.
\end{align*}
We can then write the above linear constraints in terms of $\vo{a}_d$ and $\vo{b}_d$ as
\begin{align}
\vo{a}_d^T\vo{X}_d - \vo{b}_d^T\vo{Y}_d = \vo{0}_{1\times (d+p+1)}
\eqnlabel{linPade}
\end{align}
where matrices $\vo{X}_d \in \real^{N \times (d+p+1)}$, $\vo{Y}_d \in \real^{N \times (d+p+1)}$ are defined as
\begin{align}
\vo{X}_d &:= \begin{bmatrix}
\vo{1}_{N\times 1} & \vo{m} & \cdots &
\vo{m}^{d-1} & \frac{\vo{m}^d}{d!} & \cdots &\frac{\vo{m}^{d+p}}{(d+p)!}
\end{bmatrix},\eqnlabel{structureX} \\
\vo{Y}_d &:= \begin{bmatrix}
\vo{0}_{N\times d} & \vo{1}_{N\times 1} & \vo{m}  & \cdots & \frac{\vo{m}^p}{p!}
\end{bmatrix}.
\eqnlabel{structureY}
\end{align}
The exponents of $\vo{m}$ are element wise. \eqn{linPade} represents $(d+p+1)$ constraints on $2(2\npt+1)$ variables. When the number of constraints is less than the number of variables, the system of linear equations can not be solved uniquely.
In other words, the system has extra degrees of freedom which can be used to calculate a set of variables which minimizes a well defined cost function. We make use of this freedom to minimize the spectral error.

To study the spectral behavior of finite difference schemes at different wavenumbers (or frequencies), it is customary to work in frequency domain.
Therefore, we consider a discrete Fourier mode as follows
\begin{align*}
f(x) &= \hat{f}e^{jkx}
\end{align*}
where, $j=\sqrt{-1}$, and $\hat{f}$ is the Fourier coefficient of the mode corresponding to wavenumber $k$. Then the analytical $d^\text{th}$ derivative is
\begin{align}
 f^{(d)}(x) &= (jk)^d f(x)
 \eqnlabel{dthDerAnalytical}
\end{align}
The finite difference approximation for the $d^\text{th}$ derivative from \eqn{pade1} is then
\begin{align*}
\bigg(\sum_m \mathrm{b}_m e^{jkm\Delta x} \bigg) f^{(d)}_i = \frac{1}{(\dx)^d} \bigg(\sum_m \mathrm{a}_m e^{jkm\Delta x}\bigg)f_i.
\end{align*}
Let us define normalized wavenumber as $\kdx:=k\dx$, $\kdx \in [0,\pi]$. Then the previous equation can be written in terms of $\kdx$ as
\begin{align*}
\bigg(\vo{C}^T(\kdx) + j\vo{S}^T(\kdx)\bigg)\vo{b}_df^{(d)}_i =  \frac{1}{(\dx)^d}\bigg(\vo{C}^T(\kdx) + j\vo{S}^T(\kdx)\bigg)\vo{a}_df_i,
\end{align*}
or,
\begin{align}
f^{(d)}_i =  \frac{1}{(\dx)^d} \frac{\bigg(\vo{C}^T(\kdx) + j\vo{S}^T(\kdx)\bigg)\vo{a}_d}{\bigg(\vo{C}^T(\kdx) + j\vo{S}^T(\kdx)\bigg)\vo{b}_d}  f_i,
\eqnlabel{dthDerApprox}
\end{align}
where
\begin{align}
\vo{C}(\kdx) := \begin{bmatrix} \cos(-\npt \kdx)\\\vdots \\ \cos(-\kdx) \\ 1 \\ \cos{\kdx} \\ \vdots \\\cos(\npt \kdx) \end{bmatrix}, \quad \text{and} \quad \vo{S}(\kdx) := \begin{bmatrix} \sin(-\npt \kdx)\\\vdots \\ \sin(-\kdx) \\ 0 \\ \sin(\kdx) \\ \vdots \\\sin(\npt \kdx) \end{bmatrix}.
\eqnlabel{SinCos:def}
\end{align}
 Comparison of \eqn{dthDerApprox} with \eqn{dthDerAnalytical} leads us to define the spectral error $e(\kdx)$ at wavenumber $\kdx = k\dx$ as
\begin{align}
	&e(\kdx) := \frac{\bigg(\vo{C}^T(\kdx) + j\vo{S}^T(\kdx)\bigg)\vo{a}_d}{\bigg(\vo{C}^T(\kdx) + j\vo{S}^T(\kdx)\bigg)\vo{b}_d} - (j\kdx)^d,\\
	\text{or } & e(\kdx) = \frac{\bigg(\vo{C}^T(\kdx) + j\vo{S}^T(\kdx)\bigg)\vo{a}_d - (j\kdx)^d\bigg(\vo{C}^T(\kdx) + j\vo{S}^T(\kdx)\bigg)\vo{b}_d}{\bigg(\vo{C}^T(\kdx) + j\vo{S}^T(\kdx)\bigg)\vo{b}_d}.
	 \eqnlabel{padeWave}
\end{align}
We can also express the spectral error in terms of \textit{modified wavenumber} $(j\mkdx)^d := \frac{\big(\vo{C}^T(\kdx) + j\vo{S}^T(\kdx)\big)\vo{a}_d}{\big(\vo{C}^T(\kdx) + j\vo{S}^T(\kdx)\big)\vo{b}_d}$ as
\begin{align}
	e(\kdx) = (j\mkdx)^d - (j\kdx)^d = j^d(\mkdx^d - \kdx^d).
	\eqnlabel{defmkdx}
\end{align}
We compare $\mkdx^d$ with $\kdx^d$ to show the spectral accuracy of optimal schemes in sections to follow.
The weighted ${\mathcal{L}_2}$ norm of a function $e(\kdx)$, for $\kdx \in [0,\pi]$, is defined as
\begin{align}
 \|\err\|^2_{\mathcal{L}_2}:= \int _0 ^{\pi} \wtfun \conjErr \err d\kdx =: \left\langle \conjErr \err \right\rangle.
 \eqnlabel{normWithGamma}
\end{align}
Where, $\conjErr$ is the complex conjugate of $\err$, and $\wtfun \geq 0$ is a weighting function. The objective here is to determine $\vo{a}_d$ and $\vo{b}_d$, which minimize $\Ltwo{e(\kdx)}$, subject to order accuracy constraint, i.e.

\begin{align*}
\min_{\vo{a}_d,\vo{b}_d\in\real^{N}} \|e(\kdx)\|^2_{\mathcal{L}_2}, \text{ subject to \eqn{linPade}.}
\end{align*}
Before proceeding further, we present a definition and a lemma, which is used in the subsequent discussion.

\begin{defn}
	\label{defn:symmVec}
A given vector $\vo{x} \in \real ^ n$, is symmetric if $\vo{J}\vo{x} = \vo{x}$, or skew-symmetric if $\vo{J}\vo{x} = - \vo{x}$, where, $\vo{J}$ is an anti-diagonal identity matrix of dimension $n$.
\end{defn}
Note that, $\vo{J}$ is an operator which reverses the order of elements in the vector. For example, let,
\begin{align*}
\vo{J} = \begin{bmatrix}
0 & 0 & 1 \\
0 & 1 & 0 \\
1 & 0 & 0
\end{bmatrix}; \quad \vo{x} = \begin{bmatrix} 1 \\2\\3\end{bmatrix}, \quad \text{then,} \quad \vo{J} \vo{x}  = \begin{bmatrix} 3\\2\\1\end{bmatrix}; \quad \text{and} \quad \vo{x} ^T\vo{J} = \begin{bmatrix} 3&2&1\end{bmatrix}.
\end{align*}
If $n$ is odd, then the vector is said to be symmetric or skew-symmetric \textit{about the central element}. It can be easily verified that, $\vo{C}(\kdx)$ and $\vo{S}(\kdx)$ defined in \eqn{SinCos:def} are respectively symmetric and skew-symmetric about the central element, i.e., they satisfy
\begin{align}
\vo{C}^T(\kdx)\vo{J} = \vo{C}^T(\kdx) \quad \text{and} \quad \vo{S}^T(\kdx)\vo{J} = -\vo{S}^T(\kdx).
\eqnlabel{evenCJSJ}
\end{align}
Also note that, for any $\vo{x} \in \real ^ n$, we can define a symmetric vector $\vo{x}_s:=(\vo{x}+\vo{J}\vo{x})/2$ and a skew-symmetric vector  $\vo{x}_w:=(\vo{x}-\vo{J}\vo{x})/2$ such that $\vo{x}=\vo{x}_s+\vo{x}_w$.

\begin{lem}
	\label{lem:normRelaxLemma}
	Let $f(x):[q_1,q_2] \rightarrow \complex$ and $g(x):[q_1,q_2] \rightarrow \complex \backslash \{0\}$ where $q_1,q_2 \in \real$ and $q_1\leq q_2$. If $\bar{g}(x)g(x) \geq \epsilon^2$ for some $\epsilon \in \real \backslash \{0\}$, then the following is true
	\begin{align*}
	\Ltwo{ \frac{f}{g}} \leq \frac{\Ltwo{f}}{\epsilon^2}.
	\end{align*}
\end{lem}
\begin{proof} By assumption, $\bar{g}(x)g(x) \geq \epsilon^2$. It follows immediately that,
	\begin{align*}
	 \frac{1}{	\bar{g}(x)g(x)} \leq \frac{1}{\epsilon^2} \implies \frac{\bar{f}(x)f(x)}{\bar{g}(x)g(x)} \leq \frac{\bar{f}(x)f(x)}{\epsilon^2}
	\end{align*}
	\begin{align*}
	 \implies \int_{q_1}^{q_2} \gamma(x) \frac{\bar{f}(x)f(x)}{\bar{g}(x)g(x)} dx \leq \int_{q_1}^{q_2} \gamma(x) \frac{\bar{f}(x)f(x)}{\epsilon^2} dx,
	\end{align*}
for weight function $\gamma(x)\geq 0$. By using the definition of $\mathcal{L}_2$ norm, we get the desired result.
\end{proof}

Now, we consider the minimization problem for even and odd derivatives separately, and solve it analytically.

\subsection{Even derivatives}
For even $d$, i.e. $d:=2q$, for $q=\{1, 2, 3, \cdots\}$, the spectral error from \eqn{padeWave} is
\begin{align*}
e(\kdx) = \frac{\left(\vo{C}^T(\kdx)\vo{a}_d - (-1)^{q}\kdx^{d}\vo{C}^T(\kdx)\vo{b}_d\right)	+ j \left(\vo{S}^T(\kdx)\vo{a}_d-(-1)^{q}\kdx^{d}\vo{S}^T(\kdx)\vo{b}_d\right)}{\left(\vo{C}^T(\kdx) + j\vo{S}^T(\kdx)\right)\vo{b}_d}.
\end{align*}
Therefore,
\begin{align*}
\Ltwo{e(\kdx)} &= \Ltwo{\frac{\left(\vo{C}^T(\kdx)\vo{a}_d - (-1)^{q}\kdx^{d}\vo{C}^T(\kdx)\vo{b}_d\right)	+ j \left(\vo{S}^T(\kdx)\vo{a}_d-(-1)^{q}\kdx^{d}\vo{S}^T(\kdx)\vo{b}_d\right)}{\left(\vo{C}^T(\kdx) + j\vo{S}^T(\kdx)\right)\vo{b}_d}}.
% & \leq  \frac{\Ltwo{\left(\vo{C}^T(\kdx)\vo{a}_d - (-1)^{q}\kdx^{d}\vo{C}^T(\kdx)\vo{b}_d\right)	+ j \left(\vo{S}^T(\kdx)\vo{a}_d-(-1)^{q}\kdx^{d}\vo{S}^T(\kdx)\vo{b}_d\right)}}{\Ltwo{\left(\vo{C}^T(\kdx) + j\vo{S}^T(\kdx)\right)\vo{b}_d}}.
\end{align*}
Minimization of $\Ltwo{e(\kdx)}$ is not a convex problem. Therefore, we use Lemma \ref{lem:normRelaxLemma} to relax the problem and find an upper bound on $\Ltwo{e(\kdx)}$. But, we have to ensure that $\bar{g}(\kdx)g(\kdx) \geq \epsilon^2$, where $g(\kdx):= (\vo{C}^T(\kdx) + j\vo{S}^T(\kdx))\vo{b}_d$.
Note that,
\begin{align*}
 \bar{g}(\kdx)g(\kdx) &= (\vo{C}^T(\kdx)\vo{b}_d)^2 + (\vo{S}^T(\kdx)\vo{b}_d)^2 \\
 &= (\vo{C}^T(\kdx)\vo{b}_{d,s} + \vo{C}^T(\kdx)\vo{b}_{d,w})^2 + (\vo{S}^T(\kdx)\vo{b}_{d,s} + \vo{S}^T(\kdx)\vo{b}_{d,w})^2,
\end{align*}
where $\vo{b}_{d}$ is decomposed into symmetric and skew-symmetric parts as per Definition \ref{defn:symmVec}. Using the fact that any symmetric and skew-symmetric vectors are orthogonal to each other, we get,
\begin{align*}
 \bar{g}(\kdx)g(\kdx) &= (\vo{C}^T(\kdx)\vo{b}_{d,s})^2 + (\vo{S}^T(\kdx)\vo{b}_{d,w})^2.
\end{align*}
RHS of the previous equation is sum of squares and it can be zero only if each term is zero individually. If we ensure that at least one term is non-zero, then we can use Lemma \ref{lem:normRelaxLemma}. Consider the first term,
\begin{align*}
 (\vo{C}^T(\kdx)\vo{b}_{d,s})^2 = \vo{b}_{d,s}^T \vo{C}(\kdx) \vo{C}^T(\kdx) \vo{b}_{d,s}.
\end{align*}
This term can be zero only if either $\vo{b}_{d,s} = \vo{0}$, or $\vo{b}_{d,s} \neq \vo{0}$ and $\vo{b}_{d,s} \in \text{null}(\vo{C}(\kdx) \vo{C}^T(\kdx))$. Because of the special structure of $\vo{C}(\kdx) \vo{C}^T(\kdx)$, it can be easily verified that no symmetric vector spans its null space.
This eliminates the second possibility. And we eliminate the first possibility by setting the central element of $\vo{b}_d$, i.e. $\mathrm{b}_0$ to be a non-zero real constant. Therefore,
\begin{align*}
 \mathrm{b}_0 = \kappa_0, \quad \text{where,} \quad  \kappa_0 \in \real \backslash \{0\}.
\end{align*}
This also ensures that, $\forall \mkern9mu \mathrm{b}_0 = \kappa_0,	\mkern9mu \exists \mkern9mu \epsilon \in \real \backslash \{0\}$, such that  $(\vo{C}^T(\kdx)\vo{b}_{d,s})^2 \geq \epsilon^2$.
Consequently,
\begin{align*}
 \bar{g}(\kdx)g(\kdx) &= (\vo{C}^T(\kdx)\vo{b}_d)^2 + (\vo{S}^T(\kdx)\vo{b}_d)^2 \geq \epsilon^2, \quad \text{if} \mkern9mu \mathrm{b}_0 = \kappa_0 \neq 0.
\end{align*}
Now we can use Lemma \ref{lem:normRelaxLemma} to get
\begin{align}
\Ltwo{e(\kdx)} &= \Ltwo{\frac{\left(\vo{C}^T(\kdx)\vo{a}_d - (-1)^{q}\kdx^{d}\vo{C}^T(\kdx)\vo{b}_d\right)	+ j \left(\vo{S}^T(\kdx)\vo{a}_d-(-1)^{q}\kdx^{d}\vo{S}^T(\kdx)\vo{b}_d\right)}{\left(\vo{C}^T(\kdx) + j\vo{S}^T(\kdx)\right)\vo{b}_d}} \\
& \leq  \frac{\Ltwo{\left(\vo{C}^T(\kdx)\vo{a}_d - (-1)^{q}\kdx^{d}\vo{C}^T(\kdx)\vo{b}_d\right)	+ j \left(\vo{S}^T(\kdx)\vo{a}_d-(-1)^{q}\kdx^{d}\vo{S}^T(\kdx)\vo{b}_d\right)}}{\epsilon^2}.
\eqnlabel{normRelax1}
\end{align}

We can minimize $\Ltwo{e(\kdx)}$ by minimizing the numerator of the previous inequality since $\epsilon$ is a constant.
The actual values of $\epsilon$ and $\kappa_0$ are immaterial. The constant $\kappa_0$ can be chosen arbitrarily.
In fact it can be verified empirically that numerical values of optimal $\vo{a}_d$ and $\vo{b}_d$ normalized by $\kappa_0$ are independent of its prescribed value.
For simplicity, we select $\kappa_0$ and hence $\mathrm{b}_0$ to be unity, as in \cite{lele1992compact}. This also provides an intrinsic scaling for coefficients, where all coefficients are scaled with respect to $\mathrm{b}_0$.
Therefore, a constraint on $\mathrm{b}_0$ is imposed as
\begin{align}
 \mathrm{b}_0 = 1; \quad \text{or,} \quad \bdel_N^T(\npt+1)\vo{b}_d = 1,
 \eqnlabel{nonSingular}
\end{align}
where $\bdel_n^T(\cdot) \in \real ^n$ is a vector defined as follows
\begin{align}
\bdel_n^T(i) := \begin{bmatrix} \delta(i-1), & \delta(i-2),& \hdots, & \delta(i-(\npt+1)), & \hdots, & \delta(i-n) \end{bmatrix},
\eqnlabel{deltavec}
\end{align}
where $\delta(\cdot)$ is the scalar Kronecker delta function defined as
\begin{align}
\delta(k)  := \left\{\begin{array}{c} 0 \text{ if } k\neq 0,\\ 1 \text{ if } k = 0.\end{array}\right.
\eqnlabel{delta}
\end{align}
%Consequently,
%\begin{align}
%\Ltwo{\left(\vo{C}^T(\kdx) + j\vo{S}^T(\kdx)\right)\vo{b}_d}\geq 1, \eqnlabel{finiteDenom}
%\end{align}
%which implies
%\begin{align}
%\Ltwo{e(\kdx)} \leq \Ltwo{\left(\vo{C}^T(\kdx)\vo{a}_d - (-1)^{q}\kdx^{d}\vo{C}^T(\kdx)\vo{b}_d\right)	+ j \left(\vo{S}^T(\kdx)\vo{a}_d-(-1)^{q}\kdx^{d}\vo{S}^T(\kdx)\vo{b}_d\right)}.\eqnlabel{upperBound}
%\end{align}
We therefore minimize $\Ltwo{e(\kdx)}$ by minimizing
\begin{align*}
&  \Ltwo{\left(\vo{C}^T(\kdx)\vo{a}_d - (-1)^{q}\kdx^{d}\vo{C}^T(\kdx)\vo{b}_d\right)	+ j \left(\vo{S}^T(\kdx)\vo{a}_d-(-1)^{q}\kdx^{d}\vo{S}^T(\kdx)\vo{b}_d\right)} \\
&=   \left\langle\left(\vo{C}^T(\kdx)\vo{a}_d - (-1)^{q}\kdx^{d}\vo{C}^T(\kdx)\vo{b}_d\right)^2 +  \left(\vo{S}^T(\kdx)\vo{a}_d-(-1)^{q}\kdx^{d}\vo{S}^T(\kdx)\vo{b}_d\right)^2\right\rangle,\\
& =   \begin{pmatrix}\vo{a}_d\\\vo{b}_d\end{pmatrix}^T\left\langle\begin{bmatrix}\vo{C}(\kdx)\\-(-1)^q\kdx^d\vo{C}(\kdx)\end{bmatrix}\begin{bmatrix}\vo{C}(\kdx)\\-(-1)^q\kdx^d\vo{C}(\kdx)\end{bmatrix}^T + \begin{bmatrix}\vo{S}(\kdx)\\-(-1)^q\kdx^d\vo{S}(\kdx)\end{bmatrix}\begin{bmatrix}\vo{S}(\kdx)\\-(-1)^q\kdx^d\vo{S}(\kdx)\end{bmatrix}^T\right\rangle\begin{pmatrix}\vo{a}_d\\\vo{b}_d\end{pmatrix},\\
& =   \begin{pmatrix}\vo{a}_d\\\vo{b}_d\end{pmatrix}^T\vo{Q}_d\begin{pmatrix}\vo{a}_d\\\vo{b}_d\end{pmatrix},
\end{align*}
where
\begin{align*}
\vo{Q}_d = \left\langle\begin{bmatrix}\vo{C}(\kdx)\\-(-1)^q\kdx^d\vo{C}(\kdx)\end{bmatrix}\begin{bmatrix}\vo{C}(\kdx)\\-(-1)^q\kdx^d\vo{C}(\kdx)\end{bmatrix}^T + \begin{bmatrix}\vo{S}(\kdx)\\-(-1)^q\kdx^d\vo{S}(\kdx)\end{bmatrix}\begin{bmatrix}\vo{S}(\kdx)\\-(-1)^q\kdx^d\vo{S}(\kdx)\end{bmatrix}^T\right\rangle.
\end{align*}
The optimization problem is therefore
\begin{equation}\left.
\begin{aligned}
&\min_{\vo{a}_d,\vo{b}_d\in\real^N}   \begin{pmatrix}\vo{a}_d\\\vo{b}_d\end{pmatrix}^T\vo{Q}_d\begin{pmatrix}\vo{a}_d\\\vo{b}_d\end{pmatrix},\\
\text{subject to }\;\; & \vo{a}_d^T\vo{X}_d - \vo{b}_d^T\vo{Y}_d = \vo{0}_{1\times (d+p+1)},\\
& \bdel_N^T(\npt+1)\vo{b}_d = 1.
\end{aligned}\;\;\right\}
\eqnlabel{evenOptim}
\end{equation}
This is a quadratic programming problem with equality constraints and the solution can be determined analytically. The equality constraints can be written as
\begin{align}
\begin{bmatrix}
\vo{X}_d^T & -\vo{Y}_d^T\\
\vo{0}_{1 \times N} & \bdel_N^T(\npt+1)
\end{bmatrix}\begin{pmatrix}\vo{a}_d\\\vo{b}_d\end{pmatrix} = \begin{bmatrix}\vo{0}_{(d+p+1)\times 1} \\ 1\end{bmatrix}.
\eqnlabel{padeConstr}
\end{align}
The Karush-Kuhn-Tucker (KKT) condition for this problem is
\begin{align*}
\left[\begin{array}{cc}   \vo{Q}_d & \left(\begin{array}{cc}\vo{X}_d^T & -\vo{Y}_d^T\\
\vo{0}_{1 \times N} & \bdel_N^T(\npt+1)\end{array}\right)^T\\
 \left(\begin{array}{cc}\vo{X}_d^T & -\vo{Y}_d^T\\
\vo{0}_{1 \times N} & \bdel_N^T(\npt+1)\end{array}\right) & \vo{0}_{(d+p+2)\times(d+p+2)}
\end{array}\right]\begin{pmatrix}\vo{a}_d^\ast\\\vo{b}_d^\ast\\\boldsymbol{\lambda}_d^\ast\end{pmatrix} = \begin{bmatrix}
\vo{0}_{2N\times 1}\\\vo{0}_{(d+p+1)\times 1} \\ 1
\end{bmatrix},
\end{align*}
and $\boldsymbol{\lambda}_d$ is the Lagrange multiplier associated with \eqn{padeConstr}.
Assuming the inverse exists, the optimal solution is given by
\begin{align}
\begin{pmatrix}\vo{a}_d^\ast\\\vo{b}_d^\ast\\\boldsymbol{\lambda}_d^\ast\end{pmatrix} = \left[\begin{array}{cc}  \vo{Q}_d & \left(\begin{array}{cc}\vo{X}_d^T & -\vo{Y}_d^T\\
\vo{0}_{1 \times N} & \bdel_N^T(\npt+1)\end{array}\right)^T\\
 \left(\begin{array}{cc}\vo{X}_d^T & -\vo{Y}_d^T\\
\vo{0}_{1 \times N} & \bdel_N^T(\npt+1)\end{array}\right) & \vo{0}_{(d+p+2)\times(d+p+2)}
\end{array}\right]^{-1}\begin{bmatrix}
\vo{0}_{2N\times 1}\\\vo{0}_{(d+p+1)\times 1} \\ 1
\end{bmatrix}.
\eqnlabel{padeEvenOptimal}
\end{align}
The optimal coefficients $\vo{a}_d^*$ and $\vo{b}_d^*$ turn out to be symmetric. The symmetry of coefficients, imposed a priori in the formulations discussed in \sect{intro}, appear here as a consequence of the optimization.
Analytical justification for the symmetry is given by the following lemma.

\begin{lem}
	\label{lem:evenImgZero}
	In case of even derivatives, optimal coefficients, both $\vo{a}_d^*$ and $\vo{b}_d^*$ are symmetric about the central element.
\end{lem}
\begin{proof}
	Upper bound on the norm of spectral error for an even derivative is given by \eqn{normRelax1}
	\begin{align}
	\Ltwo{e(\kdx)} \leq \left\langle\left(\vo{C}^T(\kdx)\vo{a}_d - (-1)^{q}\kdx^{d}\vo{C}^T(\kdx)\vo{b}_d\right)^2 +  \left(\vo{S}^T(\kdx)\vo{a}_d-(-1)^{q}\kdx^{d}\vo{S}^T(\kdx)\vo{b}_d\right)^2\right\rangle / \epsilon^2.
   \eqnlabel{appAnormbound}
	\end{align}
	Now, let us consider the integrand function in \eqn{appAnormbound}
	\begin{eqnarray}
	f(\vo{a}_d,\vo{b}_d) = \wtfun \left[ \left(\vo{C}^T(\kdx)\vo{a}_d - (-1)^{q}\kdx^{d}\vo{C}^T(\kdx)\vo{b}_d\right)^2 +  \left(\vo{S}^T(\kdx)\vo{a}_d-(-1)^{q}\kdx^{d}\vo{S}^T(\kdx)\vo{b}_d\right)^2 \right].
	\eqnlabel{fadbd}
	\end{eqnarray}
	After substituting $(\vo{J}\vo{a}_d,\vo{J}\vo{b}_d)$ in \eqn{fadbd} we get,
	\begin{eqnarray}
	f(\vo{J}\vo{a}_d,\vo{J}\vo{b}_d) = \wtfun \left[ \left(\vo{C}^T(\kdx)\vo{J}\vo{a}_d - (-1)^{q}\kdx^{d}\vo{C}^T(\kdx)\vo{J}\vo{b}_d\right)^2 + \left(\vo{S}^T(\kdx)\vo{J}\vo{a}_d-(-1)^{q}\kdx^{d}\vo{S}^T(\kdx)\vo{J}\vo{b}_d\right)^2 \right].
	\eqnlabel{fJadJbd}
	\end{eqnarray}
	Clearly from \eqn{evenCJSJ},
	\begin{eqnarray}
	f(\vo{J}\vo{a}_d,\vo{J}\vo{b}_d) = f(\vo{a}_d,\vo{b}_d).
	\eqnlabel{fprop}
	\end{eqnarray}
	This is a property of the function $f(\vo{a}_d,\vo{b}_d)$ and holds true for any feasible pair $(\vo{a}_d,\vo{b}_d)$.
	Let $(\vo{a}_d^*,\vo{b}_d^*)$ be an optimal pair which minimizes the norm. It immediately follows from \eqn{fprop},
	\begin{align*}
	f(\vo{J}\vo{a}_d^*,\vo{J}\vo{b}_d^*) = f(\vo{a}_d^*,\vo{b}_d^*).
	\end{align*}
	For $(\vo{a}_d^*,\vo{b}_d^*)$ being a unique solution which minimizes the norm as per equation (\ref{eqn:padeEvenOptimal}), it follows that,
	\begin{eqnarray}
	\vo{J}\vo{a}_d^* = \vo{a}_d^* \quad \text{and} \quad \vo{J}\vo{b}_d^* = \vo{b}_d^*.
	\eqnlabel{adbdSymm}
	\end{eqnarray}
	By Definition \ref{defn:symmVec}, both  $\vo{a}_d^*$ and $\vo{b}_d^*$ are symmetric about the central element.
	To show that this solution satisfies order of accuracy constraints, consider \eqn{linPade},
	\begin{eqnarray}
	\vo{a}_d^T\vo{X}_d - \vo{b}_d^T\vo{Y}_d = \vo{0}_{1\times (d+p+1)} .
	\eqnlabel{orderConstraint}
	\end{eqnarray}
	Substitute $(\vo{a}_d,\vo{b}_d)$ with $(\vo{J}\vo{a}_d, \vo{J}\vo{b}_d)$ to get,
	\begin{align*}
	\vo{a}_d^T\vo{J}^T\vo{X}_d - \vo{b}_d^T\vo{J}^T\vo{Y}_d = \vo{0}_{1\times (d+p+1)} .
	\end{align*}
	Note that, $\vo{J}^T=\vo{J}$, and it operates on the columns of $\vo{X}_d$ and $\vo{Y}_d$. Let $\vo{X}_i$ and $\vo{Y}_i$  be $i^\text{th}$ columns of $\vo{X}_d$ and $\vo{Y}_d$. From equations \eqn{structureX} and \eqn{structureY}, it is evident that $\vo{X}_i$ and $\vo{Y}_i$ are symmetric about central element for odd $i$. It implies that $\vo{a}_d^T\vo{J}\vo{X}_i = \vo{a}_d^T\vo{X}_i$ and $\vo{b}_d^T\vo{J}\vo{Y}_i = \vo{b}_d^T\vo{Y}_i$. Clearly, constraint equation \eqn{orderConstraint} is satisfied for odd columns of $\vo{X}_d$ and $\vo{Y}_d$. Similarly, for even $i$, it can be shown that $\vo{a}_d^T\vo{J}\vo{X}_i = -\vo{a}_d^T\vo{X}_i$ and $\vo{b}_d^T\vo{J}\vo{Y}_i = -\vo{b}_d^T\vo{Y}_i$, which again satisfies \eqn{orderConstraint} since right hand side is zero. Hence, it can be concluded that, if there is exists a pair $(\vo{a}_d, \vo{b}_d)$ which satisfies the order constraint then $(\vo{J}\vo{a}_d, \vo{J}\vo{b}_d)$ also satisfies the constraint given by \eqn{orderConstraint}. Therefore, $(\vo{a}_d^*, \vo{b}_d^*)$ which satisfy \eqn{adbdSymm} is a feasible solution for \eqn{orderConstraint}.
\end{proof}
The symmetry of coefficients has a compelling effect on the spectral error, which is discussed below.
\begin{cor}
	\label{cor:evenImgZero}
	For even derivatives, imaginary component of the spectral error is zero.
\end{cor}
\begin{proof}
	Multiplying \eqn{adbdSymm} by $\vo{S}^T(\kdx)$ and using \eqn{evenCJSJ}, we get,
	\begin{eqnarray}
	  \vo{S}^T(\kdx)\vo{a}_d^* = 0 \quad \text{and} \quad \vo{S}^T(\kdx)\vo{b}_d^* = 0.
		\eqnlabel{evenSadSbdZero}
	\end{eqnarray}
 Thus, spectral error from \eqn{padeWave} becomes,
	\begin{align*}
	  e(\kdx) = \frac{\vo{C}^T(\kdx)\vo{a}_d^*}{\vo{C}^T(\kdx)\vo{b}_d^*} - (-1)^q\kdx^d.
	\end{align*}
	Therefore, for even derivatives, optimal coefficients make imaginary component of the spectral error $e(\kdx)$ zero.
\end{proof}

\Fig{stencilCoeffD2} shows the solution of the optimization problem \eqn{evenOptim} for approximating second derivative with fourth order accuracy, for various stencil sizes.
As shown in Lemma \ref{lem:evenImgZero}, both $\vo{a}_d^*$ and $\vo{b}_d^*$ are symmetric about the central element. Numerical values of the coefficients are tabulated in Table \ref{table:stencilCoeffD2} in Appendix \ref{sec:appOptimCoeff}.

Note that, for $\npt=1$, there is no degree of freedom in the system for optimization. Consequently, for this case we recover the non-optimized standard scheme. In other words, \ofdC{1}{1}{1}{1}{4} is identical to the non-optimized \sfdC{1}{1}{1}{1}{4}.

\begin{figure}[ht!]
\begin{center}
\includegraphics[width=\textwidth,]{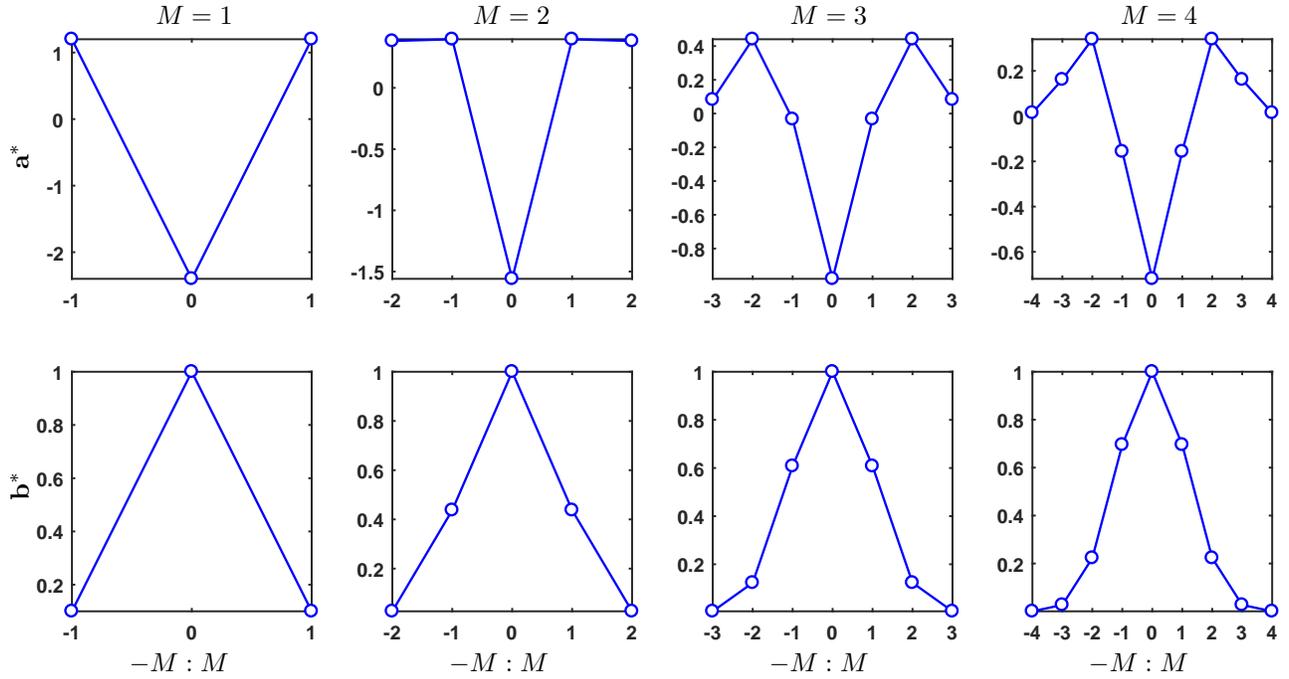} %
\begin{picture}(0,0)
        \put(-245,190){\rotatebox{90} {$\vo{a}^*$}}
        \put(-245,65){\rotatebox{90} {$\vo{b}^*$}}
				\put(-190,245){$\npt = 1$}
				\put(-70,245){$\npt = 2$}
        \put(50,245){$\npt = 3$}
        \put(175,245){$\npt = 4$}
				\put(-200,0){$-\npt:\npt$}
				\put(-79,0){$-\npt:\npt$}
        \put(42,0){$-\npt:\npt$}
        \put(163,0){$-\npt:\npt$}
\end{picture}
\caption{Optimal stencil coefficients for \ofdC{\npt}{\npt}{\npt}{\npt}{4} approximating the second derivative. $\gamma(\kdx) = 1$ for $\kdx\in[0, 3]$, and $\gamma(\kdx) = 0$ otherwise.}
\figlabel{stencilCoeffD2}
\end{center}
\end{figure}

The first row of \fig{implicitspectralError2} shows the comparison of $\mkdx^2$ (defined in \eqn{defmkdx}) for different values of $\npt$ (blue), with $\kdx^2$ which is calculated analytically (black dashed). In all subplots, numerical results obtained for \sfdC{1}{1}{1}{1}{4} (red) are also shown for comparison.
It is evident that, $\mkdx^2$ for optimized schemes is closer to analytical $\kdx^2$ than the standard scheme.

Nature of the spectral error can be further studied by analyzing real and imaginary components of $\err$, which are shown in bottom two rows of \fig{implicitspectralError2}. For a better visualization, absolute value of the real component $\Re[e(\kdx)]$ is plotted in log scale.
Clearly, optimized schemes (blue) have better spectral accuracy than the non-optimized standard scheme (red), especially at higher wavenumbers.
The real component $\Re[e(\kdx)]$ is significant for optimized schemes as well. However, it decreases as $\npt$ is increased.

In accordance with Corollary \ref{cor:evenImgZero}, we observe that the imaginary component $\Im[e(\kdx)]$ is zero for the second derivative.

\begin{figure}[ht!]
\begin{center}
\includegraphics[width=\textwidth]{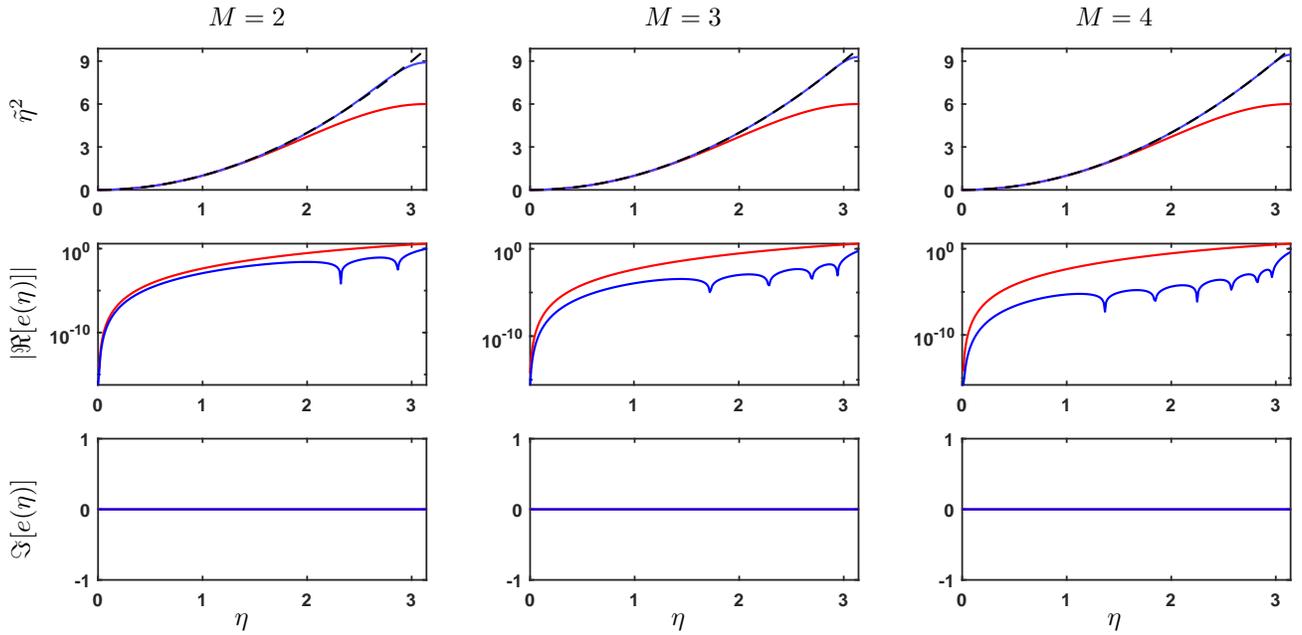}
\begin{picture}(0,0)
        \put(-250,195){\rotatebox{90} {$\mkdx ^2$}}
				\put(-250,105){\rotatebox{90} {$| \Re[e(\kdx)] |$}}
				\put(-250,30){\rotatebox{90} {$\Im[e(\kdx)]$}}
				\put(-165,5){$\kdx$}
				\put(0,5){$\kdx$}
				\put(165,5){$\kdx$}
				\put(-175,232){$\npt = 2$}
				\put(-10,232){$\npt = 3$}
				\put(153,232){$\npt = 4$}
\end{picture}
\caption{Modified wavenumber, and real and imaginary components of spectral error for \ofdC{\npt}{\npt}{\npt}{\npt}{4} (blue) approximating the second derivative. $\gamma(\kdx) = 1$ for $\kdx\in[0, 3]$, and $\gamma(\kdx) = 0$ otherwise. While analytical $\kdx^2$ is shown by black dashed line, red solid line shows the results for \sfdC{1}{1}{1}{1}{4}.}
\figlabel{implicitspectralError2}
\end{center}
\end{figure}

The overall effect of changing the stencil size (or $\npt$), on the spectral error is illustrated in  \fig{implicitSecondDerivativeL2error}. For a fixed order of accuracy, larger $\npt$ provides more freedom, and hence bigger scope of minimizing the spectral error in the optimization problem.
Therefore, as expected, the spectral error decreases with increase in $\npt$.

\begin{figure}[ht!]
\begin{center}
\includegraphics[width=0.5\textwidth]{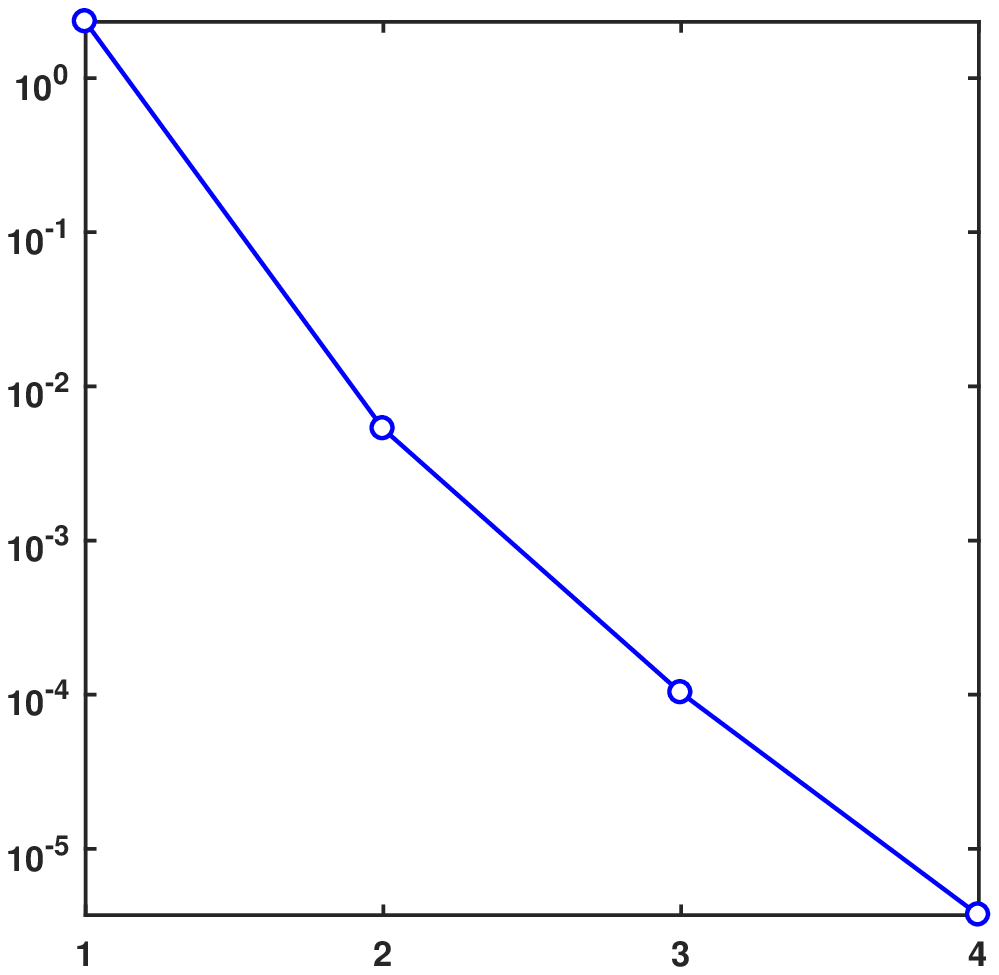}
\begin{picture}(0,0)
        \put(-255,100){\rotatebox{90} {$\|e(\kdx)\|_{\mathcal{L}_2}^2$}}
				\put(-120,-10){$\npt$}
\end{picture}
\caption{${\mathcal{L}_2}$ norm of spectral error for \ofdC{\npt}{\npt}{\npt}{\npt}{4} approximating the second derivative. $\gamma(\kdx) = 1$ for $\kdx\in[0, 3]$, and $\gamma(\kdx) = 0$ otherwise.}
\figlabel{implicitSecondDerivativeL2error}
\end{center}
\end{figure}

\subsection{Odd derivatives}
For odd $d$, i.e. $d:=2q+1$, for $q=\{0, 1, 2,  \cdots\}$, the spectral error from \eqn{padeWave} is
\begin{align*}
e(\kdx) = \frac{\left(\vo{C}^T(\kdx)\vo{a}_d + (-1)^{q}\kdx^{d}\vo{S}^T(\kdx)\vo{b}_d\right) + j\left(\vo{S}^T(\kdx)\vo{a}_d- (-1)^{q}\kdx^{d}\vo{C}^T(\kdx)\vo{b}_d\right)}{\left(\vo{C}^T(\kdx)+j\vo{S}^T(\kdx)\right)\vo{b}_d}.
\end{align*}
Using arguments similar to the case for even $d$, we can minimize $\Ltwo{e(\kdx)}$ by minimizing
\begin{align*}
& \Ltwo{\left(\vo{C}^T(\kdx)\vo{a}_d + (-1)^{q}\kdx^{d}\vo{S}^T(\kdx)\vo{b}_d\right) + j\left(\vo{S}^T(\kdx)\vo{a}_d- (-1)^{q}\kdx^{d}\vo{C}^T(\kdx)\vo{b}_d\right)},\\
&=  \begin{pmatrix}\vo{a}_d\\\vo{b}_d\end{pmatrix}^T\vo{Q}_d\begin{pmatrix}\vo{a}_d\\\vo{b}_d\end{pmatrix},
\end{align*}
where
\begin{align}
\vo{Q}_d = \left\langle\begin{bmatrix}\vo{C}(\kdx)\\(-1)^q\kdx^d\vo{S}(\kdx)\end{bmatrix}\begin{bmatrix}\vo{C}(\kdx)\\(-1)^q\kdx^d\vo{S}(\kdx)\end{bmatrix}^T + \begin{bmatrix}\vo{S}(\kdx)\\-(-1)^q\kdx^d\vo{C}(\kdx)\end{bmatrix}\begin{bmatrix}\vo{S}(\kdx)\\-(-1)^q\kdx^d\vo{C}(\kdx)\end{bmatrix}^T\right\rangle.
\eqnlabel{oddQd}
\end{align}
The optimization problem is the same as \eqn{evenOptim}, with $\vo{Q}_d$ defined by \eqn{oddQd}, and the analytical solution is given by \eqn{padeEvenOptimal}.

Symmetry or skew-symmetry of optimal coefficients obtained for odd derivatives is justified by the following lemma. And its effect on the spectral error is discussed in Corollary \ref{cor:oddRealZero}.
\begin{lem}
	\label{lem:oddRealZero}
	In case of odd derivatives, optimal coefficients, $\vo{a}_d^*$ and $\vo{b}_d^*$ are respectively, skew-symmetric and symmetric about the central element.
\end{lem}

\begin{proof}
	For odd derivatives, the integrand function is
	\begin{align*}
	g(\vo{a}_d,\vo{b}_d) = \wtfun \left[ \left(\vo{C}^T(\kdx)\vo{a}_d + (-1)^{q}\kdx^{d}\vo{S}^T(\kdx)\vo{b}_d\right)^2 + \left(\vo{S}^T(\kdx)\vo{a}_d- (-1)^{q}\kdx^{d}\vo{C}^T(\kdx)\vo{b}_d\right)^2 \right].
	\end{align*}
	Using similar arguments as for the case of even derivatives, it can be easily shown that,
	\begin{align*}
	g(-\vo{J}\vo{a}_d,\vo{J}\vo{b}_d) = g(\vo{a}_d,\vo{b}_d).
	\end{align*}
	If $(\vo{a}_d^*,\vo{b}_d^*)$ is the unique optimal solution, then
	\begin{eqnarray}
	-\vo{J}\vo{a}_d^* = \vo{a}_d^* \quad \text{and} \quad \vo{J}\vo{b}_d^* = \vo{b}_d^*.
	\eqnlabel{adSkewbdSymm}
	\end{eqnarray}
	By definition, $\vo{a}_d^*$ and $\vo{b}_d^*$ are respectively skew-symmetric and symmetric about the central element.
	To show that this solution satisfies order of accuracy constraints, consider \eqn{orderConstraint}. Now, substitute $(\vo{a}_d,\vo{b}_d)$ with $(-\vo{J}\vo{a}_d, \vo{J}\vo{b}_d)$ to get,
	\begin{align*}
	-\vo{a}_d^T\vo{J}^T\vo{X}_d - \vo{b}_d^T\vo{J}^T\vo{Y}_d = \vo{0}_{1\times (d+p+1)}.
	\end{align*}
	Using the similar arguments as for the case of even derivative given in Lemma (\ref{lem:evenImgZero}), it can be shown that if $(\vo{a}_d,\vo{b}_d)$ satisfies \eqn{orderConstraint}, then $(-\vo{J}\vo{a}_d, \vo{J}\vo{b}_d)$ also satisfies the same constraint for odd $d$.
	Therefore, $(\vo{a}_d^*, \vo{b}_d^*)$ which satisfy \eqn{adSkewbdSymm} is a feasible solution for \eqn{orderConstraint}.
\end{proof}
\begin{cor}
	For odd derivatives, real component of the spectral error is zero.
		\label{cor:oddRealZero}
\end{cor}
\begin{proof}
	Multiplying first sub-equation of \eqn{adSkewbdSymm} by $\vo{C}^T(\kdx)$ and the second by $\vo{S}^T(\kdx)$ and using \eqn{evenCJSJ}, we get,
	\begin{eqnarray}
	\vo{C}^T(\kdx)\vo{a}_d^* = 0 \quad \text{and} \quad \vo{S}^T(\kdx)\vo{b}_d^* = 0.
	\eqnlabel{oddCadSbdZero}
	\end{eqnarray}
Thus, spectral error from \eqn{padeWave} becomes,
	\begin{align*}
	e(\kdx) = j\bigg(\frac{\vo{S}^T(\kdx)\vo{a}_d^* }{\vo{C}^T(\kdx)\vo{b}_d^*} - (-1)^q\kdx^d \bigg).
	\end{align*}
Therefore, for odd derivatives, optimal coefficients make real component of the spectral error $e(\kdx)$ zero.
\end{proof}

\Fig{stencilCoeffD1} shows the solution of the optimization problem \eqn{evenOptim} for approximating first derivative with fourth order accuracy, for various stencil sizes.
As shown in Lemma \ref{lem:oddRealZero}, $\vo{a}_d^*$ and $\vo{b}_d^*$ are respectively, skew-symmetric and symmetric about the central element. Numerical values of the coefficients are tabulated in Table \ref{table:stencilCoeffD1}.
As in the case of second derivative, here also we recover the standard non-optimized scheme, \sfdC{1}{1}{1}{1}{4}, for $\npt=1$.

\begin{figure}[ht!]
\begin{center}
\includegraphics[width=\textwidth]{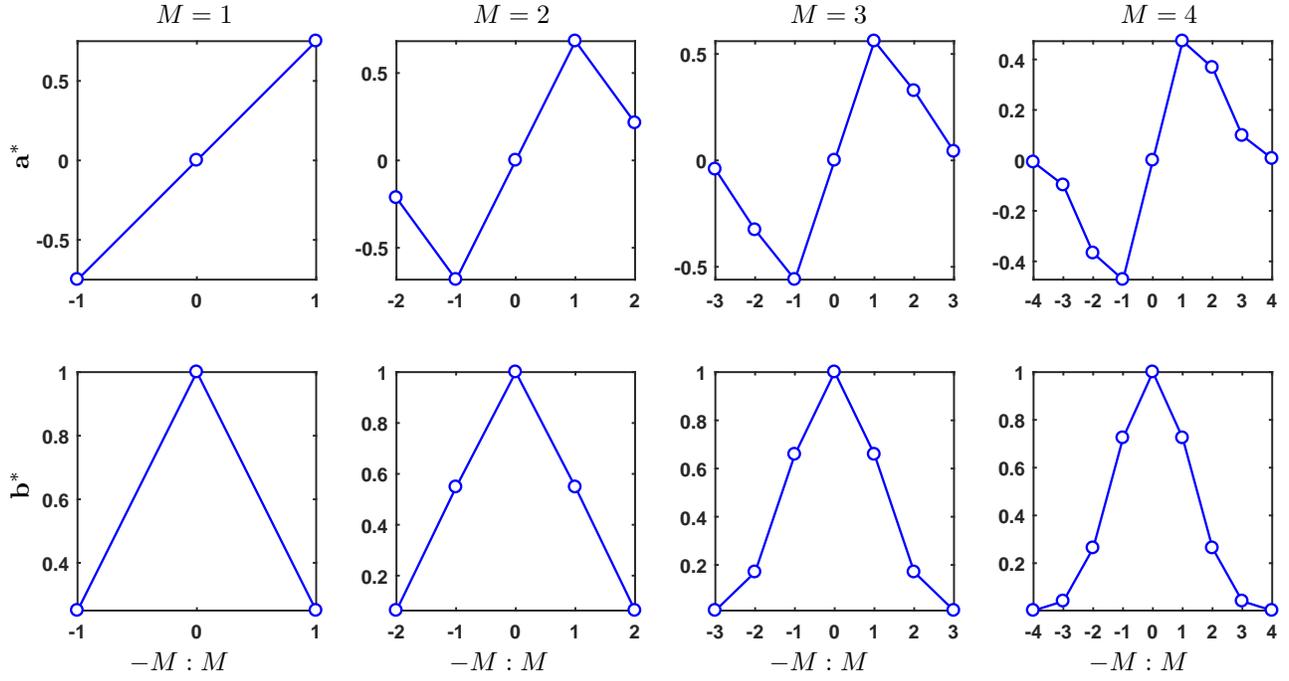} %
\begin{picture}(0,0)
    \put(-245,190){\rotatebox{90} {$\vo{a}^*$}}
    \put(-245,65){\rotatebox{90} {$\vo{b}^*$}}
    \put(-190,245){$\npt = 1$}
    \put(-70,245){$\npt = 2$}
    \put(50,245){$\npt = 3$}
    \put(175,245){$\npt = 4$}
    \put(-200,0){$-\npt:\npt$}
    \put(-79,0){$-\npt:\npt$}
    \put(42,0){$-\npt:\npt$}
    \put(163,0){$-\npt:\npt$}
\end{picture}
\caption{Optimal stencil coefficients for \ofdC{\npt}{\npt}{\npt}{\npt}{4} approximating the first derivative. $\gamma(\kdx) = 1$ for $\kdx\in[0, 3]$, and $\gamma(\kdx) = 0$ otherwise.}
\figlabel{stencilCoeffD1}
\end{center}
\end{figure}

Again, as in the case of second derivative, similar observations are made for the spectral accuracy of optimized schemes approximating the first derivative.
The top row of \fig{implicitspectralError1} shows the clear advantage of optimized schemes (blue) over the standard scheme (red).
While the absolute value of imaginary component $\Im[e(\kdx)]$ (shown in log scale) decreases with increasing $\npt$, the real component $\Re[e(\kdx)]$, as expected from Corollary \ref{cor:oddRealZero}, is zero for the first derivative.
Similar to \fig{implicitSecondDerivativeL2error}, \fig{implicitFirstDerivativeL2error} also shows that as $M$ increases, the spectral error decreases.

Although the current framework has been derived for any positive integer $\npt$, we observed empirically that KKT matrix in \eqn{padeEvenOptimal} becomes rank deficient for $\npt \geq 6$ and $\kdx \in [0,3]$, i.e., its inverse does not exist and we cannot determine optimal coefficients analytically.
We also observed that the minimum $\npt$ at which KKT matrix becomes rank deficient depends on the range of $\kdx$. However, explicit schemes discussed in \sect{sec:explicit} do not suffer from this problem. We intend to theoretically investigate this issue in our future work.

Next, we consider the effect of $\wtfun$ on spectral accuracy of optimized schemes.

\begin{figure}[ht!]
\begin{center}
\includegraphics[width=\textwidth]{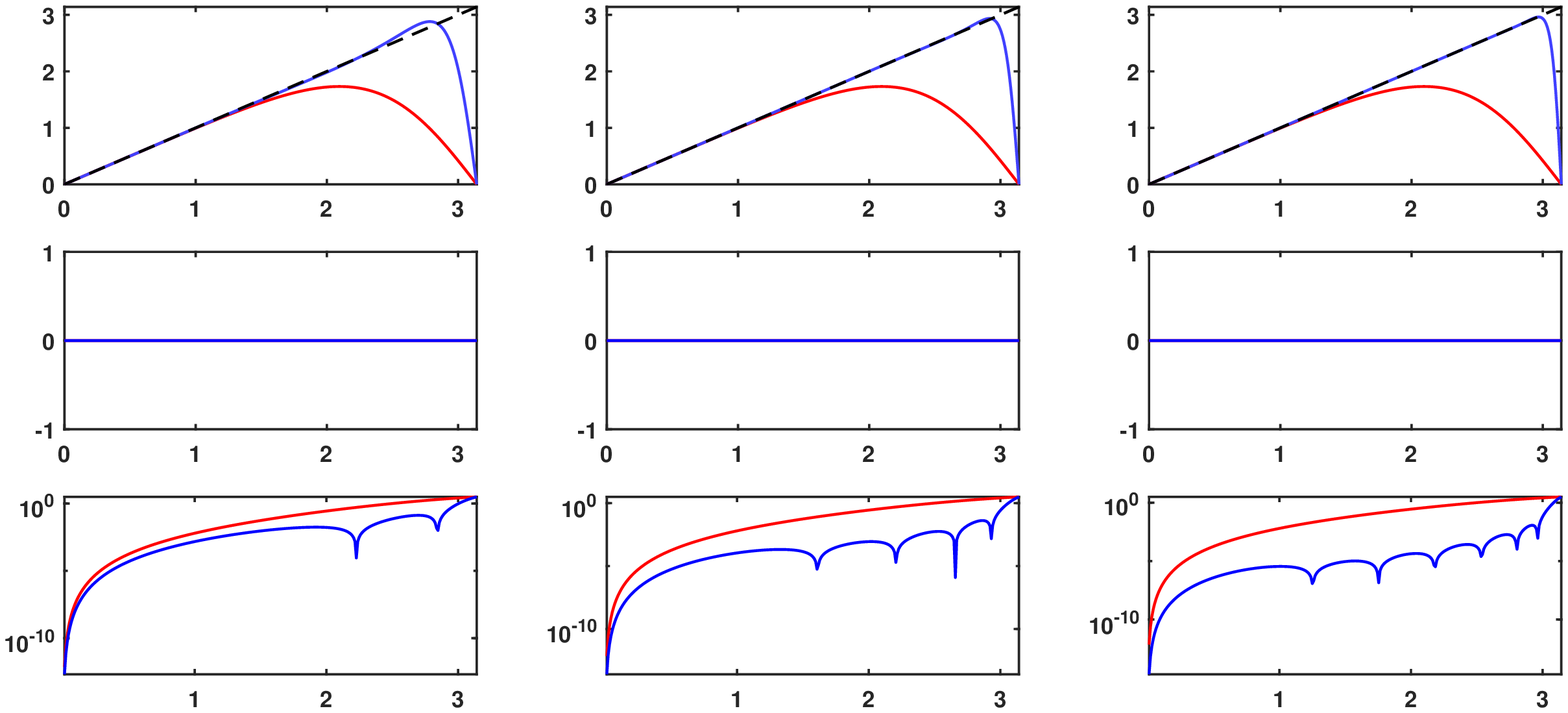}
\begin{picture}(0,0)
	\put(-250,195){\rotatebox{90} {$\mkdx$}}
	\put(-250,105){\rotatebox{90} {$\Re[e(\kdx)]$}}
	\put(-250,30){\rotatebox{90} {$|\Im[e(\kdx)]|$}}
	\put(-165,5){$\kdx$}
	\put(0,5){$\kdx$}
	\put(165,5){$\kdx$}
	\put(-175,235){$\npt = 2$}
	\put(-10,235){$\npt = 3$}
	\put(153,235){$\npt = 4$}
\end{picture}
\caption{Modified wavenumber, and real and imaginary components of spectral error for \ofdC{\npt}{\npt}{\npt}{\npt}{4} (blue) approximating the first derivative. $\gamma(\kdx) = 1$ for $\kdx\in[0, 3]$, and $\gamma(\kdx) = 0$ otherwise. While analytical $\kdx$ is shown by black dashed line, red solid line shows the results for \sfdC{1}{1}{1}{1}{4}.}
\figlabel{implicitspectralError1}
\end{center}
\end{figure}

\begin{figure}[ht!]
\begin{center}
\includegraphics[width=0.5\textwidth]{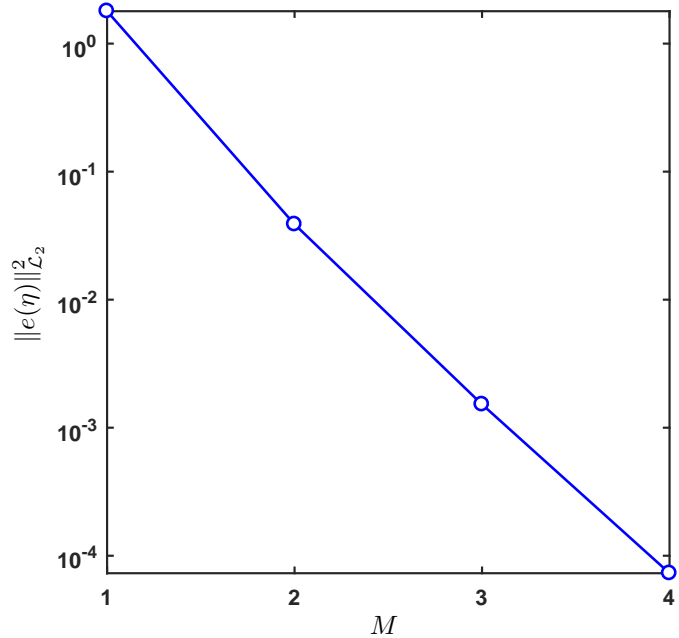} %
\begin{picture}(0,0)
        \put(-255,100){\rotatebox{90} {$\|e(\kdx)\|_{\mathcal{L}_2}^2$}}
				\put(-120,-10){$\npt$}
\end{picture}
\caption{${\mathcal{L}_2}$ norm of spectral error for \ofdC{\npt}{\npt}{\npt}{\npt}{4} approximating the first derivative. $\gamma(\kdx) = 1$ for $\kdx\in[0, 3]$, and $\gamma(\kdx) = 0$ otherwise.}
\figlabel{implicitFirstDerivativeL2error}
\end{center}
\end{figure}

\subsection{Effect of $\wtfun$}
Recall that we used weighting function $\wtfun \geq 0$ when we defined ${\mathcal{L}_2}$-norm of error in \eqn{normWithGamma}. In the optimal solutions presented so far, we have used $\wtfun = 1$ for $\kdx\in[0, 3]$ and $\wtfun = 0$ otherwise.
i.e. equal weightage was given to all wavenumbers within $[0, 3]$. We can use $\wtfun$ to weight some preferred  wavenumbers more than others and can be selected based upon physics involved in the PDE being solved.
This essentially means that, we can keep the error low at preferred wavenumbers at the expense of higher errors at other wavenumbers that are not vital to the physics of problem.

To study the effect of weighting function on spectral behavior of optimized schemes, let us consider a candidate function $\wtfun = \exp(\alpha\kdx) \geq 0$ where $\alpha \in \real$. By choosing different values of $\alpha$, one can assign different weightage for different wavenumbers.
 For example, $\alpha <0$ weights lower wavenumbers more than the higher ones, $\alpha=0$ i.e., $\wtfun=1$ weights all wavenumbers equally, and $\alpha>0$ weights higher wavenumbers more that the lower ones. This spectral behavior of optimized schemes is illustrated in \fig{gammaEffect}.

\begin{figure}[ht!]
\begin{center}

\includegraphics[width=0.65\textwidth]{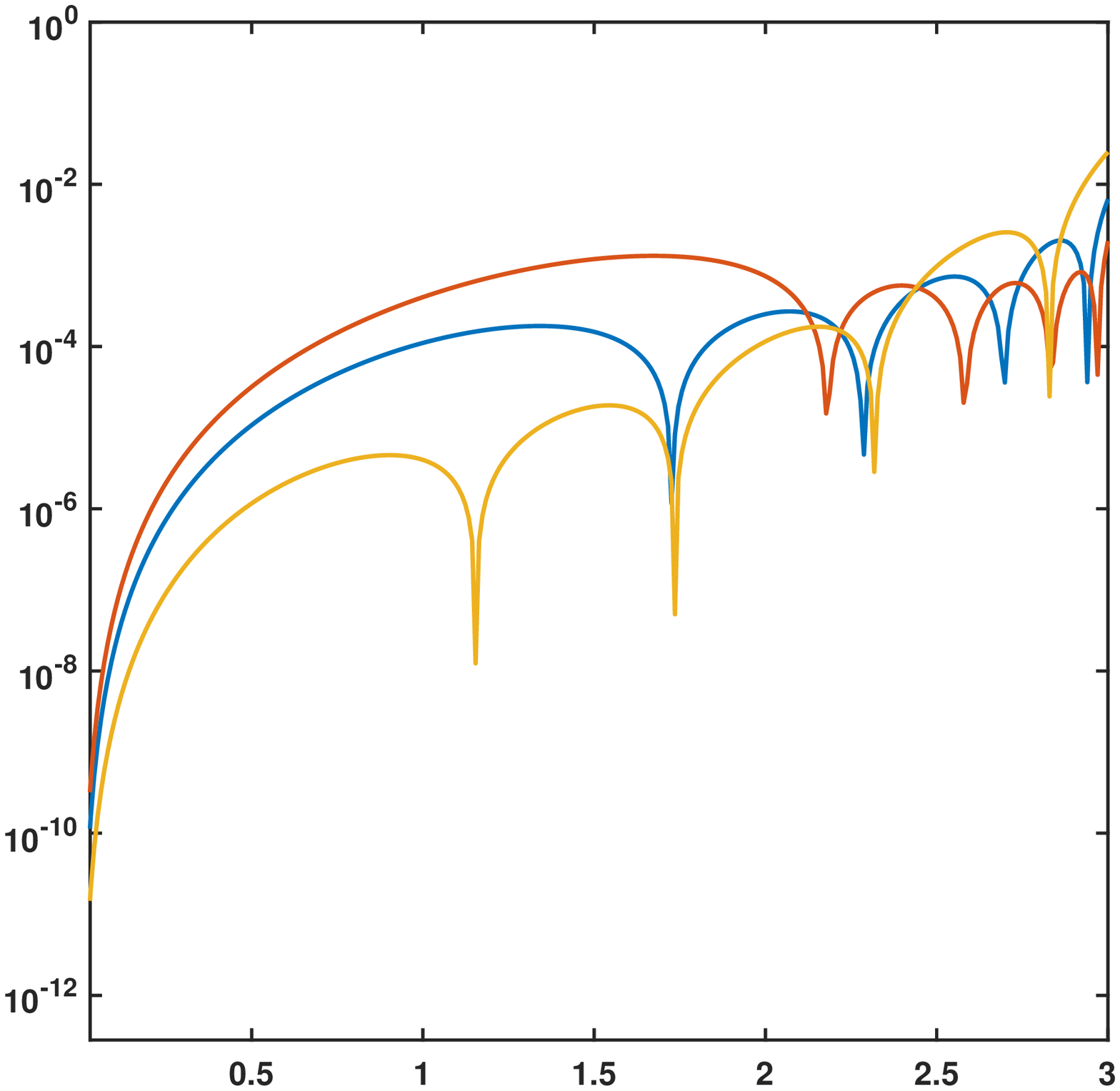} %
\begin{picture}(0,0)
        \put(-300,135){\rotatebox{90} {$\Big|\frac{\mkdx^2}{\kdx^2}-1\Big|$}}
				\put(-155,3){$\kdx$}
\end{picture}
\caption{Spectral error for \ofdC{3}{3}{3}{3}{4} approximating second derivative for $\wtfun = 1$ -- blue; $\wtfun = \exp(6\kdx)$ -- red; $\wtfun = \exp(-6\kdx)$ -- yellow, for $\kdx\in[0, 3]$, and $\wtfun = 0$ otherwise.}
\figlabel{gammaEffect}
\end{center}
\end{figure}

The choice of exponential for weighting function is purely for the purpose of illustration in \fig{gammaEffect}. One can also weight wavenumbers only over finite number of intervals, e.g., $\wtfun=\sin(\kdx)$ for $\kdx \in [0,1]$, $\wtfun=1$ for $\kdx \in [2,3]$, and $\wtfun=0$ otherwise.
From this discussion it is evident that, the flexibility of choosing $\wtfun$ can be leveraged to improve the accuracy of solution at wavenumbers that are relevant to the physics of problem.

The unified framework presented in this section allows us to derive central compact schemes with specified order of accuracy. The schemes are also optimized using $\wtfun$ to have desired spectral accuracy over the wavenumber space.
The third important aspect of discretization, in addition to order of accuracy and spectral error, is the temporal stability, which is discussed next.
But, before proceeding further, we make the following observation that is used in the stability analysis.

The optimal coefficients for $i^\text{th}$ grid point are derived using \eqn{pade1}. Since $i$ is arbitrary, it is clear that optimal coefficients must be identical for all $i$ in the domain.
For the sake of completeness, the formal proof for the same is presented as Lemma \ref{lem:gridInvariance} in Appendix \ref{sec:appInvarianceGrid}.

\section{Stability} \label{stability}
Let us consider the general linear partial differential equation
\begin{eqnarray}
\frac{\partial f}{\partial t} = \sum_{d=1}^{D} \beta_d \frac{\partial ^d f}{ \partial ^d x}.
\eqnlabel{pde}
\end{eqnarray}
Let the coefficients for the $d^\text{th}$ derivative at the $i^\text{th}$ grid point be parameterized by ${\vo{a}}_{i,d}$ and ${\vo{b}}_{i,d}$, and $\vo{A}_d$ and $\vo{B}_d$ be the vertical stacking of $\vo{a}^T_{i,d}$ and $\vo{b}^T_{i,d}$ respectively, i.e.
\begin{align}
	\vo{A}_d := \begin{bmatrix}{\vo{a}}_{1,d}^T \\ \vdots \\ {\vo{a}}_{N_p,d}^T \end{bmatrix}, \; \vo{B}_d := \begin{bmatrix}{\vo{b}}_{1,d}^T \\ \vdots \\ {\vo{b}}_{N_p,d}^T \end{bmatrix},
	\eqnlabel{defAdBd}
\end{align}
where $N_p$ is the total number of grid points in the domain.
We next introduce a shift operator $\shift{k}$, which is an $N_p\times N_p$ matrix with elements
\begin{align*}
        \Phi_{k_{ij}} := \delta ((i-j-k) \mod N_p),
\end{align*}
where $\delta(\cdot)$ is the Kronecker  delta function defined in \eqn{delta}. When applied to a vector, the operator $\shift{k}$ cyclically shifts the elements in the vector $k$ times. For example,
$$
\shift{1} :=  \delta ((i-j-1) \mod 4) = \begin{bmatrix}
        0 & 0 & 0 & 1 \\
        1 & 0 & 0 & 0 \\
        0 & 1 & 0 & 0 \\
        0 & 0 & 1 & 0
\end{bmatrix}, \quad \vo{v}:=\begin{bmatrix} 1\\2\\3\\4 \end{bmatrix}.
$$
Then,
$$
 \shift{1}\vo{v} = \begin{bmatrix} 4\\1\\2\\3 \end{bmatrix}, \quad \text{and} \quad \vo{v}^T\shift{1} = \begin{bmatrix} 2&3&4&1 \end{bmatrix}.
$$
Let   $\vo{A}_d^\vo{\Phi}$ and $\vo{B}_d^\vo{\Phi}$ be the vertical stacking of ${\vo{a}}^T_{i,d}\shift{\npt-i+1}$ and ${\vo{b}}^T_{i,d}\shift{\npt-i+1}$ respectively, i.e.
\begin{align}
	 \vo{A}_d^\vo{\Phi} := \begin{bmatrix}{\vo{a}}_{1,d}^T\shift{\npt}  \\ \vdots \\ {\vo{a}}_{N_p,d}^T \shift{-\npt} \end{bmatrix},\; \vo{B}_d^\vo{\Phi} := \begin{bmatrix}{\vo{b}}_{1,d}^T\shift{\npt} \\ \vdots \\ {\vo{b}}_{N_p,d}^T \shift{-\npt}\end{bmatrix}
	 \eqnlabel{AdBd},
\end{align}
Therefore, we can write
\begin{align}
\vo{A}_d^\vo{\Phi} = \sum_{i=1}^{N_p} \bdel_{N_p}(i)\bdel_{N_p}(i)^T\vo{A}_d\shift{\npt-i+1},\\
\vo{B}_d^\vo{\Phi} = \sum_{i=1}^{N_p}\bdel_{N_p}(i)\bdel_{N_p}(i)^T\vo{B}_d\shift{\npt-i+1}.
\eqnlabel{affineAdBd}
\end{align}
where, $\bdel_{N_p}(\cdot)\in\real^{N_p}$, as defined in \eqn{deltavec}. Thus, $\vo{A}_d^\vo{\Phi} $ and $\vo{B}_d^\vo{\Phi}$ are linear functions of $\vo{A}_d$ and $\vo{B}_d$ defined in \eqn{defAdBd}.
Let us define two vectors,
\begin{align*}
\vo{F} := \begin{bmatrix}
        f_1 \\ \vdots \\ f_{N_p}
 \end{bmatrix}, \quad \text{and} \quad
 \vo{F}^{(d)} := \begin{bmatrix}
        f_1^{(d)} \\ \vdots \\
        f_{N_p}^{(d)}
 \end{bmatrix}.
\end{align*}
The finite-difference approximation for the $d^\text{th}$ derivative at the $i^\text{th}$ grid point is
\begin{align*}
{\vo{b}}^T_{i,d}\shift{\npt-i+1}\vo{F}^{(d)} = \frac{1}{(\dx)^d}{\vo{a}}_{i,d}^T\shift{\npt-i+1}\vo{F},
\end{align*}
for $i=\{1,2,\cdots,N_p\}$. The finite-difference approximation for the $d^\text{th}$ derivative over the entire domain is therefore
\begin{align*}
	\vo{B}_d^\vo{\Phi} \vo{F}^{(d)} =  \frac{1}{(\dx)^d} \vo{A}_d^\vo{\Phi} \vo{F},
	% \eqnlabel{fdDomain}
\end{align*}
or,
\begin{align}
\vo{F}^{(d)} = \frac{1}{(\Delta x)^d}(\vo{B}_d^\vo{\Phi})^{-1}\vo{A}_d^\vo{\Phi} \vo{F}.
\eqnlabel{genlDisc}
\end{align}
subject to the existence of the inverse. For all the schemes presented in this paper, $\vo{B}_d^\vo{\Phi}$ turns out to be invertible.

% \blah{Removed the statement that inverse always exist. Difficult to prove for a general case. $\vo{B}_d^\vo{\Phi}$ is real symmetric circulant matrix, with diagonal entries unity.}

\subsection{Stability of semi-discrete scheme} \label{semiDiscStability}
Similar to the stability analysis presented in \cite{kumari2018unified} for explicit spatial schemes, herein we generalize it for implicit schemes. Using spatial discretization \eqn{genlDisc}, we can convert PDE \eqn{pde} into an ODE which is continuous in time as follows
\begin{eqnarray}
\vo{\dot{F}} = \underbrace{\bigg( \sum_{d=1}^{D} \frac{1}{(\Delta x)^d} \beta_d (\vo{B}_d^\vo{\Phi})^{-1}\vo{A}_d^\vo{\Phi} \bigg)}_{\semiDiscSysMat:=} \vo{F} = \semiDiscSysMat \vo{F},
\eqnlabel{semiDiscODE}
\end{eqnarray}
with solution
\begin{eqnarray}
\vo{F}(t) = \exp (t\semiDiscSysMat) \vo{F_0},
\eqnlabel{SolSemiDiscODE}
\end{eqnarray}
where $\vo{F_0}$ is the initial condition. For a single Fourier mode, the approximate $d^\text{th}$ derivative can be expressed in terms of the modified wavenumber as $(j\mkdx)^d \hat{f}$. Then the original PDE becomes,
\begin{align*}
\frac{d\hat{f}}{dt} = \sum_{d=1}^{D} \beta_d(j\mkdx)^d \hat{f}.
\end{align*}
The solution is given by
\begin{align*}
\frac{\hat{f}}{\hat{f_0}} = \exp \left[ \sum_{d=1}^{D} \beta_d(j\mkdx)^d t \right],
\end{align*}
where $\hat{f_0}$ is the initial condition at $t=0$.
If the solution to the original PDE is non-increasing in time, then the discretization is considered stable if
no Fourier mode grows in time. This is satisfied if
\begin{eqnarray}
\Re \left\{\sum_{d=1}^{D} \beta_d(j\mkdx)^d \right\} \le 0.
\eqnlabel{semiDiscRealZero}
\end{eqnarray}
From Lemma \ref{lem:evenImgZero} and \ref{lem:oddRealZero}, and equation \eqn{padeWave},
\begin{align*}
(j\mkdx)^d &= \frac{\vo{C}^T(\kdx)\vo{a}_d}{\vo{C}^T(\kdx)\vo{b}_d}, \quad \text{for even $d$, and} \\
(j\mkdx)^d &= j\bigg(\frac{\vo{S}^T(\kdx)\vo{a}_d }{\vo{C}^T(\kdx)\vo{b}_d}\bigg), \quad \text{for odd $d$.}
\end{align*}
Therefore,
\begin{align*}
\sum_{d=1}^{D} \beta_d(j\mkdx)^d &= j \beta_1 \frac{\vo{S}^T(\kdx)\vo{a}_1 }{\vo{C}^T(\kdx)\vo{b}_1} + \beta_2 \frac{\vo{C}^T(\kdx)\vo{a}_2}{\vo{C}^T(\kdx)\vo{b}_2} + j \beta_3 \frac{\vo{S}^T(\kdx)\vo{a}_3 }{\vo{C}^T(\kdx)\vo{b}_3} + \beta_4 \frac{\vo{C}^T(\kdx)\vo{a}_4}{\vo{C}^T(\kdx)\vo{b}_4} ... \\
&= (\beta_2 \frac{\vo{C}^T(\kdx)\vo{a}_2}{\vo{C}^T(\kdx)\vo{b}_2} + \beta_4 \frac{\vo{C}^T(\kdx)\vo{a}_4}{\vo{C}^T(\kdx)\vo{b}_4} + ...) + j(\beta_1 \frac{\vo{S}^T(\kdx)\vo{a}_1 }{\vo{C}^T(\kdx)\vo{b}_1} + \beta_3 \frac{\vo{S}^T(\kdx)\vo{a}_3 }{\vo{C}^T(\kdx)\vo{b}_3} + ...)
\end{align*}
Then \eqn{semiDiscRealZero} implies,
\begin{eqnarray}
\Re \left\{\sum_{d=1}^{D} \beta_d(j\mkdx)^d \right\} = \beta_2 \frac{\vo{C}^T(\kdx)\vo{a}_2}{\vo{C}^T(\kdx)\vo{b}_2} + \beta_4 \frac{\vo{C}^T(\kdx)\vo{a}_4}{\vo{C}^T(\kdx)\vo{b}_4} + ... \le 0.
\eqnlabel{semiDiscStabilityIneq}
\end{eqnarray}
For even $d$, i.e. $d=2q$, the optimization of spectral error guarantees,
\begin{align*}
(j\mkdx)^d = (-1)^q \mkdx^{2q} = \frac{\vo{C}^T(\kdx)\vo{a}_{2q}}{\vo{C}^T(\kdx)\vo{b}_{2q}}.
\end{align*}
Therefore, the sign of $\frac{\vo{C}^T(\kdx)\vo{a}_{2q}}{\vo{C}^T(\kdx)\vo{b}_{2q}}$ changes alternatively with $q$ as follows,
\begin{align*}
\text{for $q = 1$,} \quad \frac{\vo{C}^T(\kdx)\vo{a}_{2}}{\vo{C}^T(\kdx)\vo{b}_{2}} &= -\mkdx^2 \le 0, \\
\text{for $q = 2$,} \quad \frac{\vo{C}^T(\kdx)\vo{a}_{4}}{\vo{C}^T(\kdx)\vo{b}_{4}} &= \mkdx^4 \ge 0, \\
\text{for $q = 3$,} \quad \frac{\vo{C}^T(\kdx)\vo{a}_{6}}{\vo{C}^T(\kdx)\vo{b}_{6}} &= -\mkdx^6 \le 0, \\
\end{align*}
and so on. If coefficients $\beta_d$ for $d=2q$ satisfy $\beta_{2q} = (-1)^{q+1} \zeta_{2q}^2$ for $\zeta_{2q} \in \real$, then  \eqn{semiDiscStabilityIneq} is implicitly satisfied. In general, if
\begin{eqnarray}
-\beta_2 \mkdx^2 + \beta_4 \mkdx^4 - \beta_6 \mkdx^6 + ... \le 0,
\eqnlabel{semiDiscStabilityIneqGen}
\end{eqnarray}
then the overall discretization is stable.

\subsection{Stability of fully discretized scheme}
After temporal discretization of \eqn{semiDiscODE} using a given temporal scheme, we get a fully discretized system. We analyze the stability of such system in this section. Our goal is to find maximum $\dt$ for which the overall discretization is stable.
We begin with the forward Euler method.

\subsubsection{Forward Euler method}
Using forward Euler method for $\partial f/\partial t$ and the optimal finite difference approximation for $\partial^df/\partial x^d$, from \eqn{semiDiscODE} we get following

\begin{align}
	\vo{F}^{k+1} = \left(\vo{I}_{N_p} +  \sum_{d=1}^{D} \frac{\dt}{(\Delta x)^d}\beta_d (\vo{B}_d^\vo{\Phi})^{-1}\vo{A}_d^\vo{\Phi}\right)\vo{F}^k.
	\eqnlabel{discDyn}
\end{align}
	This is a linear discrete-time system in $\vo{F}$, where the system matrix is dependent on the stencil coefficients $\vo{A}_d$ and $\vo{B}_d$. For stability, the eigenvalues of $ \left(\vo{I}_{N_p} + \dt \sum_d \frac{1}{(\Delta x)^d}\beta_d (\vo{B}_d^\vo{\Phi})^{-1}\vo{A}_d^\vo{\Phi}\right)$ should be less than one, which can be achieved by constraining the $2$-norm as
	\begin{align}
	\left\|\vo{I}_{N_p} + \dt \sum_{d=1}^{D} \frac{1}{(\Delta x)^d}\beta_d (\vo{B}_d^\vo{\Phi})^{-1}\vo{A}_d^\vo{\Phi}\right\|_2 \leq 1.
	\eqnlabel{pade:Stability}
	\end{align}
We use the analytical solution for $\vo{A}_d$ and $\vo{B}_d$ from spectral error optimization, and maximize $\dt$ subject to the stability constraint \eqn{pade:Stability}. The reader is referred to \cite{kumari2018unified} for the solution of this maximization problem using software packages such as \texttt{cvx} \cite{grant2008cvx}.
 Solving optimization problem by this approach may become difficult if temporal schemes other than forward Euler method are used. Therefore, we present an approach to determine maximum $\dt$ for generalized temporal schemes in the following section.

\subsubsection{Generalized temporal scheme} \label{fullyDiscStabilityRK}
Runge-Kutta (RK) methods are widely used for temporal discretization of ODEs to attain higher order of accuracy. Many temporal schemes can be expressed as special cases of implicit RK method, e.g., see Appendix \ref{sec:appButcherTab} for different temporal schemes written in Butcher tableau form of RK methods.
Thus, herein, we consider implicit RK method as a generalized temporal scheme to determine maximum $\dt$ which guarantees the stability.
An implicit $\rks$-stage RK scheme with the following Butcher tableau
\begin{align}
\begin{array}{c|c}
\rkC  & \rkA \\
\hline
\quad & \rkB^{T}
\end{array}
:=
\begin{array}{c|c c c c}
\rkc_1 & \rka_{11} & \rka_{12} & \dots & \rka_{1 \rks}\\
\rkc_2 & \rka_{21} & \rka_{22} & \dots & \rka_{2 \rks}\\
\vdots       & \vdots           & \vdots         & \ddots & \vdots  \\
\rkc_{\rks} & \rka_{\rks 1} & \rka_{\rks 2} & \dots & \rka_{\rks \rks}\\
\hline
\quad & \rkb_{1} & \rkb_{2} & \dots & \rkb_{\rks} \\
\end{array}
\eqnlabel{butcherTab}
\end{align}
is given by
\begin{eqnarray}
f^{n+1} &=& f^{n} + \dt \sum _{\indi = 1}^{\rks} \rkb_{\indi} k_{\indi},
\eqnlabel{rkAtPt}
\end{eqnarray}
\begin{align*}
k_{\indi} = \dot{f} \bigg(t_n + c_{\indi} \dt, \quad f^{n} + \dt \sum _{\indj = 1}^{\rks} \rka_{\indi \indj} k_{\indj} \bigg),
\end{align*}
where, $f^{n}$ is the value of function $f(t)$ at $n^{th}$ time step and $\dot{f}(t,f)$ is the time derivative of $f(t)$ at time $t$.
For a scalar equation such as $\dot{f}= \lambda  f$, $f^{n+1}$ can be written in terms of \textit{stability function} of the RK scheme $r(z)$, $z \in \complex$, as
\begin{align}
f^ {n+1} &= r (\lambda  \dt) f ^n,
\eqnlabel{rkScalarDisc}
\end{align}
\begin{align*}
\text{where,} \quad r (z) &= 1+z\rkB^T(\vo{I} - z\rkA)^{-1}\vo{1}_{\rks}
= \frac{\det(\vo{I} + z(\vo{1}_{\rks}\rkB^T - \rkA))}{\det(\vo{I}-z\rkA)},
\end{align*}
and $\vo{1}_{\rks}$ represents a column vector of dimension $\rks$ whose all elements are equal to $1$ \cite{butcher2008ode}. Thus, $r(z)$ is a polynomial-over-polynomial function of $z$.
The sequence formed by \eqn{rkScalarDisc} is bounded iff $|r(z)|\leq 1$, which is the stability condition for the RK scheme. For all such $z = \lambda \dt$, and known $\lambda$,
we can find bounds on $\dt$ such that the product $\lambda \dt$ lies in the \textit{stability region} of the RK scheme determined by the stability condition. The equality $|r(z)|= 1$ defines a closed curve in the complex plane.
Generally, for the explicit RK schemes, area enclosed by the stability curve defines the stability region. While in the case of implicit RK schemes, stability region may lie either inside or outside the closed curve.
If stability curve is the imaginary axis itself, then the scheme is stable for all $z$ which lie in the left half of the complex plane. Note that, $\dt>0$ is a real number while $\lambda$ can be a complex number, so that $z = \lambda\dt$ is complex.
The timestep $\dt$ just \textit{scales} the radius of $\lambda$ while keeping the argument same in the complex plane. We are interested in finding maximum scaling $\dt$ for which scaled $\lambda$ lies in the stability region.

As in the forward Euler case, here also we assume that $\vo{A}_d$ and $\vo{B}_d$ are determined analytically and hence $\semiDiscSysMat$ is known. Now, let us consider multi-dimensional semi-discrete system defined by \eqn{semiDiscODE}.
We reduce this $N_p$-dimensional system into $N_p$ scalar decoupled systems using the similarity transformation.
Let $\vo{D}$ be a diagonal matrix with eigenvalues of the $\semiDiscSysMat$ as its diagonal elements and $\vo{M}$ be the modal matrix with eigenvectors of $\semiDiscSysMat$ as its columns, such that $\vo{D}=\vo{M}^{-1}\semiDiscSysMat\vo{M}$.
Let us define the transformation $\vo{L}:=\vo{M}\vo{F}$. Then \eqn{semiDiscODE} can be written as
\begin{align*}
\vo{M}\dot{\vo{L}} &= \semiDiscSysMat\vo{M}\vo{L},
\end{align*}
\begin{align}
\dot{\vo{L}}  &= \vo{M}^{-1}\semiDiscSysMat\vo{M}\vo{L} = \vo{D}\vo{L}.
\eqnlabel{decoupledSys}
\end{align}
Since $\vo{D}$ is diagonal, \eqn{decoupledSys} represents $N_p$ scalar systems as follows
\begin{align*}
\dot{l}_ i &= \lambda _ i l _ i, \quad i = 1 \dots N_p,
\end{align*}
where $\lambda _ i$, are $N_p$ eigenvalues of $\semiDiscSysMat$. The $\ith$ scalar system can be fully discretized using the RK scheme, and from \eqn{rkScalarDisc} it follows that
\begin{align*}
l_ i ^ {n+1} &= r (\lambda _ i \dt) l _ i ^n, \quad i = 1 \dots N_p.
\end{align*}
For overall stability of \eqn{decoupledSys}, it is required that the products $\lambda _ i \dt$ lie within stability region of the RK scheme. Thus, for known $\lambda _i$, bounds on
$\dt$ can be calculated numerically to guarantee the stability of the system.

\begin{figure}[htb!]
\begin{center}
\includegraphics[width=0.4\textwidth,]{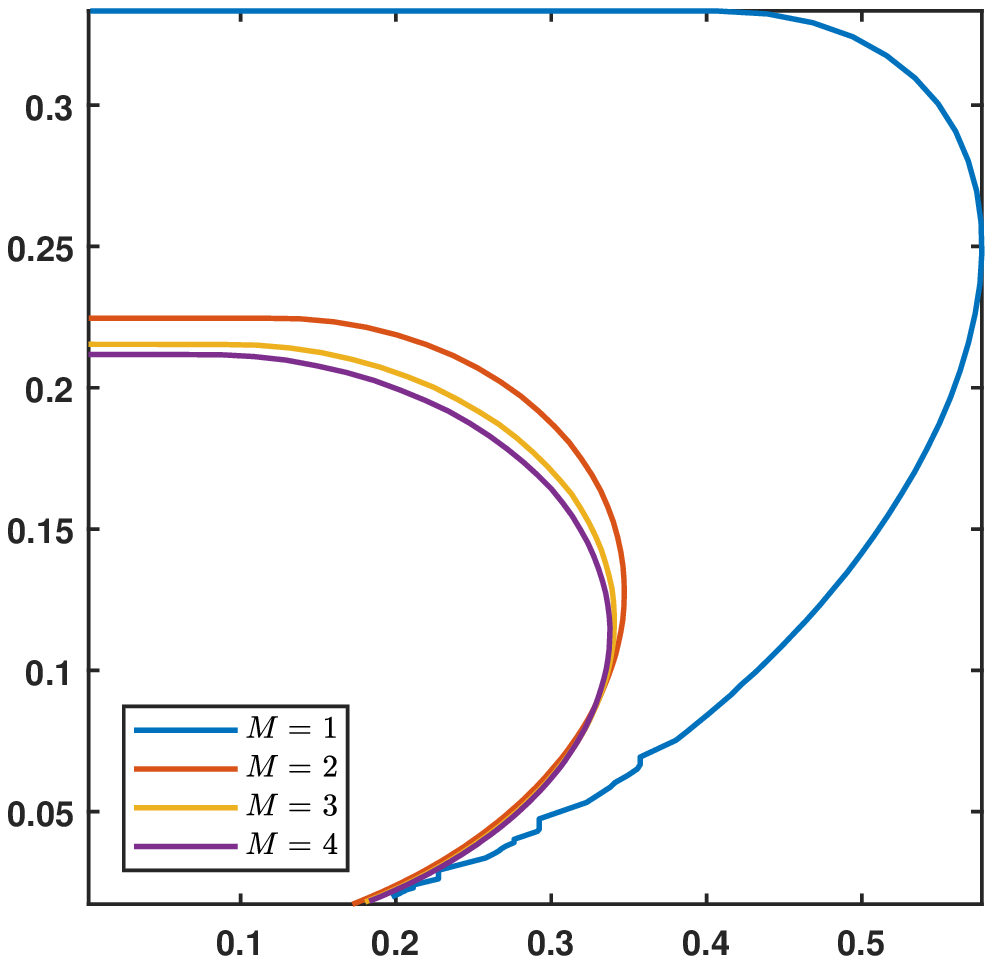} \hspace{20mm}
\includegraphics[width=0.4\textwidth,]{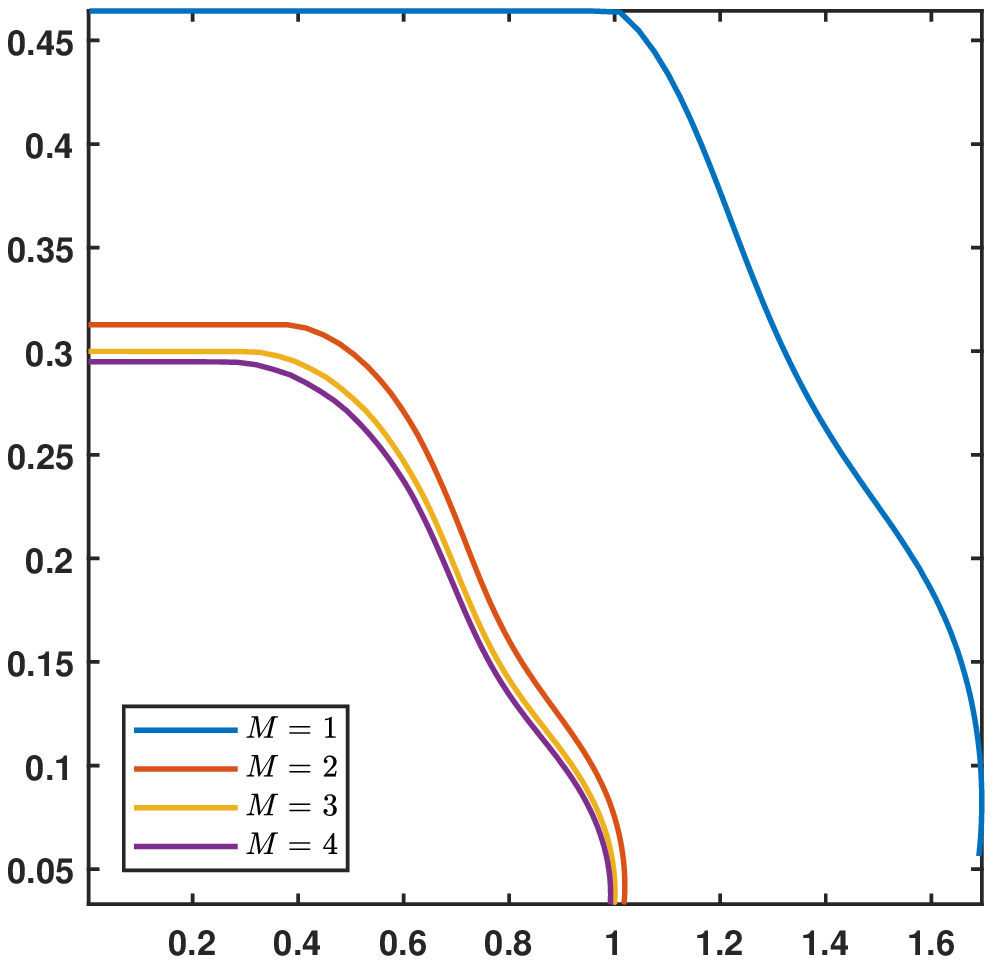}
\begin{picture}(0,0)
        \put(-455,100){\rotatebox{90} {$r_2$}}
				\put(-342,-6){$r_1$}
        \put(-85,-6){$r_1$}
        \put(-345,190){\footnotesize FE}
        \put(-100,190){\footnotesize ERK4}
\end{picture}
\caption{$r_2$ vs $r_1$ for forward Euler and ERK4 method. First and second derivatives approximated using \ofdC{\npt}{\npt}{\npt}{\npt}{4}, $\gamma(\kdx) = 1$ for $\kdx\in[0, 3]$, and $\gamma(\kdx) = 0$ otherwise.}
\figlabel{r1VSr2}
\end{center}
\end{figure}

Generally, $\dx$ and $\dt$ used for the complete discretization are expressed as a normalized quantity called Courant-Friedrichs-Lewy (CFL) number.
The CFL number with respect to $d^\text{th}$ derivative is defined as $r_d:=|\beta_d| \dt / \dx^d$.
By varying $\dx$ and calculating corresponding maximum $\dt$, we can calculate and plot CFL numbers $r_1$ and $r_2$ for the advection-diffusion equation (see \eqn{pdeAdvDiff}).
A similar plot obtained for optimal schemes of different stencil sizes for forward Euler and ERK4 method is shown in \fig{r1VSr2}. $\npt=1$ (blue) corresponds to the standard scheme \sfdC{1}{1}{1}{1}{4}, and $\npt=2$ onwards correspond to optimized schemes, \ofdC{\npt}{\npt}{\npt}{\npt}{4}. There is a drastic reduction in stability region from $\npt=1$ to $\npt=2$.
As $\npt$ is increased further, stability region decreases.

\begin{figure}[ht!]
\begin{center}
\includegraphics[width=0.5\textwidth]{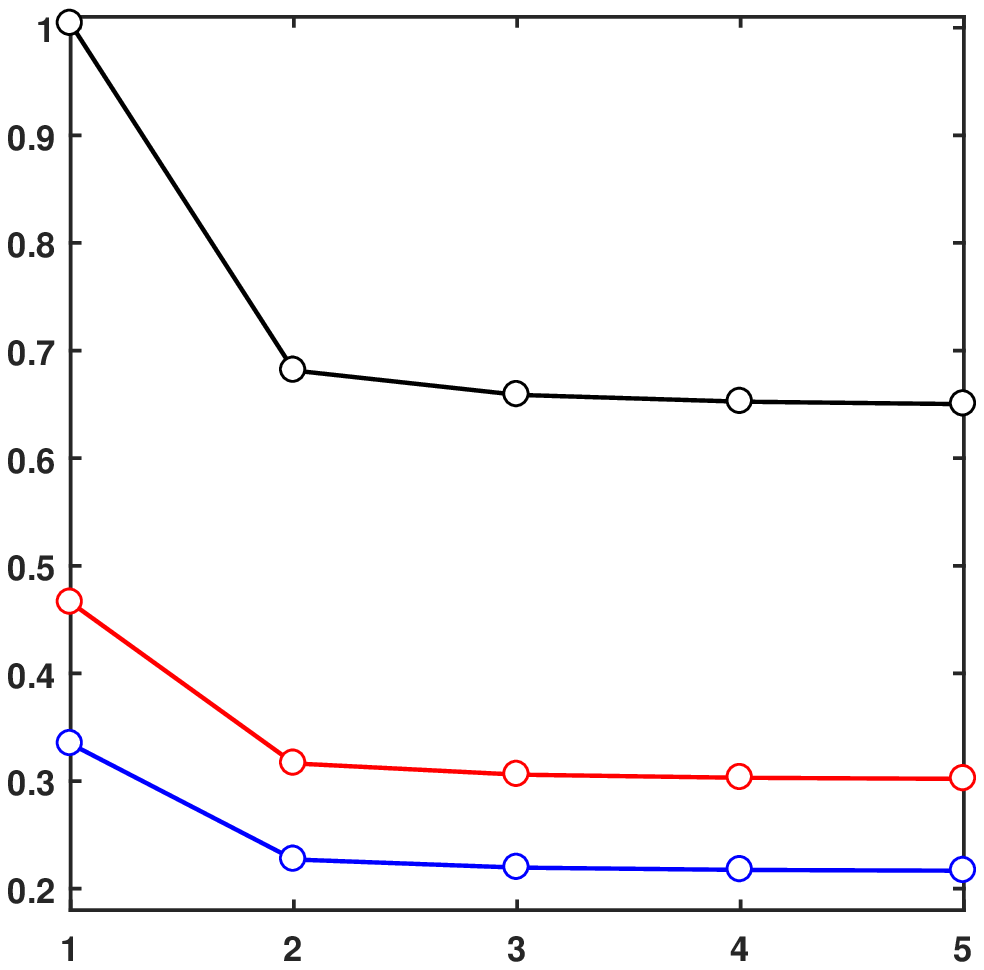}
\begin{picture}(0,0)
        \put(-252,100){\rotatebox{90} {$r_{2}$}}
				\put(-120,-10){$M$}
\end{picture}
\caption{Variation of maximum allowable $r_{2}$ with different stencil size, for different temporal schemes: FE (blue), ERK4 (red), and IRK2 (black). First and second derivatives approximated using \ofdC{\npt}{\npt}{\npt}{\npt}{4}.}
\figlabel{dtRKCompare}
\end{center}
\end{figure}

\begin{figure}[ht!]
\begin{center}
\includegraphics[width=\textwidth]{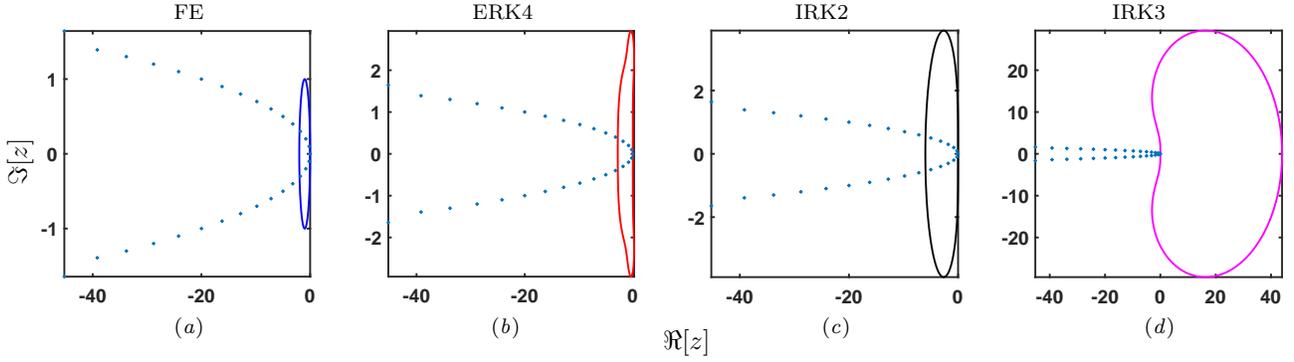}
\begin{picture}(0,0)
        \put(-248,57){\rotatebox{90} {$\Im[z]$}}
				\put(0,-5){$\Re[z]$}
				\put(-185,1){{\footnotesize (\textit{a})}}
				\put(-65,1){{\footnotesize (\textit{b})}}
				\put(60,1){{\footnotesize (\textit{c})}}
				\put(182,1){{\footnotesize (\textit{d})}}
				\put(-185,120){\footnotesize FE}
				\put(-72,120){\footnotesize ERK4}
				\put(50,120){\footnotesize IRK2}
				\put(170,120){\footnotesize IRK3}
\end{picture}
\caption{Closed curve shows stability regions for different RK schemes. Asterisks are eigenvalues of $\semiDiscSysMat$ for $\npt = 4$, $\beta = [-0.1, 0.2]$, $N_p = 31$.
(\textit{a}), (\textit{b}) and (\textit{c}): Inside the curve is stable.  (\textit{d}): Outside the curve is stable.}
\figlabel{rkStabRegion}
\end{center}
\end{figure}

Variation of maximum allowable $r_2$ with stencil size for different temporal schemes is shown in \fig{dtRKCompare}. Butcher tableaux for the RK schemes are given in Appendix \ref{sec:appButcherTab}.
We observe that, for all RK schemes in \fig{dtRKCompare}, there is a drastic drop in maximum allowable $\dt$ (or $r_2$) from $\npt=1$ to $\npt=2$. This is a consequence of drastic reduction in stability region from $\npt=1$ to $\npt=2$, similar to one shown in \fig{r1VSr2}.
Also, maximum allowable $r_2$ for IRK2 (black) is the largest followed by ERK4 (red) and FE (blue). This is due to the fact that IRK2 has the largest, and FE has the smallest stability region in the complex plane, as shown in \fig{rkStabRegion}$(a),(b),(c)$. For the same problem specifications, IRK3 is stable for all $0<\dt<\infty$, which is evident from \fig{rkStabRegion}$(d)$, as all eigenvalues, $\lambda_i$, lie in the stability region of IRK3 (outside the closed curve).

We next consider the relationship between $\lambda _ i$ and the derivatives in a PDE. For even derivatives from Lemma \ref{lem:evenImgZero}, both $\vo{a}_d$ and $\vo{b}_d$ are symmetric about the central element.
Consequently, both $\vo{A}_d^\vo{\Phi}$ and $\vo{B}_d^\vo{\Phi}$ defined in \eqn{AdBd} are real symmetric circulant matrices. Therefore, if a PDE consists of only even derivatives, then $\semiDiscSysMat$ defined in \eqn{semiDiscODE} is also symmetric and hence all $\lambda _ i$ are real.
Real parts of $\lambda _ i$ cause only amplitude decay (diffusion) or amplification but phase of the solution remains constant.
On the other hand, for odd derivatives, $\vo{a}_d$ is skew-symmetric and $\vo{b}_d$ is symmetric about the central element. Therefore, circulant matrices $\vo{A}_d^\vo{\Phi}$ and $\vo{B}_d^\vo{\Phi}$ become skew-symmetric and symmetric respectively, which in turn make $\semiDiscSysMat$ a skew-symmetric matrix.
Consequently, if a PDE consists of only odd derivatives, then all $\lambda _ i$ are pure imaginary which affect only phase of the solution preserving magnitude of amplitudes, leading to pure dispersion.
It is clear that even derivatives contribute to the real part of $\lambda _ i$  and hence diffusion of the numerical solution, and odd derivatives cause dispersion by means of the imaginary part of $\lambda _ i$.
However, this is true only for central schemes which admit results from Lemma \ref{lem:evenImgZero} and \ref{lem:oddRealZero}. It is shown in the subsequent section that in the case of biased schemes, both even and odd derivatives individually contribute to real and imaginary components of $\lambda _ i$.

In what follows next, we extend the stability analysis discussed in this section, and framework presented in \sect{implicitFramework} to derive and analyze special types of optimized schemes.

\section{Special cases of the framework} \label{specialCases}
We generalize \eqn{pade1} further as
\begin{align}
\sum_{m = -\npt^B_L} ^{\npt^B_R} \mathrm{b}_m f^{(d)}_{i+m}	= \frac{1}{(\dx)^d}\sum_{m = -\npt^A_L} ^{\npt^A_R} \mathrm{a}_m f_{i+m},
\eqnlabel{padeMostGen}
\end{align}
where, each one of $\npt^B_L,\npt^B_R,\npt^A_L$ and $\npt^A_R$ can be chosen independently.
 In \sect{implicitFramework}, we have presented a general framework to derive optimal implicit finite difference approximations.
These approximations are obtained for equal LHS and RHS stencil sizes, i.e. $\npt^B_L=\npt^B_R=\npt^A_L=\npt^A_R = \npt$.
In this section, we show that other types of schemes, namely, compact schemes with unequal LHS and RHS stencil sizes, spatially explicit, and biased finite difference approximations can be derived as special cases of the framework presented in \sect{implicitFramework}.
These special cases are derived by imposing additional constraints on the coefficients $\vo{a_d}$ and $\vo{b_d}$ while solving the minimization problem defined in \eqn{evenOptim}.

In the most generalized form, we denote $(p+1)^{\text{th}}$ order accurate optimized schemes derived using \eqn{padeMostGen} as \ofd{\npt^A_L}{\npt^A_R}{\npt^B_L}{\npt^B_R}{p+1}. For example, $5^\text{th}$ order accurate optimized scheme obtained using $\npt^A_L=1,\npt^A_R=2, \npt^B_L=3$ and $\npt^B_R=4$ is denoted by \ofd{1}{2}{3}{4}{5}.
See \fig{genStencil} as an illustration of this. The $i^\text{th}$ grid point at which derivative is to be approximated is shown by a green asterisk, and neighboring grid points are shown by black circles. Blue and red braces respectively show stencil points used in RHS and LHS of \eqn{padeMostGen}.
\begin{figure}[ht!]
\begin{center}
\includegraphics[width=0.8\textwidth]{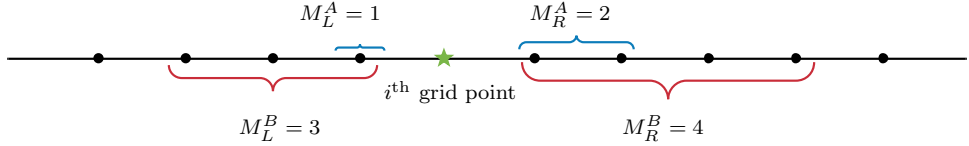}
\begin{picture}(0,0)
        \put(-262,55){\footnotesize $\npt^A_L=1$}
        \put(-175,55){\footnotesize $\npt^A_R=2$}
        \put(-285,12){\footnotesize $\npt^B_L=3$}
				\put(-140,12){\footnotesize $\npt^B_R=4$}
        \put(-230,25){\footnotesize $i^\text{th}$ grid point}
\end{picture}
\caption{Illustration for the notation of optimized schemes in the most generalized form.}
\figlabel{genStencil}
\end{center}
\end{figure}

\subsection{Central compact schemes with unequal LHS and RHS stencil sizes}
In \eqn{pade1}, indices of both $\mathrm{a}_m$ and $\mathrm{b}_m$ vary from $-\npt$ to $\npt$. We can also have different ranges of index $m$ for $\mathrm{a}_m$ and $\mathrm{b}_m$.
Let $\npt^B_L=\npt^B_R=\npt^B$ and $\npt^A_L=\npt^A_R=\npt^A$. i.e., \eqn{padeMostGen} becomes,
\begin{align}
\sum_{m = -\npt ^B} ^{\npt ^B} \mathrm{b}_m f^{(d)}_{i+m}	= \frac{1}{(\dx)^d}\sum_{m = -\npt ^A} ^{\npt ^A} \mathrm{a}_m f_{i+m},
\eqnlabel{padeNonSymm}
\end{align}
where $\npt^B\neq\npt^A$.

Continuing the use of the notation defined in \sect{implicitFramework}, we refer optimized schemes derived using \eqn{padeNonSymm} as \ofdC{\npt^A}{\npt^A}{\npt^B}{\npt^B}{\cdot}.
On the other hand, if a scheme is derived by the \textit{standard} method of coefficient matching without optimization, then it is referred by \sfdC{\npt^A}{\npt^A}{\npt^B}{\npt^B}{\cdot}.
When \eqn{penta7stencil} is compared with \eqn{padeNonSymm}, we observe that $\npt^A=3$ and $\npt^B=2$ for \eqn{penta7stencil}. Therefore, pentadiagonal optimized schemes are denoted by \ofdC{3}{3}{2}{2}{\cdot}, and pentadiagonal standard schemes are denoted by \sfdC{3}{3}{2}{2}{\cdot}.

The stencil size for the right side of \eqn{padeNonSymm} is defined as $N^A :=2\npt ^A+1$, and similarly the stencil size for the left side is $N^B :=2\npt ^B+1$.
Let us define \textit{augmented} parameters as $\nptAug:= \max \{\npt ^A,\npt ^B \}$ and $\NAug:= 2\nptAug + 1$.

First, we consider the case $\npt ^A > \npt ^B$, and hence $\nptAug = \npt ^A$. This \ofdC{\npt^A}{\npt^A}{\npt^B}{\npt^B}{\cdot} scheme can be treated as an \ofdC{\npt^A}{\npt^A}{\npt^A}{\npt^A}{\cdot} scheme with augmented stencil size $\NAug=N^A = 2\npt ^A + 1$,
and an extra imposed constraint that $\mathrm{b}_m = 0$ for $m=\{-\npt ^A,-\npt ^A+1,\cdots,-(\npt ^B+1),\npt ^B+1,\npt ^B+2,\cdots,\npt ^A \}$. In other words, we are imposing a constraint that $(N^A-N^B)$ elements of $\vo{b}_d$ which are located symmetrically about the central element, to be zero.
Therefore, we can use the central difference framework presented in \sect{implicitFramework} with $\npt = \npt ^A$ and the following additional constraint

\begin{align}
 \underbrace{\begin{bmatrix}
 \bdel_{n}^T(1) \\ \bdel_{n}^T(2) \\ \vdots \\ \bdel_{n}^T(\npt ^A-\npt ^B) \\ \bdel_{n}^T(\npt ^A+\npt ^B+2) \\ \bdel_{n}^T(\npt ^A+\npt ^B+3) \\ \vdots \\ \bdel_{n}^T(\NAug)
\end{bmatrix}}_{\vo{\Delta}^T:=} \vo{b}_d = \vo{0}_{(\NAug-N^B)\times 1},
\eqnlabel{specialNonSymmCon0}
\end{align}
for $n=\NAug$, where $\bdel_n^T(\cdot)$ is defined in \eqn{deltavec}. Therefore, the constraint can be written as
\begin{align}
\begin{bmatrix} \vo{0}_{(\NAug - N ^B)\times \NAug}  & \vo{\Delta} ^T \end{bmatrix}
\begin{bmatrix} \vo{a}_d \\ \vo{b}_d \end{bmatrix}  = \vo{0}_{(\NAug - N ^B) \times 1}.
\eqnlabel{specialNonSymmCon}
\end{align}
The minimization problem defined in \eqn{evenOptim} is solved analytically using the KKT condition constructed by incorporating the additional constraint \eqn{specialNonSymmCon}.
Then the optimal solution similar to \eqn{padeEvenOptimal} is given by
\begin{align}
\begin{pmatrix}\vo{a}_d^\ast\\\vo{b}_d^\ast\\\boldsymbol{\lambda}_d^\ast\end{pmatrix} = \vo{P}
\begin{bmatrix} \vo{0}_{2\NAug\times 1} \\ \vo{0}_{(d+p+1)\times 1} \\ 1 \\ \vo{0}_{(\NAug - N ^B) \times 1} \end{bmatrix},
\eqnlabel{specialNonSymmOptimal}
\end{align}
where,
\begin{align*} \vo{P} =
\left[\begin{array}{cc}  \vo{Q}_d &  \left(\begin{array}{cc}\vo{X}_d^T & -\vo{Y}_d^T\\
\vo{0}_{1 \times \NAug} & \bdel_{\NAug}^T(\nptAug+1) \\
\vo{0}_{(\NAug - N ^B) \times \NAug} & \vo{\Delta}^T \end{array}\right)^T\\
 \left(\begin{array}{cc}\vo{X}_d^T & -\vo{Y}_d^T\\
\vo{0}_{1 \times \NAug} & \bdel_{\NAug}^T(\nptAug+1) \\
\vo{0}_{(\NAug - N ^B) \times \NAug} & \vo{\Delta}^T \end{array}\right) & \vo{0}_{(\NAug - N ^B+d+p+2) \times (\NAug - N^B +d+p+2)}
\end{array}\right]^{-1}.
\end{align*}
In a similar manner, we can calculate optimal coefficients for the complementary case $\npt ^A < \npt ^B$ by setting $(N^B-N^A)$ elements of $\vo{a}_d$ to zero.

Note that by setting $M^A=3$ and $M^B=2$ in current framework, we get pentadiagonal compact schemes such as presented in \cite{lele1992compact} and \cite{kim1996opt}.
Latter frameworks restrict the stencil size and require that $\npt ^A > \npt ^B$.
On the other hand, the framework presented in this paper is more general because  $\npt ^A$ and $\npt ^B$ can be chosen independently.

Also note that, the constraint \eqn{specialNonSymmCon0} sets $\mathrm{b}_m = 0$ symmetrically about the central element of $\vo{b}_d$.
Therefore, Lemmas \ref{lem:evenImgZero} and \ref{lem:oddRealZero}, and stability analysis presented in Section \ref{stability} are valid for this special case.

\subsection{Central spatially explicit schemes} \label{sec:explicit}
A finite difference approximation is said to be spatially explicit if coefficients $b_m$ in \eqn{pade1} are
\begin{align*}
b_m  = \left\{\begin{array}{c} 1 \text{ if } m = 0,\\ 0 \text{ if } m \neq 0.\end{array}\right.
\end{align*}
i.e.,
\begin{align}
  f^{(d)}_{i}	= \frac{1}{(\dx)^d}\sum_{m = -\npt ^A} ^{\npt ^A} \mathrm{a}_m f_{i+m}.
\eqnlabel{padeExplicit}
\end{align}
It is obvious that central approximations which are spatially explicit can be obtained by substituting $\npt^B=0$ in \eqn{padeNonSymm}. Therefore, an optimized explicit central schemes is represented by \ofdC{\npt^A}{\npt^A}{0}{0}{\cdot}.
Constraints \eqn{specialNonSymmCon0} and \eqn{nonSingular} can be combined together as
\begin{align}
 \underbrace{\begin{bmatrix}
 \bdel_n^T(1) \\ \bdel_n^T(2) \\ \vdots \\ \bdel_n^T(\npt^A+1) \\ \vdots \\ \bdel_n^T(\NAug-1) \\ \bdel_n^T(\NAug)
\end{bmatrix}}_{\vo{\Delta}^T:=} \vo{b}_d = \begin{bmatrix}
													0 \\0 \\ \vdots \\ 1 \\ \vdots \\ 0 \\ 0
													\end{bmatrix} =  \bdel_n(\npt^A+1),
\eqnlabel{specialExplicitCon0}
\end{align}
for $n=\NAug$, where $\NAug=2\npt^A+1$. Since $\vo{\Delta}^T$ in the above equation is an identity matrix, the constraint reduces to
\begin{align}
\vo{b}_d  =  \bdel_n (\npt^A+1).  % \begin{bmatrix} 0 & 0 & \cdots & 1 & \cdots & 0 & 0 \end{bmatrix}
\eqnlabel{specialExplicitCon}
\end{align}
 Therefore, optimal solution can be found analytically as
 \begin{align}
 \begin{pmatrix}\vo{a}_d^\ast\\\vo{b}_d^\ast\\\boldsymbol{\lambda}_d^\ast\end{pmatrix} = \left[\begin{array}{cc}
 \vo{Q}_d & \left(\begin{array}{cc}\vo{X}_d^T & -\vo{Y}_d^T\\
 \vo{0}_{\NAug \times \NAug} & \vo{I}_{\NAug \times \NAug}\end{array}\right)^T\\
  \left(\begin{array}{cc}\vo{X}_d^T & -\vo{Y}_d^T\\
 \vo{0}_{\NAug \times \NAug} & \vo{I}_{\NAug \times \NAug}\end{array}\right) & \vo{0}_{(\NAug+d+p+1)\times(\NAug+d+p+1)}
 \end{array}\right]^{-1}\begin{bmatrix}
 \vo{0}_{2\NAug\times 1}\\\vo{0}_{(d+p+1)\times 1} \\ \bdel_{\NAug}(\npt^A+1)
 \end{bmatrix}.
 \eqnlabel{specialExplicitOptimal}
 \end{align}
 The only difference between spatially implicit and explicit schemes is the structure of matrix $\vo{B}_d^\vo{\Phi}$. Since in the case of explicit schemes $\vo{b}_d$ is given by \eqn{specialExplicitCon0}, $\vo{B}_d^\vo{\Phi}$ becomes an identity matrix.
 Therefore stability analysis discussed in \sect{stability} holds true for explicit schemes as well with the substitution $\vo{B}_d^\vo{\Phi} = \vo{I}$ in \eqn{genlDisc}.

A similar framework exclusively for spatially explicit schemes has been presented in \cite{kumari2018unified} and optimal coefficients calculated using this framework match exactly with the ones obtained using \eqn{specialExplicitOptimal}.
In \cite{kumari2018unified}, authors extensively discuss error characteristics and stability analysis for optimal explicit schemes.

\begin{figure}[ht!]
\begin{center}
\includegraphics[width=0.6\textwidth]{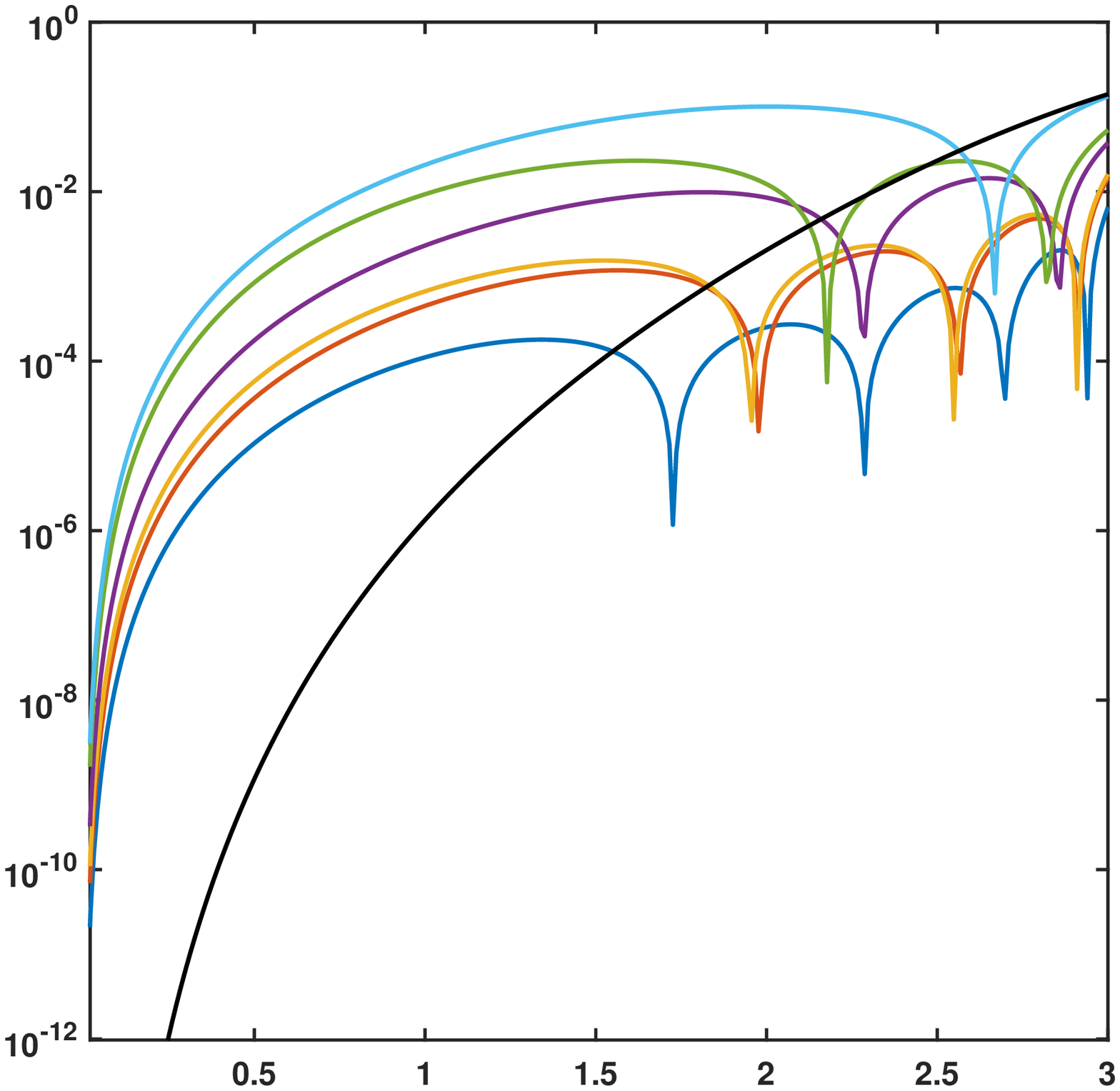}
\begin{picture}(0,0)
        \put(-280,120){\rotatebox{90} {$\Big|\frac{\mkdx^2}{\kdx^2} -1\Big|$}}
				\put(-140,0){$\kdx$}
\end{picture}
\caption{Error in modified wavenumber for various \ofdC{\npt^A}{\npt^A}{\npt^B}{\npt^B}{4} schemes approximating second derivative. Bottom to top: Black -- \sfdC{3}{3}{2}{2}{10}; blue -- \ofdC{3}{3}{3}{3}{4}; red -- \ofdC{3}{3}{2}{2}{4}; yellow -- \ofdC{2}{2}{3}{3}{4}; purple -- \ofdC{3}{3}{1}{1}{4}; green -- \ofdC{1}{1}{3}{3}{4}; cyan -- \ofdC{3}{3}{0}{0}{4}. $\NAug = 7$, $\gamma(\kdx) = 1$ for $\kdx\in[0, 3]$, and $\gamma(\kdx) = 0$ otherwise.}
\figlabel{M3Compare}
\end{center}
\end{figure}

A comparison of $\mkdx^2$ with $\kdx^2$ for optimal schemes derived using different values of $\npt^A$ and $\npt^B$ for fixed augmented stencil size $\NAug=7$ is shown in \fig{M3Compare}. It is clear that optimized scheme with equal LHS and RHS stencil sizes ($\npt^A = \npt^B$), i.e. \ofdC{3}{3}{3}{3}{4} (blue), performs better than all other optimized schemes with unequal LHS and RHS stencil sizes ($\npt^A \neq \npt^B$).
This is expected, since \ofdC{3}{3}{3}{3}{4} has the largest number of degrees of freedom in the optimization. We state this observation as the following remark.
\begin{remark} \label{rem:symmBest}
Among central compact optimized schemes which have same augmented stencil size, the scheme with $\npt^A =\npt^B$ demonstrates the best spectral accuracy.
\end{remark}

Note that, $\npt^A >\npt^B$ ($\npt^A = 3$, $\npt^B = 2$) for \ofdC{3}{3}{2}{2}{4} (red), and $\npt^A <\npt^B$ ($\npt^A = 2$, $\npt^B = 3$) for \ofdC{2}{2}{3}{3}{4} (yellow). These both schemes have same number of degrees of freedom, i.e., total sum of LHS and RHS stencil sizes is same for both of them. However, from \fig{M3Compare}, it can be seen that \ofdC{3}{3}{2}{2}{4} (red) performs slightly better than \ofdC{2}{2}{3}{3}{4} (yellow).
Similarly, $\npt^A >\npt^B$ for \ofdC{3}{3}{1}{1}{4} (purple), and $\npt^A <\npt^B$ for \ofdC{1}{1}{3}{3}{4} (green). Clearly, \ofdC{3}{3}{1}{1}{4} (purple) performs better than \ofdC{1}{1}{3}{3}{4} (green). Based on these observations, we infer the following.
\begin{remark}
Let $\npt_1>\npt_2$. Central compact optimized schemes with $\npt^A=M_1$ and $\npt^B=M_2$ i.e., $\npt^A>\npt^B$, demonstrate better spectral accuracy than those with $\npt^A=M_2$ and $\npt^B=M_1$ i.e., $\npt^A<\npt^B$.
\end{remark}

\iffalse
\blah{ Under investigation: For $M\geq7$, $\vo{b}_d^{\ast} <0$. Ill conditioned matrix in KKT condition. Explicit schemes do not face this problem.
\begin{remark}
A correlation between the signs of optimal coefficients $\vo{b}_d^{\ast}$, and the inequality relation of $\npt^A$ and $\npt^B$ is observed empirically.
The optimal coefficients $\vo{b}_d^{\ast}$ satisfy $\vo{b}_d^{\ast} \geq 0$ for central compact optimized schemes with $\npt^A \geq \npt^B$. Some elements of $\vo{b}_d^{\ast}$ are negative for schemes with $\npt^A < \npt^B$.
\end{remark}
}
\fi

Finally, we observe that the central optimized explicit scheme \ofdC{3}{3}{0}{0}{4} (cyan) performs worse than all other central compact optimized schemes with equal $\NAug$. Again, this is expected, since the explicit scheme has the minimum number of degrees of freedom in the optimization.

For the standard pentadiagonal scheme \sfdC{3}{3}{2}{2}{10} (black), we see that the error is very small in the low wavenumber region. However, the error is large at higher wavenumbers. This leads to steep variation of error in the wavenumber space. On the other hand, the optimized pentadiagonal scheme \ofdC{3}{3}{2}{2}{4} (red) demonstrates relatively \textit{spectrally flat} behavior due to the weighting function $\wtfun=1$, which weights all wavenumbers equally.

\subsection{Biased schemes}
In the discussion so far, we have assumed the domain to be periodic, which enables us to use central schemes at all grid points in the domain.
Central schemes can not be used if the domain is not periodic or there are specified boundary conditions.
This is because of the fact that there are unequal number of usable grid points on either side of a point which lies near boundary of the domain.
Thus, we have to use \textit{one-sided} or \textit{biased} schemes near boundaries.
A finite difference approximation is said to be biased if the number of grid points used on the left and right side of a point at which derivative is calculated is not equal.
For the sake of discussion, in this special case we assume that $\npt^B_L=\npt^A_L=\npt_L$ and $\npt^B_R=\npt^A_R=\npt_R$.
Therefore, \eqn{padeMostGen} becomes
\begin{align}
\sum_{m = -\npt _L} ^{\npt _R} \mathrm{b}_m f^{(d)}_{i+m}	= \frac{1}{(\dx)^d}\sum_{m = -\npt _L} ^{\npt _R} \mathrm{a}_m f_{i+m},
\eqnlabel{padeNonCentral}
\end{align}
Thus, actual or true stencil size is $\NAc:= \npt _L + \npt _R + 1$. The augmented parameters are $\nptAug:= \max \{\npt _L,\npt _R \}$ and $\NAug:= 2\nptAug + 1$.
Note that, in special cases, biased scheme is \textit{backward} finite difference if $\npt _L \neq 0$ and $\npt _R = 0$, and \textit{forward} finite difference if $\npt _L = 0$ and $\npt _R \neq 0$.

Let $\npt _L > \npt _R$, and hence $\nptAug = \npt _L$. We refer this type of approximation as \textit{left-biased}. This biased scheme can be treated as a central scheme with augmented stencil size $\NAug = 2\npt _L + 1$,
and an extra imposed constraint that rightmost  $(\NAug - \NAc)$ coefficients to be zero. Therefore, we can use the central difference framework presented in \sect{implicitFramework} with $\npt = \npt _L$
and the following additional constraint
\begin{align}
\underbrace{\begin{bmatrix}
 \bdel_n^T(\NAc+1) \\ \bdel_n^T(\NAc+2) \\ \vdots \\ \bdel_n^T(\NAug)
\end{bmatrix}}_{\vo{\Delta}^T:=} \vo{a}_d = \underbrace{\begin{bmatrix}
 \bdel_n^T(\NAc+1) \\ \bdel_n^T(\NAc+2) \\ \vdots \\ \bdel_n^T(\NAug)
\end{bmatrix}}_{\vo{\Delta}^T:=} \vo{b}_d = \begin{bmatrix}
													0 \\0 \\ \vdots \\ 0
													\end{bmatrix},
\eqnlabel{specialNonCentralCon0}
\end{align}
for $n=\NAug$. Therefore, the constraint can be written as
\begin{align}
\begin{bmatrix} \vo{\Delta}^T & \vo{0}_{ (\NAug - \NAc) \times \NAug } \\ \vo{0}_{ (\NAug - \NAc) \times \NAug } & \vo{\Delta}^T \end{bmatrix}
\begin{bmatrix} \vo{a}_d \\ \vo{b}_d \end{bmatrix}  = \vo{0}_{2(\NAug - \NAc) \times 1}.
\eqnlabel{specialNonCentralCon}
\end{align}
Now, we can write the analytical optimal solution for \eqn{evenOptim} subject to \eqn{specialNonCentralCon} as follows
\begin{align}
\begin{pmatrix}\vo{a}_d^\ast\\\vo{b}_d^\ast\\\boldsymbol{\lambda}_d^\ast\end{pmatrix} = \vo{P}
\begin{bmatrix} \vo{0}_{2\NAug\times 1} \\ \vo{0}_{(d+p+1)\times 1} \\ 1 \\ \vo{0}_{2(\NAug - \NAc) \times 1} \end{bmatrix},
\eqnlabel{specialNonCentralOptimal}
\end{align}
where,
\begin{align*} \vo{P} =
\left[\begin{array}{cc}  \vo{Q}_d &  \left(\begin{array}{cc}\vo{X}_d^T & -\vo{Y}_d^T\\
\vo{0}_{1 \times \NAug} & \bdel_{\NAug}^T(\nptAug+1) \\
\vo{\Delta}^T & \vo{0}_{(\NAug - \NAc) \times \NAug} \\
\vo{0}_{(\NAug - \NAc) \times \NAug} & \vo{\Delta}^T \end{array}\right)^T\\
 \left(\begin{array}{cc}\vo{X}_d^T & -\vo{Y}_d^T\\
\vo{0}_{1 \times \NAug} & \bdel_{\NAug}^T(\nptAug+1) \\
\vo{\Delta}^T & \vo{0}_{(\NAug - \NAc) \times \NAug} \\
\vo{0}_{(\NAug - \NAc) \times \NAug} & \vo{\Delta}^T \end{array}\right) & \vo{0}_{(2(\NAug - \NAc)+d+p+2) \times (2(\NAug - \NAc)+d+p+2)}
\end{array}\right]^{-1}.
\end{align*}

Similarly, we can determine optimal \textit{right-biased} approximation for the case $\npt _L < \npt _R$ by setting leftmost $(\NAug - \NAc)$ coefficients to zero.
It can be proved that optimal coefficients of left-biased scheme ($\vo{a}^\ast _{d_L}, \vo{b}^\ast _{d_L}$) with $\npt _L = \npt_1$,  $\npt _R = \npt_2$ (where, $\npt_1>\npt_2$)
and its complementary right-biased scheme ($\vo{a}^\ast _{d_R}, \vo{b}^\ast _{d_R}$) with $\npt _L = \npt_2$,  $\npt _R = \npt_1$ satisfy
\begin{align*}
\vo{a}^\ast _{d_R} = \vo{J} \vo{a}^\ast _{d_L}; \quad \vo{b}^\ast _{d_R} &= \vo{J} \vo{b}^\ast _{d_L}, \quad \text{for even derivatives}, \\
\vo{a}^\ast _{d_R} = - \vo{J} \vo{a}^\ast _{d_L}; \quad \vo{b}^\ast _{d_R} &= \vo{J} \vo{b}^\ast _{d_L}, \quad \text{for odd derivatives},
\end{align*}
where $\vo{J}$ is an anti-diagonal identity matrix of the appropriate dimension.

\begin{figure}[ht!]
\begin{center}
\includegraphics[width=\textwidth]{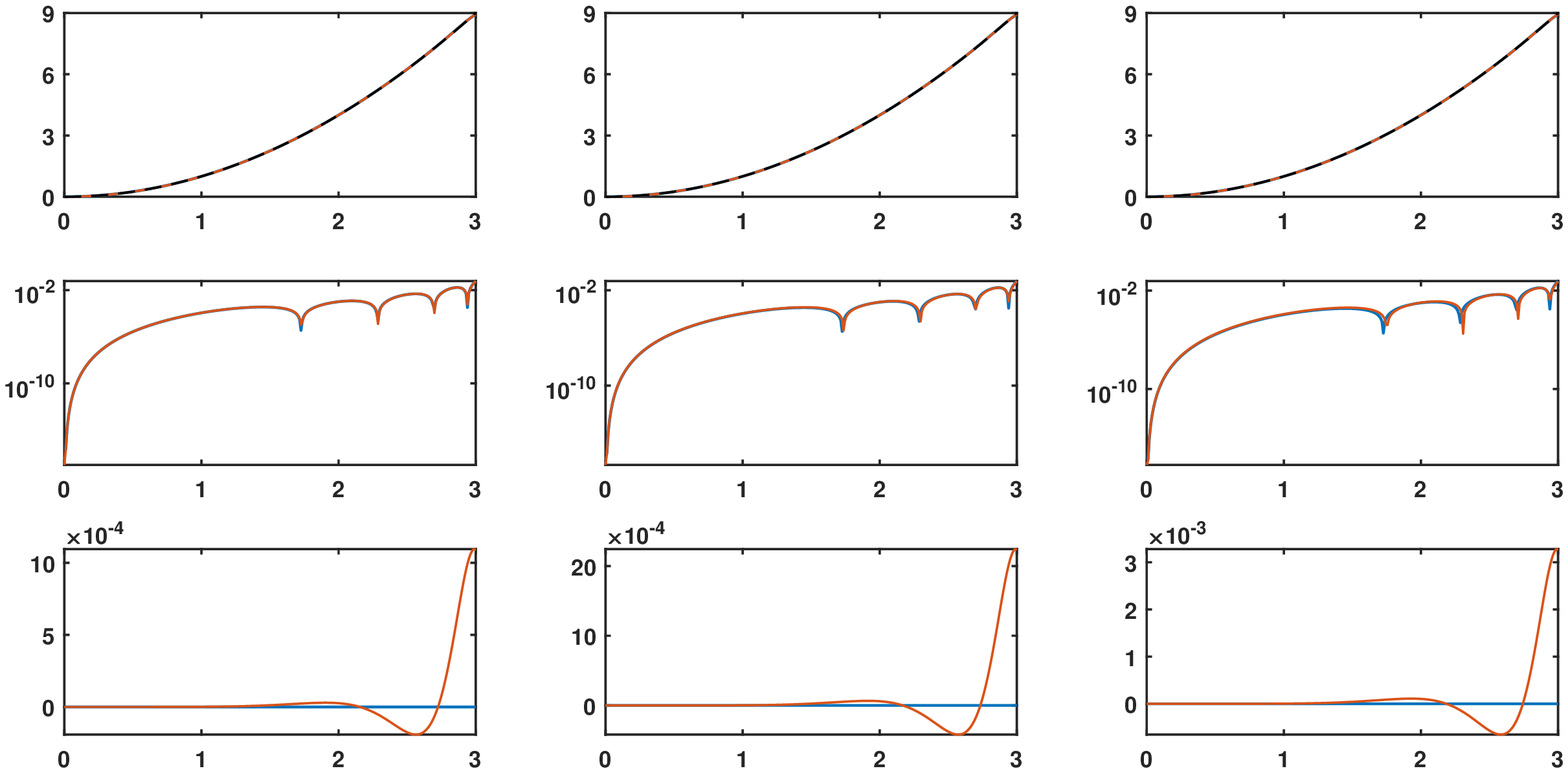}
\begin{picture}(0,0)
        \put(-250,210){\rotatebox{90} {$\mkdx ^2$}}
				\put(-250,115){\rotatebox{90} {$|\Re[e(\kdx)]|$}}
				\put(-250,35){\rotatebox{90} {$\Im[e(\kdx)]$}}
				\put(-165,5){$\kdx$}
				\put(0,5){$\kdx$}
				\put(165,5){$\kdx$}
				\put(-175,250){$\npt _L = 4$}
				\put(-10,250){$\npt _L= 5$}
				\put(153,250){$\npt _L= 6$}
\end{picture}
\caption{Modified wavenumber, and real and imaginary components of the spectral error for left-biased \ofd{\npt_L}{\npt_R}{\npt_L}{\npt_R}{4} approximating the second derivative. $\NAc = 7$, $\gamma(\kdx) = 1$ for $\kdx\in[0, 3]$, and $\gamma(\kdx) = 0$ otherwise.}
\figlabel{implicitspectralErrorNonPeriodD2}
\end{center}
\end{figure}

\begin{figure}[ht!]
\begin{center}
\includegraphics[width=\textwidth]{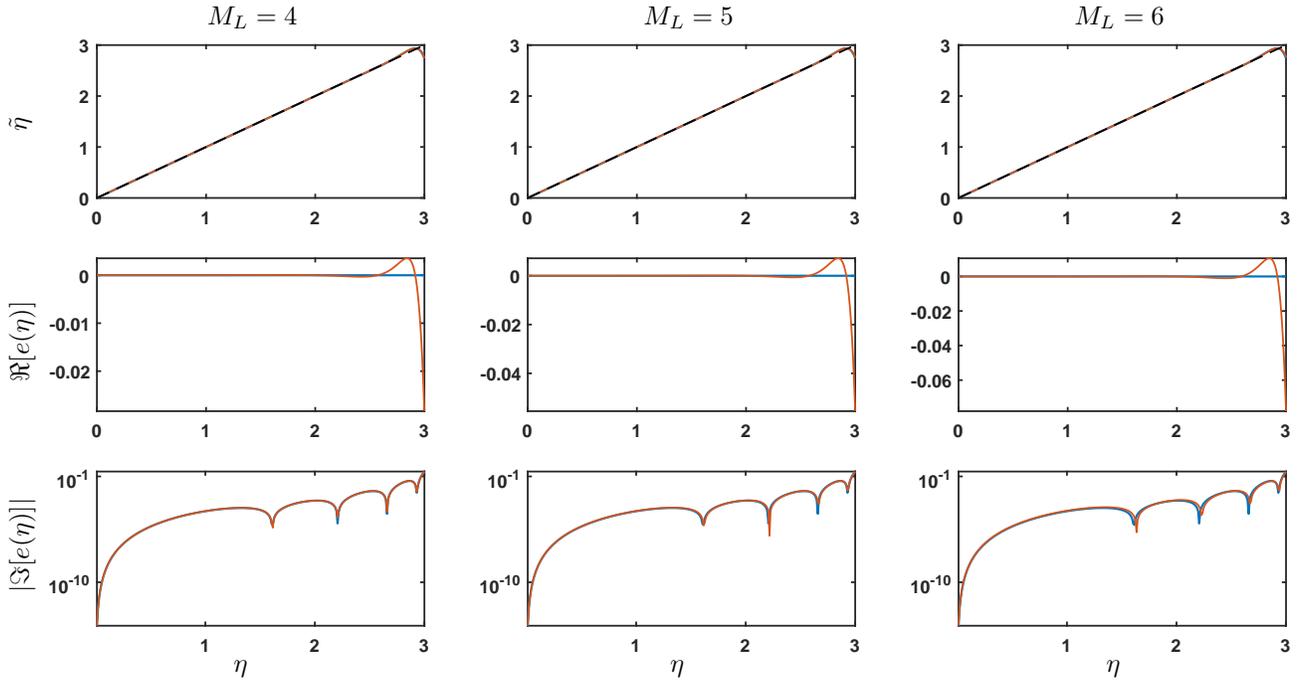}
\begin{picture}(0,0)
        \put(-250,210){\rotatebox{90} {$\mkdx$}}
				\put(-250,115){\rotatebox{90} {$\Re[e(\kdx)]$}}
				\put(-250,35){\rotatebox{90} {$|\Im[e(\kdx)]|$}}
				\put(-165,5){$\kdx$}
				\put(0,5){$\kdx$}
				\put(165,5){$\kdx$}
				\put(-175,250){$\npt _L = 4$}
				\put(-10,250){$\npt _L= 5$}
				\put(153,250){$\npt _L= 6$}
\end{picture}
\caption{Modified wavenumber, and real and imaginary components of the spectral error for left-biased \ofd{\npt_L}{\npt_R}{\npt_L}{\npt_R}{4} approximating the first derivative. $\NAc = 7$, $\gamma(\kdx) = 1$ for $\kdx\in[0, 3]$, and $\gamma(\kdx) = 0$ otherwise.}
\figlabel{implicitspectralErrorNonPeriodD1}
\end{center}
\end{figure}
\fig{implicitspectralErrorNonPeriodD2} and \fig{implicitspectralErrorNonPeriodD1} compare modified wavenumbers of different implicit left-biased schemes ($\npt _L= 4,5,6$, red) with the implicit central scheme ($\npt = 3$, blue) of equal true stencil size ($\NAc = 2\npt+1 = 7$) for second and first derivatives respectively.
Numerical values of the optimal coefficients are tabulated in Table \ref{table:stencilCoeffNonCentralD2} and  \ref{table:stencilCoeffNonCentralD1}.

From the middle row of \fig{implicitspectralErrorNonPeriodD2}, we observe that the variation of the real component $\Re[\err]$ (shown in log scale) for biased schemes (red) is very close to that of the central scheme of same stencil size (blue).

In case of central scheme (blue), the imaginary component $\Im[\err]$ for even derivatives is zero, which is seen in the last row of \fig{implicitspectralErrorNonPeriodD2} (see Lemma \ref{lem:evenImgZero} and \fig{implicitspectralError2}). And as expected, the imaginary component $\Im[\err]$ for biased schemes (red) which approximate even derivatives is not zero.
This is due to the fact that Lemma \ref{lem:evenImgZero} does not hold for biased schemes as constraint \eqn{specialNonCentralCon} imposes asymmetry on optimal coefficients. For fixed true stencil size, as $\npt_L$ increases, $\npt_R = \NAc-(\npt_L+1)$ decreases and schemes become more biased.
It can be observed from the last row of \fig{implicitspectralErrorNonPeriodD2} that the imaginary component $\Im[\err]$ increases with increase in biasedness of the scheme.

Similarly, the middle row of \fig{implicitspectralErrorNonPeriodD1} shows that unlike the central scheme (blue), real component $\Re[\err]$ for biased schemes (red) which approximate odd derivatives is not zero, and its magnitude increases with the increase in biasedness of the scheme.
However, the variation of imaginary component $\Im[\err]$ (shown in log scale) for biased schemes (red) is very close to that of the central scheme (blue).

While solving PDEs numerically, we use a central scheme of stencil size ($2\npt+1$) at ($N_p - 2\npt$) points within the domain and biased schemes at $\npt$ points adjacent to each one of left and right boundaries of the domain.
Therefore, ($N_p - 2\npt$) middle rows  of $\vo{A}_d$ and $\vo{B}_d$ defined in \eqn{defAdBd} contain same central scheme coefficients, last $M$ rows contain left-biased scheme coefficients and first $M$ rows contain complementary right-biased scheme coefficients.
Then $\vo{A}_d^\vo{\Phi}$ and $\vo{B}_d^\vo{\Phi}$ for such non-periodic discretization can be constructed similar to \eqn{AdBd}.
Semi-discrete stability discussed in \sect{semiDiscStability} is not valid for biased approximations since coefficients do not satisfy properties described in Lemma \ref{lem:evenImgZero} and \ref{lem:oddRealZero}.
However, \eqn{genlDisc} and \eqn{semiDiscODE} hold true for non-periodic discretization. Therefore, fully discretized stability analysis presented in \sect{fullyDiscStabilityRK} is valid for this case.

Special cases discussed in this section are a few examples which demonstrate how to derive optimal finite difference schemes with desired structure by imposing suitable constraints on coefficients.
Since all constraints are linear, one can impose any number and combination of constraints \eqn{specialNonSymmCon0}, \eqn{specialExplicitCon} and \eqn{specialNonCentralCon0} to obtain desirable schemes.
For example, we can impose constraints in \eqn{specialExplicitCon} and \eqn{specialNonCentralCon0} simultaneously and solve the resulting optimization problem to derive biased schemes which are spatially explicit.
However, to ensure the invertibility of the matrix $\vo{P}$, such as in \eqn{specialNonCentralOptimal}, duplicate constraints should be removed while constructing the KKT condition.

\section{Numerical Results} \label{numericalResults}
In this section we use optimized schemes derived in preceding sections to solve some PDEs numerically and compare the results with the standard schemes.
In particular, we benchmark our optimized schemes against following:
\begin{enumerate}
\item[(a)] Standard schemes of same order
\item[(b)] Standard schemes of same stencil size
\end{enumerate}

Initial condition for solving PDEs is taken to be the superposition of multiple sinusoidal waves with wavenumbers $k$, given by \eqn{initF}.
$A(k)$ is amplitude of the wave corresponding to wavenumber $k$, and $\phi_k \in [0,2\pi]$ is the corresponding phase initialized randomly.
We use the same initial condition for solving different PDEs in sections to follow.

\begin{align}
f_0(x):=f(x,t=0) = \sum _k A(k)\sin(kx + \phi _k).
\eqnlabel{initF}
\end{align}

\subsection{Advection-diffusion equation}
Consider the following simplified version of \eqn{pde}, where coefficients of only first two spatial derivatives are non-zero.
\begin{align}
\frac{\partial f}{\partial t} = \beta_1 \frac{\partial f}{ \partial x} + \beta_2 \frac{\partial ^2 f}{ \partial x^2}.
\eqnlabel{pdeAdvDiff}
\end{align}
To satisfy \eqn{semiDiscStabilityIneqGen}, we have diffusivity $\beta_2>0$. Analytical solution for \eqn{pdeAdvDiff} is known and given by
\begin{align}
f_t(x)_a := f(x,t>0)_a = \sum _k \exp({-\beta_2 k^2 t})A(k)\sin(k(x+\beta_1 t) + \phi _k),
\eqnlabel{advDiffSol}
\end{align}
where subscript $a$ stands for the analytical solution. In this case, amplitude decays with time due to non-zero $\beta_2$ and phase also changes due to  non-zero $\beta_1$.

We used CFL number to express $\dt$ and $\dx$ as a normalized quantity. Similarly, we define normalized time with respect to $d^{th}$ derivative as $t^\ast _d:= |\beta_d| t k_{\max}^d$, where $k_{\max}$ is the largest wavenumber present in the initial condition \eqn{initF}.
\eqn{pdeAdvDiff} is solved using various optimized and standard spatial finite difference schemes. For a fixed diffusive CFL ($r_2$), we have used matching order explicit Runge-Kutta (ERK) schemes for the temporal discretization. For example, for spatial schemes which are $4^{\text{th}}$ and $10^{\text{th}}$ order accurate, we have used $2^{\text{nd}}$ and $5^{\text{th}}$ order ERK for the temporal discretization respectively.

The spectral energy at wavenumber $k$ at time $t$ is defined to be $|\hat{f}_t(k)|^2$, where $\hat{f}_t(k)$ is the discrete Fourier coefficient corresponding to wavenumber $k$ at time $t$.
To assess the dissipative numerical errors at different wavenumbers, we plot normalized error in spectral energy of the numerical solution with respect to the analytical solution in \Fig{spectralErrAdvectionDiffusion}.
 These results are obtained for diffusive CFL number $r_2 = 0.01$ at $t^\ast_2 \approx 25$.

Error is small for the standard scheme, \sfdC{1}{1}{1}{1}{4} (cyan), in low wavenumber region. However, error grows drastically as wavenumber increases. The standard pentadiagonal scheme, \sfdC{3}{3}{2}{2}{10} (black), shows better spectral accuracy than \sfdC{1}{1}{1}{1}{4} across all wavenumbers. Although, error in high wavenumber region is significantly large for \sfdC{3}{3}{2}{2}{10} as well.

The optimized pentadiagonal scheme, \ofdC{3}{3}{2}{2}{4} (red), shows better accuracy than \sfdC{1}{1}{1}{1}{4} (cyan) at all wavenumbers. \ofdC{3}{3}{2}{2}{4} has larger error than \sfdC{3}{3}{2}{2}{10} (black) at lower wavenumbers. But at higher wavenumbers, the $4^\text{th}$ order accurate \ofdC{3}{3}{2}{2}{4} demonstrates much better accuracy than $10^\text{th}$ order accurate \sfdC{3}{3}{2}{2}{10}.

Overall, \ofdC{3}{3}{3}{3}{4} (blue) shows better accuracy than pentadiagonal \ofdC{3}{3}{2}{2}{4} (red), which is in accordance with Remark \ref{rem:symmBest}.

Both standard schemes, \sfdC{1}{1}{1}{1}{4} (cyan) and \sfdC{3}{3}{2}{2}{10} (black), demonstrate steep variation of error as wavenumber is increased. On the other hand, optimized schemes \ofdC{3}{3}{2}{2}{4} (red) and \ofdC{3}{3}{3}{3}{4} (blue) show relatively slower variation of error with increasing wavenumber. Again, this behavior is expected due to equally weighted wavenumbers ($\wtfun=1$) in the optimization.

\begin{figure}[ht!]
\begin{center}
\includegraphics[width=0.5\textwidth]{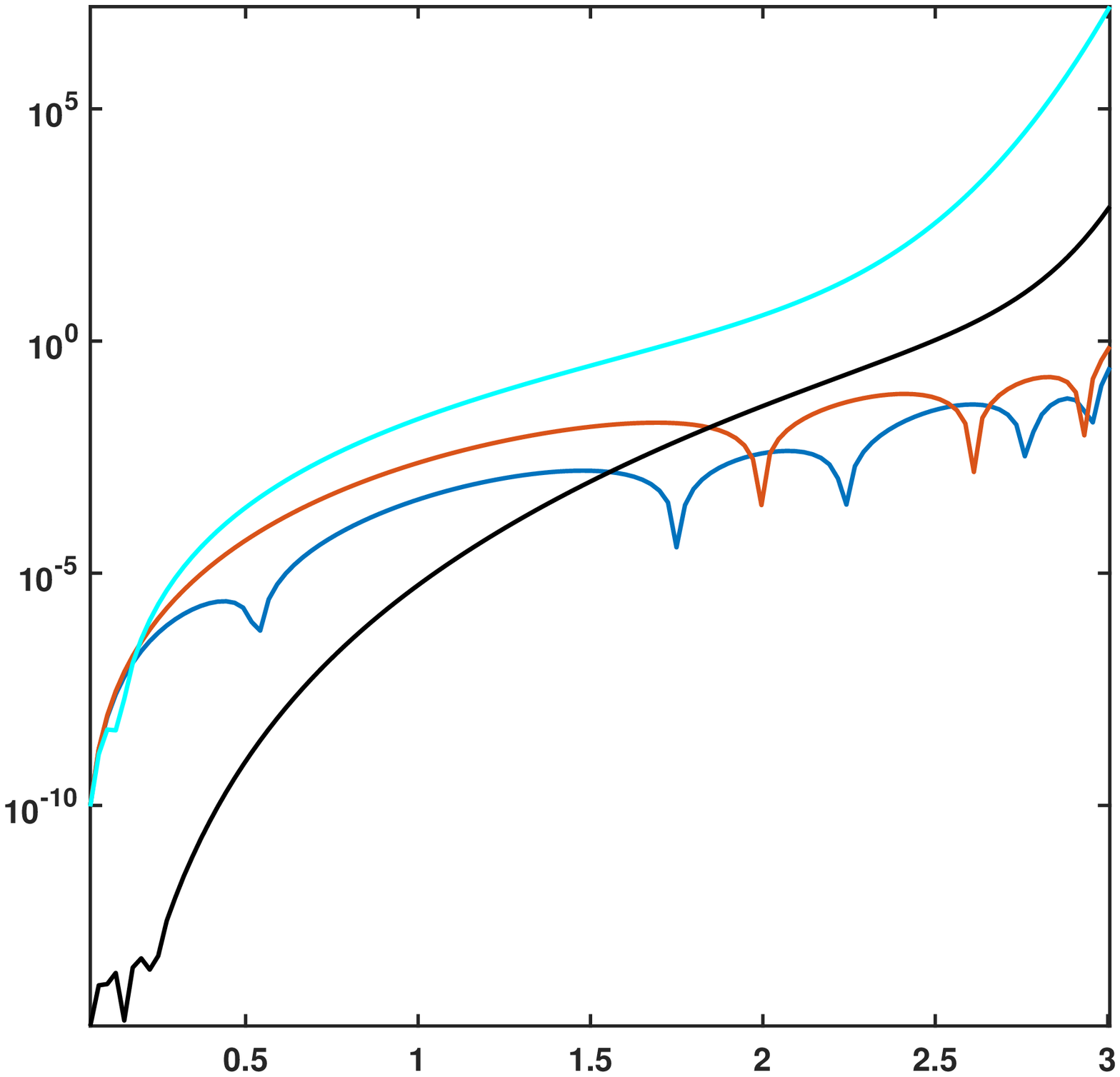}
\begin{picture}(0,0)
        \put(-260,100){\rotatebox{90} {$\Big| \Big|\frac{\hat{f}_t(k)}{\hat{f}_t(k)_a}\Big|^2 -1 \Big|$}}
				\put(-120,-10){$\kdx$}
\end{picture}
\caption{Diffusion error for advection-diffusion equation. Blue -- \ofdC{3}{3}{3}{3}{4}; red -- \ofdC{3}{3}{2}{2}{4}; black -- \sfdC{3}{3}{2}{2}{10}; cyan -- \sfdC{1}{1}{1}{1}{4}. $r_1\approx r_2 = 0.01$, $t^\ast_2 \approx 25$, $A(k) = 1$, $\gamma(\kdx) = 1$ for $\kdx\in[0, 3]$, and $\gamma(\kdx) = 0$ otherwise.}
\figlabel{spectralErrAdvectionDiffusion}
\end{center}
\end{figure}

While \Fig{spectralErrAdvectionDiffusion} shows numerical dissipation error, the dispersion error is shown in \fig{spectralErrAdvectionDiffusionSpeed}.
Dispersion error is quantified by calculation of numerical speed, $c_n(k):= \arg[\hat{f}_t(k)/\hat{f}_0(k)]/kt$. Analytical speed is $c_a := \beta_1$ for all wavenumbers.
Normalized numerical speed, $c^{\ast}(k):=c_n(k)/c_a$, is unity when there is no dispersion error. Normalized numerical speed and error for different schemes are shown in \fig{spectralErrAdvectionDiffusionSpeed}.

Observations made for \Fig{spectralErrAdvectionDiffusion} are consistent with \fig{spectralErrAdvectionDiffusionSpeed} as well. From  \fig{spectralErrAdvectionDiffusionSpeed}, it is evident that normalized numerical speed for standard schemes, \sfdC{1}{1}{1}{1}{4} (cyan) and \sfdC{3}{3}{2}{2}{10} (black), deviates from unity at lower wavenumbers leading to large dispersion error at higher wavenumbers.
The normalized numerical speed for optimized schemes \ofdC{3}{3}{2}{2}{4} (red) and \ofdC{3}{3}{3}{3}{4} (blue), is close to unity even at higher wavenumbers.
As in the case of dissipative error, here also we observe that the standard schemes show steep variation of dispersion error, while the optimized schemes are relatively slower in variation.

\begin{figure}[ht!]
\begin{center}
\includegraphics[width=\textwidth]{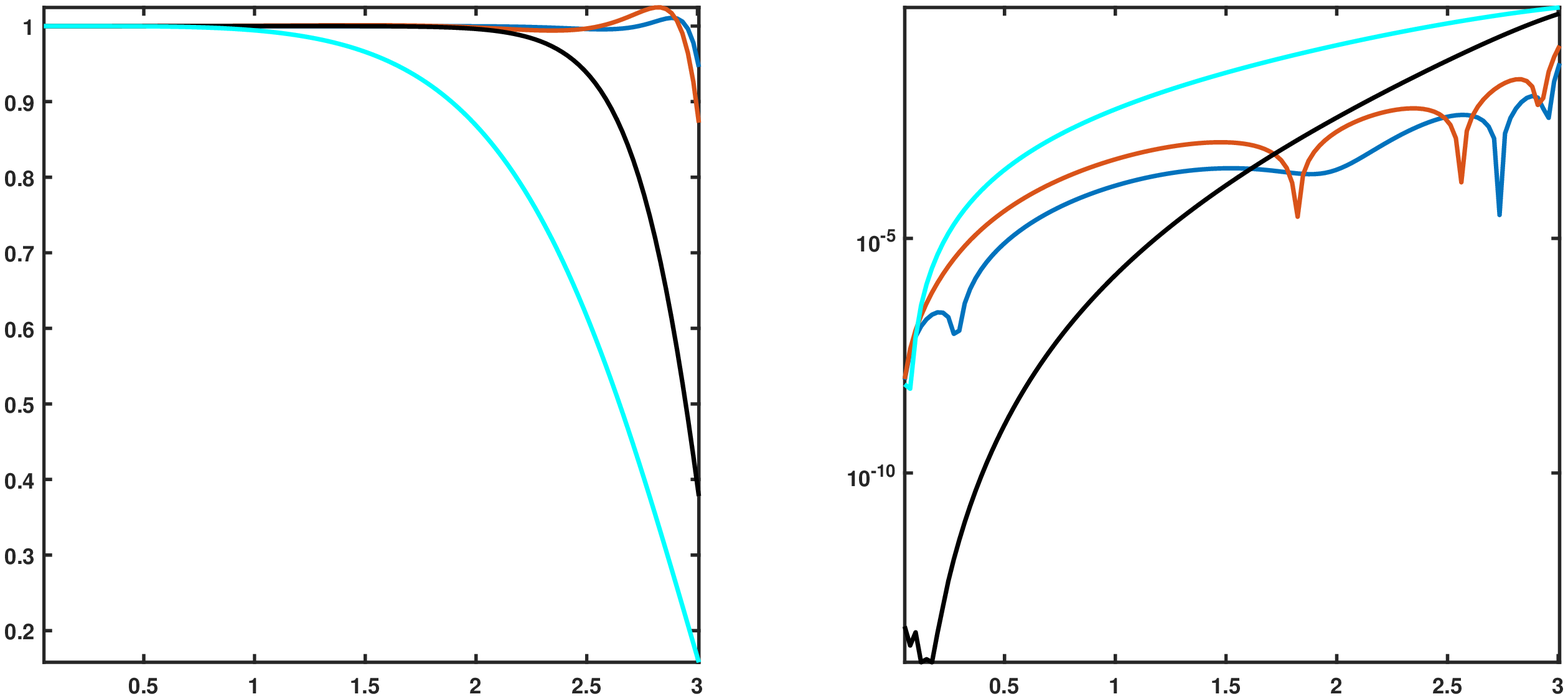}
\begin{picture}(0,0)
				\put(-255,120){\rotatebox{90} {$c^{\ast}$}}
				\put(0,105){\rotatebox{90} {$|c^{\ast} -1|$}}
				\put(-130,-2){$\kdx$}
				\put(130,-2){$\kdx$}
\end{picture}
\caption{Dispersion error for advection-diffusion equation. Blue -- \ofdC{3}{3}{3}{3}{4}; red -- \ofdC{3}{3}{2}{2}{4}; black -- \sfdC{3}{3}{2}{2}{10}; cyan -- \sfdC{1}{1}{1}{1}{4}. $r_1\approx r_2 = 0.01$, $t^\ast_2 \approx 25$, $A(k) = 1$, $\gamma(\kdx) = 1$ for $\kdx\in[0, 3]$, and $\gamma(\kdx) = 0$ otherwise.}
\figlabel{spectralErrAdvectionDiffusionSpeed}
\end{center}
\end{figure}

\subsection{Non-linear advection-diffusion equation}
Now let us consider non-linear advection-diffusion equation which is also known as Burgers' equation and widely studied in various areas of applied mathematics.
\begin{align}
\frac{\partial f}{\partial t} = -f \frac{\partial f}{ \partial x} + \beta_2 \frac{\partial ^2 f}{ \partial x^2}.
\eqnlabel{pdeNonLinAdvDiff}
\end{align}
Because of the non-linear term, different Fourier modes interact with each other and produce new wavenumbers in the solution. Therefore, to ensure that all wavenumbers are
well resolved, we performed grid convergence study and observed that space-averaged energy ($K:=\sum _{i=1} ^ {N_p} f_i^2(x,t)/N_p$) becomes independent of the number of grid points in the domain at $N_p = 256$.

Numerical solution of \eqn{pdeNonLinAdvDiff} is obtained in a periodic domain $x\in[0,2\pi]$ with initial condition given by \eqn{initF}, where $A(k) = k^{-1/2}$. The negative exponent ensures that
higher wavenumbers have sufficient energy content without being unstable.
For this particular PDE, we define timescale as $t = K/ \epsilon$, where $\epsilon$ is space averaged energy dissipation rate ($\epsilon:=\sum _{i=1} ^ {N_p} (\partial f_i(x,t)/ \partial x)^2/N_p$).
Then time is expressed as a normalized quantity $t^\ast := t/t_0$ where $t_0 = K_0/ \epsilon_0$ is the initial timescale.

The analytical solution of \eqn{pdeNonLinAdvDiff} is obtained by applying Cole-Hopf transformation as follows \cite{eberhard1950, cole1951}.
First, let us define the transformation
\begin{align*}
f = -2\beta_2\frac{1}{\phi}\frac{\partial \phi}{ \partial x}.
\end{align*}
With this substitution, \eqn{pdeNonLinAdvDiff} becomes diffusion equation and the analytical solution for $f(x,t)$ is given by
\begin{align}
f(x,t) = \frac{\int _{-\infty} ^{\infty} \frac{x-y}{t} \phi(y,0) \exp \Big(\frac{(x-y)^2}{4\beta_2 t} \Big) dy} {\int _{-\infty} ^{\infty} \phi(y,0) \exp \Big(\frac{(x-y)^2}{4\beta_2 t} \Big) dy},
\eqnlabel{nonLinAdvDiffSol}
\end{align}
where initial $\phi(y,0)$ is obtained using
\begin{align*}
\phi(y,0) = \exp\Big(  -\int_0 ^y \frac{f(z,0)}{2\beta_2} dz\Big).
\end{align*}
\eqn{nonLinAdvDiffSol} is the exact solution for non-linear advection-diffusion equation. However, integrals involved are computed numerically.

\fig{spectralErrNonLinAdvDiffEnergyContent} shows the spectral energy content normalized with initial energy for different wavenumbers. All numerical results for this PDE are obtained for $\beta_2 = 0.04$, $r_2 = 0.01$ at $t^\ast \approx 104$.
It can be clearly seen that the spectral energy content for standard schemes \sfdC{1}{1}{1}{1}{4} (cyan) and \sfdC{3}{3}{2}{2}{10} (black solid) deviates from analytical solution (black dashed) at high wavenumbers. In the same wavenumber region, the optimized schemes \ofdC{3}{3}{2}{2}{4} (red) and \ofdC{3}{3}{3}{3}{4} (blue) are relatively closer to the analytical solution.

\begin{figure}[ht!]
\begin{center}
\includegraphics[width=0.65\textwidth]{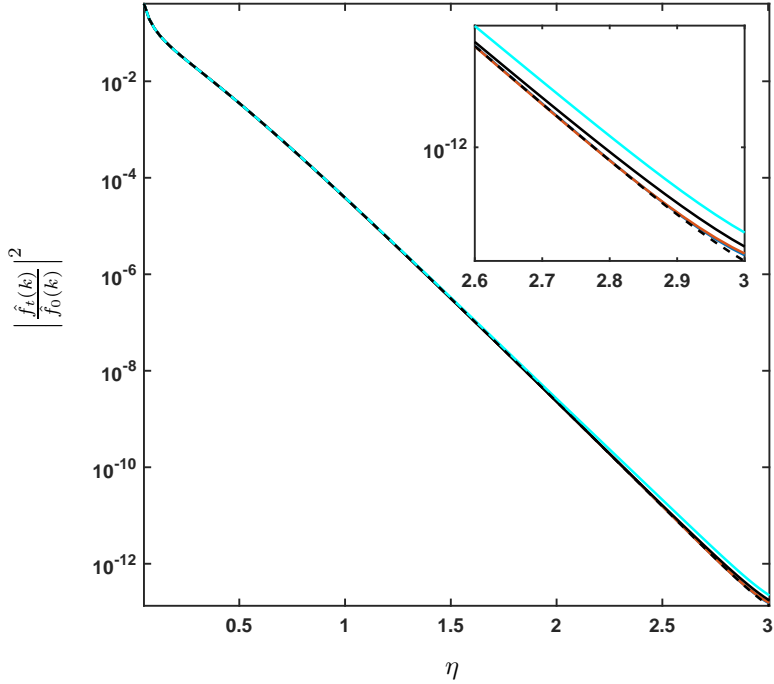}
\begin{picture}(0,0)
	\put(-320,135){\rotatebox{90} {$\Big|\frac{\hat{f}_t(k)}{\hat{f}_0(k)}\Big|^2$}}
	\put(-155,5){$\kdx$}
\end{picture}
\caption{Spectral energy content for non-linear advection-diffusion equation normalized with initial energy. Blue -- \ofdC{3}{3}{3}{3}{4}; red -- \ofdC{3}{3}{2}{2}{4}; black -- \sfdC{3}{3}{2}{2}{10}; cyan -- \sfdC{1}{1}{1}{1}{4}; black dashed -- analytical solution. $\beta_2 = 0.04$, $r_2 = 0.01$, $A(k) = k^{-1/2}$, $t^\ast \approx 104$, $k_{\max} = 121$, $\gamma(\kdx) = 1$ for $\kdx\in[0, 3]$, and $\gamma(\kdx) = 0$ otherwise.}
\figlabel{spectralErrNonLinAdvDiffEnergyContent}
\end{center}
\end{figure}

Similar to \fig{spectralErrAdvectionDiffusion}, we show the normalized error in spectral energy for the non-linear advection-diffusion equation in \fig{spectralErrNonLinAdvDiff}. Due to the interaction of Fourier modes with each other, in general, we do not expect the spectral behavior of different schemes for non-linear PDE to be in exact agreement with the behavior observed for the linear case.
Unlike the linear case (\fig{spectralErrAdvectionDiffusion}), \fig{spectralErrNonLinAdvDiff} shows that the optimized schemes \ofdC{3}{3}{2}{2}{4} (red) and \ofdC{3}{3}{3}{3}{4} (blue) have better accuracy than \sfdC{3}{3}{2}{2}{10} (black) in both low and high wavenumber region. Overall, \sfdC{1}{1}{1}{1}{4} (cyan) has the largest error. The error for \ofdC{3}{3}{3}{3}{4} is somewhat comparable to the pentadiagonal \ofdC{3}{3}{2}{2}{4}.

\begin{figure}[ht!]
\begin{center}
\includegraphics[width=0.5\textwidth]{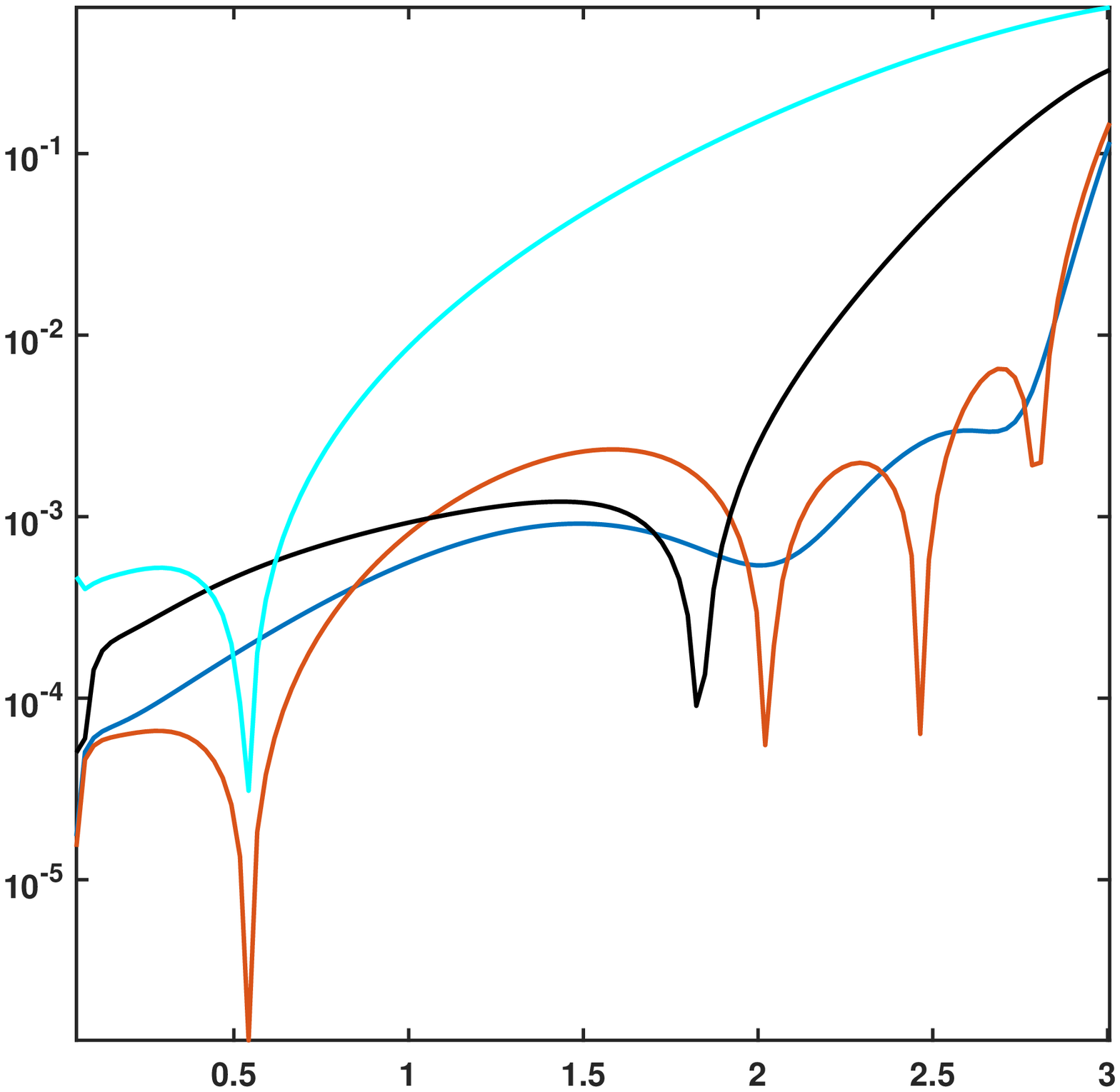}
\begin{picture}(0,0)
				\put(-260,100){\rotatebox{90} {$\Big| \Big|\frac{\hat{f}_t(k)}{\hat{f}_t(k)_a}\Big|^2 -1 \Big|$}}
				\put(-120,-10){$\kdx$}
\end{picture}
\caption{Diffusion error for non-linear advection-diffusion equation. Blue -- \ofdC{3}{3}{3}{3}{4}; red -- \ofdC{3}{3}{2}{2}{4}; black -- \sfdC{3}{3}{2}{2}{10}; cyan -- \sfdC{1}{1}{1}{1}{4}. $\beta_2 = 0.04$, $r_2 = 0.01$, $A(k) = k^{-1/2}$, $t^\ast \approx 104$, $k_{\max} = 121$, $\gamma(\kdx) = 1$ for $\kdx\in[0, 3]$, and $\gamma(\kdx) = 0$ otherwise.}
\figlabel{spectralErrNonLinAdvDiff}
\end{center}
\end{figure}

Because of the non-linearity, there is no well-defined analytical speed in this case and therefore, the notion of numerical speed can not be used to quantify the dispersion error.
However, we can calculate and compare the argument of Fourier coefficients, $\hat{\theta_k}:=\arg[\hat{f}_t(k)]$, for different wavenumbers. Normalized error in $\hat{\theta_k}$ is shown in \fig{spectralErrNonLinAdvDiffPhase}. Clearly, the optimized schemes \ofdC{3}{3}{2}{2}{4} (red) and \ofdC{3}{3}{3}{3}{4} (blue) have smaller phase error than the standard schemes \sfdC{1}{1}{1}{1}{4} (cyan) and \sfdC{3}{3}{2}{2}{10} (black) across all wavenumbers.
Whereas, in the linear case (\fig{spectralErrAdvectionDiffusionSpeed}) \sfdC{3}{3}{2}{2}{10} demonstrates better accuracy than the optimized schemes in the low wavenumber region. Also, \ofdC{3}{3}{3}{3}{4} shows better accuracy than the pentadiagonal \ofdC{3}{3}{2}{2}{4} in \fig{spectralErrNonLinAdvDiffPhase}.

\begin{figure}[ht!]
\begin{center}
\includegraphics[width=0.5\textwidth]{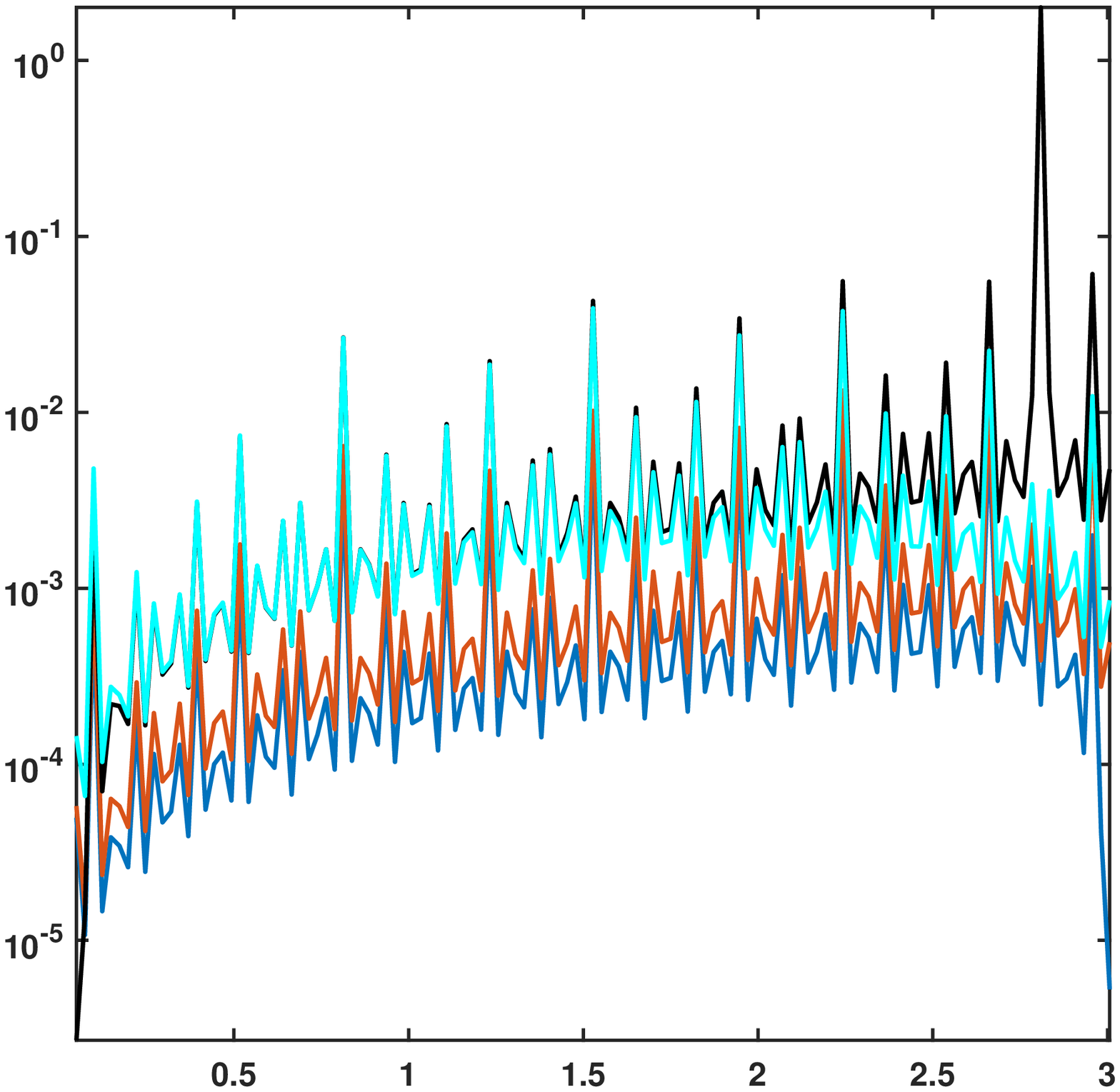}
\begin{picture}(0,0)
        \put(-260,100){\rotatebox{90} {$\Big| \frac{\hat{\theta_k}}{\hat{\theta_k}_a} -1 \Big|$}}
				\put(-120,-10){$\kdx$}
\end{picture}
\caption{Phase error for non-linear advection-diffusion equation. Blue -- \ofdC{3}{3}{3}{3}{4}; red -- \ofdC{3}{3}{2}{2}{4}; black -- \sfdC{3}{3}{2}{2}{10}; cyan -- \sfdC{1}{1}{1}{1}{4}. $\beta_2 = 0.04$, $r_2 = 0.01$, $A(k) = k^{-1/2}$, $t^\ast \approx 104$, $k_{\max} = 121$, $\gamma(\kdx) = 1$ for $\kdx\in[0, 3]$, and $\gamma(\kdx) = 0$ otherwise.}
\figlabel{spectralErrNonLinAdvDiffPhase}
\end{center}
\end{figure}

\fig{spectralErrNonLinAdvDiff} and \fig{spectralErrNonLinAdvDiffPhase} independently quantify the errors in $|\hat{f}_t(k)|$ and $\hat{\theta_k}$. \fig{spectralErrNonLinAdvDiffAmpPhase} accounts for $|\hat{f}_t(k) - \hat{f}_t(k)_a|$, which quantifies the combined effect of error in $|\hat{f}_t(k)|$ and $\hat{\theta_k}$ in a single plot. When we look at the combined error in the non-linear case, it is observed from \fig{spectralErrNonLinAdvDiffAmpPhase} that the optimized schemes
\ofdC{3}{3}{2}{2}{4} (red) and \ofdC{3}{3}{3}{3}{4} (blue) have better accuracy than the standard schemes \sfdC{1}{1}{1}{1}{4} (cyan) and \sfdC{3}{3}{2}{2}{10} (black) at all wavenumbers. Moreover, \ofdC{3}{3}{3}{3}{4} demonstrates better accuracy than the pentadiagonal \ofdC{3}{3}{2}{2}{4} and hence reinforcing Remark \ref{rem:symmBest}.
Thus, for solving practical physical problems which are governed by non-linear PDEs, optimized schemes seem to provide much better spectral resolution than the non-optimized standard schemes.

\begin{figure}[ht!]
\begin{center}
\includegraphics[width=0.5\textwidth]{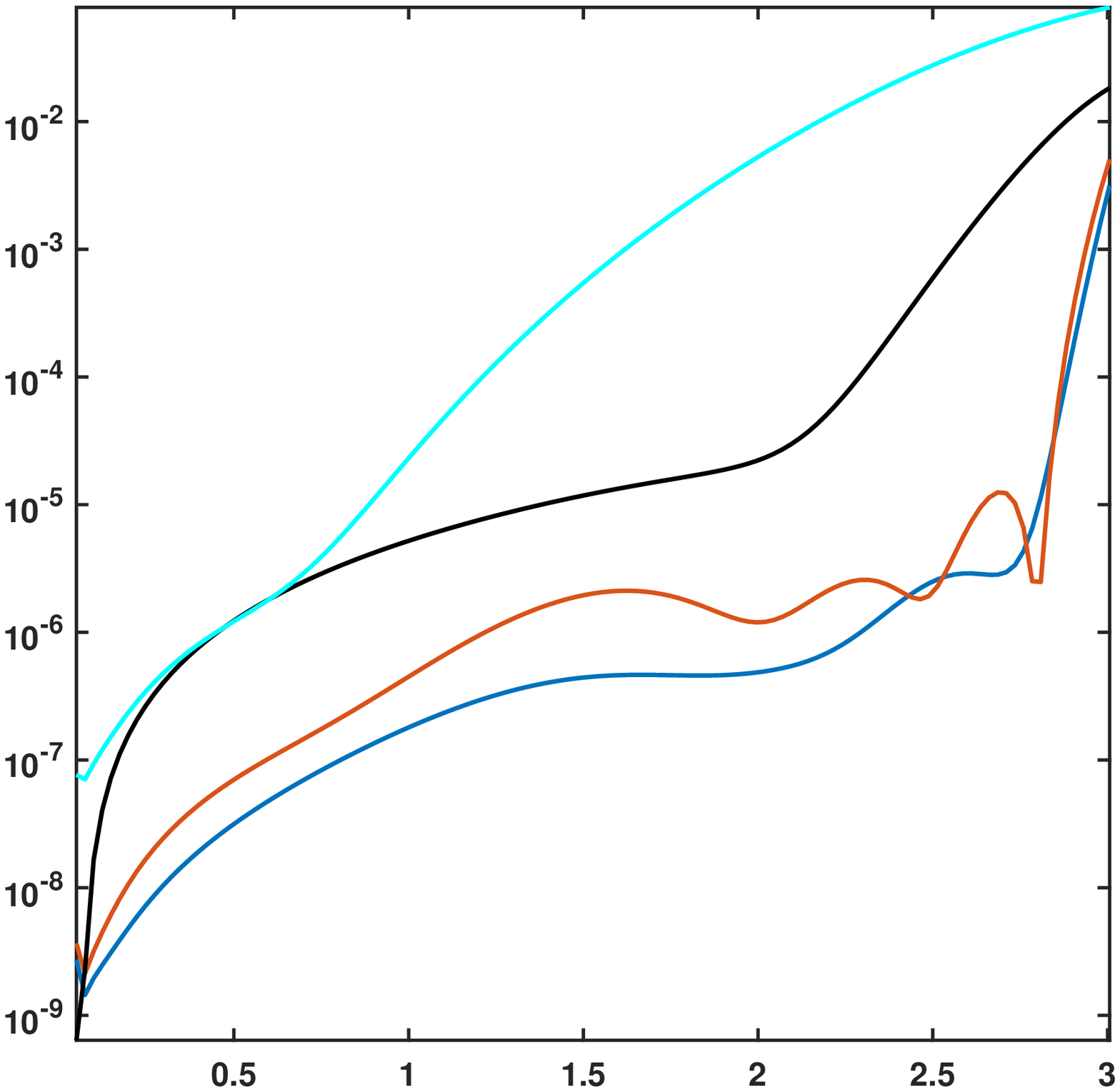}
\begin{picture}(0,0)
				\put(-260,100){\rotatebox{90} {$\Big|\frac{\hat{f}_t(k)}{\hat{f}_t(k)_a} -1 \Big|^2$}}
				\put(-120,-10){$\kdx$}
\end{picture}
\caption{Amplitude error for non-linear advection-diffusion equation. Blue -- \ofdC{3}{3}{3}{3}{4}; red -- \ofdC{3}{3}{2}{2}{4}; black -- \sfdC{3}{3}{2}{2}{10}; cyan -- \sfdC{1}{1}{1}{1}{4}. $\beta_2 = 0.04$, $r_2 = 0.01$, $A(k) = k^{-1/2}$, $t^\ast \approx 104$, $k_{\max} = 121$, $\gamma(\kdx) = 1$ for $\kdx\in[0, 3]$, and $\gamma(\kdx) = 0$ otherwise.}
\figlabel{spectralErrNonLinAdvDiffAmpPhase}
\end{center}
\end{figure}

\section{Conclusions} \label{conclusions}

Central compact standard finite differences are derived using the popular pentadiagonal 7 point stencil formulation to maximize formal order of accuracy of the approximation. In this paper, we presented a generalized unified framework to derive central compact optimal schemes with equal LHS and RHS stencil sizes.
The optimal coefficients are determined analytically by solving an optimization problem. The cost function, defined to be the $\mathcal{L}_2$ norm of spectral error, is minimized subject to linear equality constraints in coefficients. We have shown that for a given order of accuracy, we can increase the stencil size and
hence the number of degrees of freedom in the optimization to achieve better spectral resolution. The freedom of selecting weighting function ($\wtfun$) allows us to achieve better resolution of the wavenumbers which are important to the physics of problem by giving them more weightage in the optimization problem.
We have also presented an analytical justification for the symmetry/skew-symmetry of optimal coefficients resulting in purely real or imaginary error which depends on order of the derivative being approximated.

We have shown that special cases namely, central schemes with unequal LHS and RHS stencil sizes (e.g. pentadiagonal), explicit, and biased schemes can be derived by imposing additional linear constraints in the optimization problem. Using these three special cases as building blocks,
an optimal scheme with any desired structure can be derived as all constraints are linear and therefore, they can be imposed in any number and combination.

We have also presented a rigorous stability analysis for a generalized implicit RK scheme. We showed that by increasing the stencil size the spectral resolution can be improved, but simultaneously it results in reduced stability region in $r_1-r_2$ space.

Finally, we presented numerical results for optimized schemes to benchmark them against the standard schemes of the same accuracy order, and the same stencil size. The optimized schemes show better accuracy than the standard scheme of the same stencil size with larger formal order of accuracy in the high wavenumber region. The better resolution in
high wavenumber region is achieved at the expense of relatively larger error at the lower wavenumbers, because $\wtfun=1$ used in simulations weights all wavenumbers equally. However, in the case of non-linear PDE, we observed that the optimized schemes show overall better accuracy than the standard schemes across the wavenumber spectrum.

The observation, compact schemes with equal LHS and RHS stencil sizes demonstrate better performance than those with unequal stencil sizes, has been reinforced through numerical simulations. However, we recognize that the usage of schemes with equal LHS and RHS stencil sizes e.g., \ofdC{3}{3}{3}{3}{\cdot},
might prove to be impractical for large simulations as the computation cost for \ofdC{3}{3}{3}{3}{\cdot} can be substantially higher than that of pentadiagonal \ofdC{3}{3}{2}{2}{\cdot}. A different study will be performed in the future to assess the tradeoff between the accuracy we gain and the computation cost we incur due to usage of schemes such as \ofdC{3}{3}{3}{3}{\cdot}.
To reduce the computational cost, we intend to develop the present framework further to derive prefactored optimized schemes such as presented in \cite{zang2003prefactor, zang2011prefactor, rona2017prefactor}.

Empirical observations show that KKT matrix used in determining optimal coefficients of implicit schemes suffer from rank deficiency at certain stencil size ($\npt$). However, explicit schemes are unaffected by this problem. This issue, perhaps, hints at the fundamental limitation of implicit schemes to scale up beyond certain stencil size, and it will be investigated in our future studies.

Although our framework allows us to derive biased schemes, for the brevity of discussion we have restricted numerical results presented in this work on periodic domains. Also, implementation of boundary conditions without compromising the accuracy of numerical solution is an important area of research and will be addressed in our future work.

% \clearpage
\appendix
\section{Invariance of grid point location} \label{sec:appInvarianceGrid}
\begin{lem}
Optimal implicit central finite-difference approximation is invariant of the grid point location in the domain.
\label{lem:gridInvariance}
\end{lem}
\begin{proof}
Let the coefficients for the $d^\text{th}$ derivative at the $i^\text{th}$ grid point corresponding to \eqn{pade1} be parameterized by ${\vo{a}}_{i,d}$ and ${\vo{b}}_{i,d}$, and $\vo{A}_d$ and $\vo{B}_d$ be the vertical stacking of $\vo{a}^T_{i,d}$ and $\vo{b}^T_{i,d}$ respectively, i.e.
\begin{align*}
	\vo{A}_d := \begin{bmatrix}{\vo{a}}_{1,d}^T \\ \vdots \\ {\vo{a}}_{N_p,d}^T \end{bmatrix}, \; \vo{B}_d := \begin{bmatrix}{\vo{b}}_{1,d}^T \\ \vdots \\ {\vo{b}}_{N_p,d}^T \end{bmatrix}.
	% \eqnlabel{defAdBd}
\end{align*}
The order accuracy constraint, for every $\vo{a}_{i,d}, \vo{b}_{i,d}$ pair,  can be compactly written as
\begin{align*}
\begin{bmatrix} \vo{A}_d & \vo{B}_d \end{bmatrix} \begin{bmatrix} \vo{X}_d \\ -\vo{Y}_d\end{bmatrix} = \vo{0}_{N_p\times (d+p+1)}.
\end{align*}
The cost function over $N_p$ grid points in the entire domain is
\begin{align*}
\sum_{i=1}^{N_p}
\begin{pmatrix}\vo{a}_{i,d}\\\vo{b}_{i,d}\end{pmatrix}^T\vo{Q}_d\begin{pmatrix}\vo{a}_{i,d}\\\vo{b}_{i,d}\end{pmatrix}.
\end{align*}
Optimization of the spectral error is separable for different derivatives, and can be formulated as following from \eqn{evenOptim}
\begin{equation}\left.
\begin{aligned}
& \min_{\vo{a}_{i,d},\vo{b}_{i,d}} \sum_{i=1}^{N_p}
\begin{pmatrix}\vo{a}_{i,d}\\\vo{b}_{i,d}\end{pmatrix}^T\vo{Q}_d\begin{pmatrix}\vo{a}_{i,d}\\\vo{b}_{i,d}\end{pmatrix},\\
&\text{ subject to } \\
&\begin{bmatrix} \vo{A}_d & \vo{B}_d \end{bmatrix} \begin{bmatrix} \vo{X}_d \\ -\vo{Y}_d\end{bmatrix} = \vo{0}_{N_p\times (d+p+1)},\\
&\vo{B}_d\bdel_N(\npt+1) = \vo{1}_{N\times 1},
\end{aligned}\;\;\right\}
\eqnlabel{padeAnalytical1}
\end{equation}
for $d=1,\cdots,D$. Let us define vector,
\begin{align*}
\vo{v}_d := \begin{pmatrix}\vo{a}_{1,d} \\ \vo{b}_{1,d} \\ \vdots \\ \vo{a}_{N_p,d} \\ \vo{b}_{N_p,d}\end{pmatrix}.
\end{align*}
The optimization in \eqn{padeAnalytical1}, can be written in terms of $\vo{v}_d$ as
\begin{align*}
\min_{\vo{v}_d} \vo{v}_d^T\left(\vo{I}_{N_p}\otimes \vo{Q}_d\right)\vo{v}_d, \text{ subject to}
\left(\vo{I}_{N_p}\otimes \begin{bmatrix}
\vo{X}_d^T & -\vo{Y}_d^T\\
\vo{0}_{1 \times N} & \bdel_N^T(\npt+1)
\end{bmatrix}\right)\vo{v}_d =  \vo{1}_{N_p\times 1} \otimes \bdel_{n}(n),
\end{align*}
for $n=(d+p+2)$. Then the analytical solution given by
\begin{align*}
\begin{pmatrix}\vo{v}_d^\ast\\\boldsymbol{\Lambda}_d^\ast\end{pmatrix} = \left[\begin{array}{cc} \vo{I}_{N_p}\otimes \vo{Q}_d & \vo{I}_{N_p}\otimes\left(\begin{array}{cc}\vo{X}_d^T & -\vo{Y}_d^T\\
\vo{0}_{1 \times N} & \bdel_N^T(\npt+1)\end{array}\right)^T\\
 \vo{I}_{N_p}\otimes\left(\begin{array}{cc}\vo{X}_d^T & -\vo{Y}_d^T\\
\vo{0}_{1 \times N} & \bdel_N^T(\npt+1)\end{array}\right) & \vo{I}_{N_p}\otimes\vo{0}_{(d+p+2)\times(d+p+2)}
\end{array}\right]^{-1}
\begin{bmatrix} \vo{1}_{N_p\times 1} \otimes \vo{0}_{2N\times 1} \\ \vo{1}_{N_p\times 1} \otimes \bdel_{n}(n) \end{bmatrix}.
%\eqnlabel{padeDomainOptimal}
\end{align*}
Inverse of a partitioned matrix is given by
\begin{align}
\begin{bmatrix}\mathbf {A} &\mathbf {B} \\\mathbf {C} & \mathbf {D} \end{bmatrix}^{-1}={\begin{bmatrix}\mathbf {A} ^{-1}+\mathbf {A} ^{-1}\mathbf {B} (\mathbf {D} -\mathbf {CA} ^{-1}\mathbf {B} )^{-1}\mathbf {CA} ^{-1}&-\mathbf {A} ^{-1}\mathbf {B} (\mathbf {D} -\mathbf {CA} ^{-1}\mathbf {B} )^{-1}\\-(\mathbf {D} -\mathbf {CA} ^{-1}\mathbf {B} )^{-1}\mathbf {CA} ^{-1}&(\mathbf {D} -\mathbf {CA} ^{-1}\mathbf {B} )^{-1}\end{bmatrix}}.
\eqnlabel{partitionInverse}
\end{align}
Let,
\begin{align*}
\left[\begin{array}{cc} \vo{Q}_d & \left(\begin{array}{cc}\vo{X}_d^T & -\vo{Y}_d^T\\
\vo{0}_{1 \times N} & \bdel_N^T(\npt+1)\end{array}\right)^T\\
 \left(\begin{array}{cc}\vo{X}_d^T & -\vo{Y}_d^T\\
\vo{0}_{1 \times N} & \bdel_N^T(\npt+1)\end{array}\right) & \vo{0}_{(d+p+2)\times(d+p+2)}
\end{array}\right]^{-1} := \begin{bmatrix}\vo{M}_1 & \vo{M}_2\\ \vo{M}_3 & \vo{M}_4 \end{bmatrix},
\end{align*}
where $\vo{M}_i$ are obtained using \eqn{partitionInverse}. Using the identity
\begin{align*}
(\vo{A}\otimes\vo{B})(\vo{C}\otimes\vo{D}) = \vo{AC}\otimes\vo{BD},
\end{align*}
and \eqn{partitionInverse}, we get
\begin{align*}
\left[\begin{array}{cc} \vo{I}_{N_p}\otimes \vo{Q}_d & \vo{I}_{N_p}\otimes\left(\begin{array}{cc}\vo{X}_d^T & -\vo{Y}_d^T\\
\vo{0}_{1 \times N} & \bdel_N^T(\npt+1)\end{array}\right)^T\\
 \vo{I}_{N_p}\otimes\left(\begin{array}{cc}\vo{X}_d^T & -\vo{Y}_d^T\\
\vo{0}_{1 \times N} & \bdel_N^T(\npt+1)\end{array}\right) & \vo{I}_{N_p}\otimes\vo{0}_{(d+p+2)\times(d+p+2)}
\end{array}\right]^{-1} = \begin{bmatrix}\vo{I}_{N_p}\otimes \vo{M}_1 & \vo{I}_{N_p}\otimes \vo{M}_2\\ \vo{I}_{N_p}\otimes \vo{M}_3 & \vo{I}_{N_p}\otimes \vo{M}_4 \end{bmatrix}.
\end{align*}
Therefore, the optimal solution
\begin{align*}
\begin{pmatrix}
\vo{v}_d^\ast\\\boldsymbol{\Lambda}_d^\ast
\end{pmatrix} &= \begin{bmatrix}\left(\vo{I}_{N_p}\otimes \vo{M}_1\right)\left(\vo{1}_{N_p\times 1}\otimes \vo{0}_{2N \times 1}\right) + \left(\vo{I}_{N_p}\otimes \vo{M}_2\right)\left(\vo{1}_{N_p\times 1} \otimes \bdel_{n}(n)\right)\\
\left(\vo{I}_{N_p}\otimes \vo{M}_3\right)\left(\vo{1}_{N_p\times 1}\otimes \vo{0}_{2N \times 1}\right) + \left(\vo{I}_{N_p}\otimes \vo{M}_4\right)\left(\vo{1}_{N_p\times 1} \otimes \bdel_{n}(n)\right)\end{bmatrix},\\
& = \begin{bmatrix} \vo{1}_{N_p\times 1} \otimes \left(\vo{M}_1 \vo{0}_{2N \times 1} + \vo{M}_2 \bdel_{n}(n) \right)\\[2mm]
\vo{1}_{N_p\times 1}\otimes \left(\vo{M}_3 \vo{0}_{2N \times 1} + \vo{M}_4 \bdel_{n}(n) \right)\end{bmatrix},\\
& = \begin{bmatrix}
\vo{1}_{N_p\times 1} \otimes \begin{pmatrix} \vo{a}_d^\ast\\ \vo{b}_d^\ast \end{pmatrix} \\
\vo{1}_{N_p\times 1} \otimes \boldsymbol{\lambda}_d^\ast
\end{bmatrix},
\end{align*}
where $n=(d+p+2)$ and $(\vo{a}_d^\ast,\vo{b}_d^\ast,,\boldsymbol{\lambda}_d^\ast)$ is the solution of \eqn{padeEvenOptimal}. Therefore, the optimal solution is identical for all grid points.
\end{proof}
Note that, Lemma \ref{lem:gridInvariance} holds only for central schemes used on a periodic domain. In a scenario where we have boundary conditions, optimal coefficients will be different for grid points near boundaries of the domain.

\clearpage
\section{Optimal coefficients} \label{sec:appOptimCoeff}

\FloatBarrier
\begin{table}[htb]
\centering
\begin{tabular}{|c|c|c|c|c|}
 \hline
& $M=1$ & $M=2$ & $M=3$ & $M=4$ \\
 \hline
 $\mathrm{a}_{0} ^\ast$ & -2.4 & -1.55920152194026 & -0.979288292571078 & -0.719422653838933 \\
 $\mathrm{a}_{1} ^\ast$ &  1.2 & 0.396897309677732 & -0.033306701818875 & -0.156640799785708 \\
 $\mathrm{a}_{2} ^\ast$ &      & 0.382703451292396 & 0.440495791275238  & 0.340037669820826  \\
 $\mathrm{a}_{3} ^\ast$ &      &                   & 0.0824550568291757 & 0.161702448995973  \\
 $\mathrm{a}_{4} ^\ast$ &      &                   &                    & 0.0146120078883761 \\
 \hline
 $\mathrm{b}_{0} ^\ast$ & 1   & 1                  & 1                   & 1                    \\
 $\mathrm{b}_{1} ^\ast$ & 0.1 & 0.437358728499431  & 0.607804000534683   & 0.69537501810989     \\
 $\mathrm{b}_{2} ^\ast$ &     & 0.0264968289242269 & 0.122983617052232   & 0.223008838348136    \\
 $\mathrm{b}_{3} ^\ast$ &     &                    & 0.00459836978541528 & 0.0272452165890212   \\
 $\mathrm{b}_{4} ^\ast$ &     &                    &                     & 0.000682950290635715 \\
\hline
\end{tabular}
\caption{Coefficients for the second derivative corresponding to \fig{stencilCoeffD2}.
Note that  $\mathrm{a}_{-m} ^\ast= \mathrm{a}_{m} ^\ast$ and $\mathrm{b}_{-m} ^\ast= \mathrm{b}_{m} ^\ast$.}
\label{table:stencilCoeffD2}
\end{table}

\FloatBarrier
\begin{table}[htb]
\centering
\begin{tabular}{|c|c|c|c|c|}
 \hline
 & $M=1$ & $M=2$ & $M=3$ & $M=4$ \\
 \hline
 $\mathrm{a}_{0} ^\ast$ & 0    & 0                 & 0                  & 0                   \\
 $\mathrm{a}_{1} ^\ast$ & 0.75 & 0.682194069313335 & 0.560054939856331  & 0.472419664132013   \\
 $\mathrm{a}_{2} ^\ast$ &      & 0.214144479273011 & 0.326746645436286  & 0.367572867069987   \\
 $\mathrm{a}_{3} ^\ast$ &      &                   & 0.0418602478971568 & 0.0980340659498803  \\
 $\mathrm{a}_{4} ^\ast$ &      &                   &                    & 0.00699750157631073 \\
\hline
 $\mathrm{b}_{0} ^\ast$ & 1    & 1                  & 1                  & 1                   \\
 $\mathrm{b}_{1} ^\ast$ & 0.25 & 0.547827381201651  & 0.658367308183134  & 0.72407136413065    \\
 $\mathrm{b}_{2} ^\ast$ &      & 0.0626556466577058 & 0.170094141092335  & 0.26326428439027    \\
 $\mathrm{b}_{3} ^\ast$ &      &                    & 0.0106675251449049 & 0.0407179433494389  \\
 $\mathrm{b}_{4} ^\ast$ &      &                    &                    & 0.00160401055651088 \\
\hline
\end{tabular}
\caption{Coefficients for the first derivative corresponding to \fig{stencilCoeffD1}.
Note that  $\mathrm{a}_{-m} ^\ast= -\mathrm{a}_{m} ^\ast$ and $\mathrm{b}_{-m} ^\ast= \mathrm{b}_{m} ^\ast$.}
\label{table:stencilCoeffD1}
\end{table}

Note that, in Table \ref{table:stencilCoeffNonCentralD2} and \ref{table:stencilCoeffNonCentralD1}, $^\#$ denotes coefficient of the grid point at which derivative is to be approximated.
\FloatBarrier
\begin{table}[htb]
\centering
\begin{tabular}{|c|c|c|c|c|}
 \hline
 & $\npt_L=4$ & $\npt_L=5$ & $\npt_L=6$ \\
 \hline
               & 0.135141927552199        & 0.662304984252634     & 17.3670624080996  \\
               & 0.722707534591416        & 3.53558092392018      & 92.3021288782114  \\
               & -0.0524729395599328      & -0.250337371007728    & -6.49624407257855 \\
$\vo{a}^\ast$  & -1.60731259496048        & -7.85976328865994     & -204.957850510534 \\
               & $^\#$-0.0583231083517459 & -0.310357190338834    & -8.85074823578218 \\
               & 0.72408652155438         & $^\#$3.54961920557169 & 92.8470054538251  \\
               & 0.13617265917416         & 0.672952736261999     & $^\#$17.7886460787584 \\
\hline
               & 0.0075365800956553  & 0.0369784406883765 & 0.971865865002116 \\
               & 0.201615926044887   & 0.987464737949076  & 25.8562746553017  \\
               & 0.99713483322273    & 4.88096076338054   & 127.576933509124  \\
 $\vo{b}^\ast$ & 1.64238167997833    & 8.04522660766854   & 210.335351055833  \\
               & $^\#  $1            & 4.91039583089326   & 128.733779558004  \\
               & 0.202831788738001   & $^\#  $1           & 26.3512560574889  \\
               & 0.00760492052475932 & 0.0376863400465079 & $^\#  $1          \\
\hline
\end{tabular}
\caption{Coefficients for the left-biased schemes which approximate second derivative corresponding to \fig{implicitspectralErrorNonPeriodD2}.
Corresponding central scheme is $M=3$ from Table \ref{table:stencilCoeffD2}.}
\label{table:stencilCoeffNonCentralD2}
\end{table}

\FloatBarrier
\begin{table}[htb]
\centering
\begin{tabular}{|c|c|c|c|c|}
 \hline
 & $\npt_L=4$ & $\npt_L=5$ & $\npt_L=6$ \\
 \hline
               &-0.0621972998530267        & -0.232619531619737       & -3.52690296300194 \\
               &-0.488868232296276         & -1.83780260017807        & -27.9116587382461 \\
               &-0.846178148476397         & -3.2098376453028         & -49.0922987950105 \\
$\vo{a}^\ast$  &-0.00918720281492579       & -0.0700718215089091      & -1.57497223752058 \\
               & $^\#  $0.844294225659752  &  3.19534746352695        & 48.7613719684532  \\
               & 0.497811414459153         & $^\#  $1.90604922449128  & 29.4470383302118  \\
               & 0.0643252433217213        & 0.248934910591297        & $^\#  $3.89742243511407 \\
\hline
               & 0.0158345798757476 & 0.0591989978487381 & 0.898171816719291 \\
               & 0.253744537333521  & 0.951710004840582  & 14.4387550915639  \\
               & 0.987857487221426  & 3.72365571735861   & 56.675421516294   \\
 $\vo{b}^\ast$ & 1.50948607590456   & 5.7229905302476    & 87.5412499132427  \\
               & $^\#  $1           & 3.81632520991078   & 58.7608049690781  \\
               & 0.260059400465721  & $^\#  $1           & 15.5296377878293  \\
               & 0.016417216370275  & 0.0636716245952377 & $^\#  $1          \\
\hline
\end{tabular}
\caption{Coefficients for the left-biased schemes which approximate first derivative corresponding to \fig{implicitspectralErrorNonPeriodD1}.
Corresponding central scheme is $M=3$ from Table \ref{table:stencilCoeffD1}.}
\label{table:stencilCoeffNonCentralD1}
\end{table}
\FloatBarrier

\section{Butcher tableaux of RK schemes} \label{sec:appButcherTab}
\iffalse
\begin{array}{c|c}
\rkC  & \rkA \\
\hline
\quad & \rkB^{T}
\end{array}
\fi

\begin{enumerate}
\item Forward Euler method (FE) % RKFE
\begin{align*}
\begin{array}{c|c }
0 & 0 \\
\hline
\quad & 1 \\
\end{array}
\end{align*}

Following RK schemes are taken from \cite{butcher2008ode}.
\item Explicit four stage RK method (ERK4) % RKBut41
\begin{align*}
\begin{array}{c|c c c c}
0 &   0 & 0 & 0 & 0\\
1/2 & 1/2 & 0 & 0 & 0\\
1/2 & 0 & 1/2 & 0 & 0\\
1 &   0 & 0 & 1 & 0\\
\hline
\quad & 1/6 & 1/3 & 1/3 & 1/6\\
\end{array}
\end{align*}

\item Implicit two stage RK method (IRK2) % RKRad1p3
\begin{align*}
\begin{array}{c|c c}
0 & 0 & 0\\
2/3 & 1/3 & 1/3\\
\hline
\quad & 1/4 & 3/4\\
\end{array}
\end{align*}

\item Implicit three stage RK method (IRK3) % RKBut33
\begin{align*}
\begin{array}{c|c c c }
0.158984 & 0.158984 & 0 & 0 \\
0.579492 & 0.420508 & 0.158984 & 0 \\
1 & 0.348023 & 0.492993 & 0.158984 \\
\hline
\quad & 0.348022 & 0.492994 & 0.158984 \\
\end{array}
\end{align*}
\end{enumerate}

\section*{Acknowledgements}
Funding: This work was supported by the National Science Foundation [grant numbers 1762825, 1439145].

\bibliographystyle{unsrt}
\bibliography{implicitFD}

\begin{thebibliography}{10}

\bibitem{kumari2018unified}
Komal Kumari, Raktim Bhattacharya, and Diego~A. Donzis.
\newblock A unified approach for deriving optimal finite differences.
\newblock {\em Journal of Computational Physics}, 399, 2019.

\bibitem{lele1992compact}
Sanjiva~K Lele.
\newblock Compact finite difference schemes with spectral-like resolution.
\newblock {\em Journal of Computational Physics}, 103(1):16--42, 1992.

\bibitem{kim1996opt}
Jae~Wook Kim and Duck~Too Lee.
\newblock Optimized compact finite difference schemes with maximum resolution.
\newblock {\em AIAA Journal}, 34(5):887--893, 1996.

\bibitem{kim2007bcopt}
Jae~Wook Kim.
\newblock Optimised boundary compact finite difference schemes for
  computational aeroacoustics.
\newblock {\em Journal of Computational Physics}, 225(1):995--1019, 2007.

\bibitem{zang2003prefactor}
Graham Ashcroft and Xin Zhang.
\newblock Optimized prefactored compact schemes.
\newblock {\em Journal of Computational Physics}, 190(2):459--477, 2003.

\bibitem{zang2011prefactor}
Hongbo Zhou and Guanquan Zhang.
\newblock Prefactored optimized compact finite-difference schemes for second
  spatial derivatives.
\newblock {\em GEOPHYSICS}, 76(5):WB87--WB95, 2011.

\bibitem{hamr2008opt}
Zhanxin Liu, Qibai Huang, Zhigao Zhao, and Jixuan Yuan.
\newblock Optimized compact finite difference schemes with high accuracy and
  maximum resolution.
\newblock {\em International Journal of Aeroacoustics}, 7:123--146, 2008.

\bibitem{rona2017prefactor}
A.~Rona, I.~Spisso, E.~Hall, M.~Bernardini, and S.~Pirozzoli.
\newblock Optimised prefactored compact schemes for linear wave propagation
  phenomena.
\newblock {\em Journal of Computational Physics}, 328:66--85, 2017.

\bibitem{webb1993drp}
C.~K. Tam and J.~C. Webb.
\newblock Dispersion-relation-preserving finite difference schemes for
  computational acoustics.
\newblock {\em Journal of Computational Physics}, 107:262--281, 1993.

\bibitem{zhuang2002upwind}
M.~Zhuang and R.~F. Chen.
\newblock Applications of high-order optimized upwind schemes for computational
  aeroacoustics.
\newblock {\em AIAA Journal}, 40(3):443--449, 2002.

\bibitem{bogey2004bailly}
C.~Bogey and C.~Bailly.
\newblock A family of low dispersive and low dissipative explicit schemes for
  flow and noise computations.
\newblock {\em Journal of Computational Physics}, 194:194--214, 2004.

\bibitem{zhang2013maxnorm}
Jin-Hai Zhang and Zhen-Xing Yao.
\newblock Optimized explicit finite-difference schemes for spatial derivatives
  using maximum norm.
\newblock {\em Journal of Computational Physics}, 250:511--526, 2013.

\bibitem{grant2008cvx}
Michael Grant, Stephen Boyd, and Yinyu Ye.
\newblock Cvx: Matlab software for disciplined convex programming, 2008.

\bibitem{butcher2008ode}
J.~C. Butcher.
\newblock {\em Numerical Methods for Ordinary Differential Equations}.
\newblock John Wiley \& Sons Ltd., 2008.

\bibitem{eberhard1950}
Hopf Eberhard.
\newblock The partial differential equation $u_t + uu_x = \mu u_{xx}$.
\newblock {\em Communications on Pure and Applied Mathematics}, 3(3):201--230,
  1950.

\bibitem{cole1951}
Julian~D. Cole.
\newblock On a quasi-linear parabolic equation occurring in aerodynamics.
\newblock {\em Quarterly of Applied Mathematics}, 9(3):225--236, 1951.

\end{thebibliography}

\end{document}